\documentclass[a4paper,10pt]{article}
\usepackage[T1]{fontenc}
\usepackage{graphicx,amsmath,amsfonts,amssymb,amsthm,framed,color,cancel
}
\usepackage{mathrsfs,eufrak}

\newcommand{\dint}{\mathrm{d}}

\newcommand{\tcond}{t_{\rm c}}
\newcommand{\texit}{t_{\rm e}}

\newtheorem{thm}{Theorem}[section]
\newtheorem{corollary}[thm]{Corollary}

\newtheorem{prop}[thm]{Proposition}

\numberwithin{equation}{section}

\title{Exact simulation of first exit times for one-dimensional diffusion processes}
\begin{document}
\author{S. Herrmann$^1$ and C. Zucca$^2$\\[5pt]
\small{$^1$Institut de Math{\'e}matiques de Bourgogne (IMB) - UMR 5584, CNRS,}\\
\small{Universit{\'e} de Bourgogne Franche-Comt\'e, F-21000 Dijon, France} \\
\small{Samuel.Herrmann@u-bourgogne.fr}\\[5pt]
\small {$^2$Department of Mathematics 'G. Peano', }\\
\small{University of Torino, Via Carlo Alberto 10,
10123 Turin, Italy,}\\ 
\small{cristina.zucca@unito.it}
}
\maketitle
\begin{abstract}
The simulation of exit times for diffusion processes is a challenging task since it concerns many applications in different fields like mathematical finance, neuroscience, reliability... The usual procedure is to use discretization schemes which unfortunately introduce some error in the target distribution. Our aim is to present a new algorithm which simulates exactly the exit time for one-dimensional diffusions. This acceptance-rejection algorithm requires to simulate exactly the exit time of the Brownian motion on one side and the Brownian position at a given time, constrained not to have exit before, on the other side. Crucial tools in this study are the Girsanov transformation, the convergent series method for the simulation of random variables and the classical rejection sampling. The efficiency of the method is described through theoretical results and numerical examples.
\end{abstract}
\par\medskip

\par\medskip
\section*{Introduction}
First exit time distributions for general stochastic processes are of prime importance in many contexts. In mathematical finance, they permit to appreciate the risk of default for a given path-dependent option; in neuroscience, they characterize for instance the interspike time distribution... Since diffusion processes (solutions of stochastic differential equations) form an important family of stochastic processes, we aim to describe quite precisely their exit times. Unfortunately a simple and explicit expression of their distribution is not generally available which leads us to consider numerical approximations. Our task is to point out an algorithm which permits to simulate the exit time of the diffusion. This objective was already concerned by several studies introducing a discretization scheme for the corresponding stochastic differential equation. Most of them are based on improvements of the classical Euler scheme (see for instance \cite{Broadie-Glasserman-Kou-1997}, \cite{Gobet-Menozzi-10}, \cite{Gobet-2000}) which essentially consist in reducing the error stem from the approximation procedure. Let us also note another point of view which consists in approximating the probability density function of the exit time and therefore to approximate the solution of an integral equation \cite{Sacerdote-2014}.

Our approach is completely different since we emphasize an exact simulation procedure: the distribution of the random outcome  of the algorithm is identical to the distribution of the first exit time of the diffusion process. For such simulation methods based on an acceptance-rejection sampling, the challenge is to describe and reduce if possible the time consumption of the simulation. Beskos \& Roberts, in their founding work \cite{beskos2005exact}, already introduced the exact simulation for diffusion paths on some fixed time interval. Meanwhile several modifications of this algorithm have been proposed \cite{Beskos-2006, Beskos-2008, Jenkins}. The basic idea of the rejection sampling is to sequentially observe independent random objects generated according to a proposal distribution until a condition is satisfied. That means that each object is accepted or rejected according to a certain probability weight. 

For diffusion processes on a fixed interval $[0,T]$, the proposed object pointed out by Beskos \& Roberts is the Brownian paths (easy to simulate) and the rejection weight can be computed using the Girsanov transformation which essentially requires the simulation of Brownian bridges. In a previous work \cite{Herrmann-Zucca}, the authors proposed a similar procedure in order to exactly simulate the first passage time of a one-dimensional diffusion through a given threshold: the proposal object is then the Brownian first passage time (inverse gaussian random variable) and the rejection weight requires the simulation of Brownian bridges conditioned to stay under a given threshold (Bessel processes). In order to adapt such a procedure to exit times, it suffices to generate Brownian exit times and to use the rejection weight suggested by the Girsanov transformation.  Unfortunately this random weight requires the simulation of a Brownian bridge conditioned to stay in a given interval, which corresponds to a SDE with singular coefficients: there is no way to exactly simulate such a stochastic paths. Such an intuitive generalization leads therefore to a deadlock.

That's why we aim to present a quite different rejection sampling based on a similar concept (Girsanov's transformation) and avoiding the simulation of the whole conditioned Brownian paths.
Let us consider the stochastic process $(X_t,\ t\ge 0)$, solution of the SDE:
\[
dX_t=\mu(X_t)dt+\sigma(X_t)dB_t,\quad X_0=x\in(a,b),
\]
where $(B_t,\ t\ge 0)$ stands for the standard one-dimensional Brownian motion. We denote by $\tau_{a,b}$ the first time the diffusion exits from the interval $[a,b]$:
\begin{equation}\label{eq:def:tau}
\tau_{a,b}(X):=\inf\{t> 0: \ X_t\notin [a,b]  \}.
\end{equation}
The aim of the study is to propose an efficient algorithm in order to simulate the first exit time $\tau_{a,b}(X)$. Under suitable conditions, the \emph{Lamperti transform} permits to reduce the area of investigation to the constant diffusion case: $\sigma(x)\equiv 1$ providing to change the interval $[a,b]$. That's why we shall only focus our attention on the diffusion process:
\begin{equation}
\label{eq:diff}
dX_t=\mu(X_t)dt+dB_t,\quad X_0=0,
\end{equation}
for $t\le \tau_{a,b}(X)$ with $a<0<b$. The exact algorithm presented here (Section \ref{sec:DFET}) is essentially based on the Girsanov transformation which permits to relate the diffusion \eqref{eq:diff} to a standard one-dimensional Brownian motion. Let us roughly describe the crucial tools needed to execute such an algorithm. The proposal random variable is the Brownian exit time of the interval $[a,b]$, denoted $\mathcal{T}_{\rm prop}$. The weight used in order to accept or reject this proposal depends both on the conditional distribution of the Brownian paths $(B_t)$ given that $\tau_{a,b}(B)=\mathcal{T}_{\rm prop}$ and on a Poisson process defined on the time axis $\mathbb{R}_+$ which is independent of the Brownian motion $(B_t)$ and whose events occur at time $T_1,\ldots,T_n,\ldots$ Then conditionally to $\tau_{a,b}(B)=\mathcal{T}_{\rm prop}$,  only the values of $B_{T_k}$, for any $T_k\le \mathcal{T}_{\rm prop}$, and $B_{\mathcal{T}_{\rm prop}}$ (dots in Figure \ref{fig:explanation}) are involved in the specific conditions associated to the acceptance of $\mathcal{T}_{\rm prop}$.

\begin{figure}[ht]
\centerline{\includegraphics[scale=0.8]{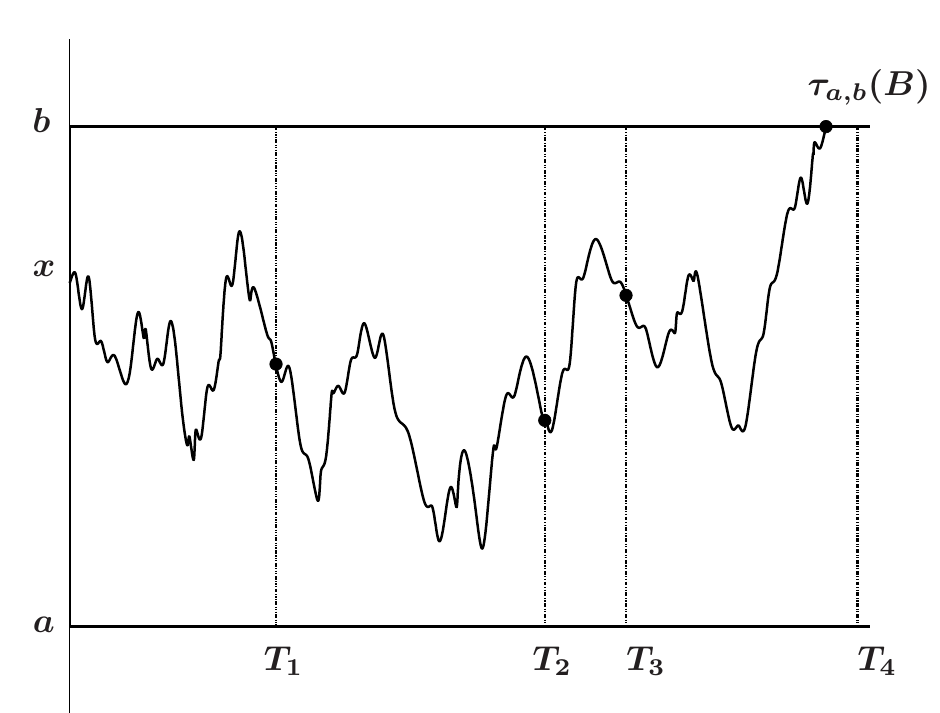}}
\caption{Decomposition of the Brownian paths}
\label{fig:explanation}
\end{figure}

The exact simulation of $\tau_{a,b}(X)$ therefore requires to describe:
\begin{enumerate}
\item The conditional distribution of $B_t$ for a given $t>0$, given that the first exit of the interval $\tau_{a,b}(B)$ is larger than $t$ (Section \ref{sec:BMconstrained}). Using the classical convergent series method, we propose an algorithm denoted by {\scriptsize CONDITIONAL\_DISTR\,(x,[a,b],t)} for the simulation of the constrained Brownian motion $B_t$ with initial condition $x$. 
\item The distribution of the Brownian exit time of the interval $[a,b]$ as the Brownian trajectory starts in $x$. The corresponding algorithm, presented in Section \ref{sec:BFET}, is denoted {\scriptsize BROWNIAN\_EXIT\_ASYMM\, (x,[a,b]) }.
\end{enumerate}
The two main algorithms (DET) and (GDET) presented in Section \ref{sec:DFET} use these two previous algorithms as basic components. They permit the exact simulation of diffusion exit times under weak conditions on the coefficient $\mu$. Efficiency results are pointed out (Theorem \ref{thm:efficiency-exit} and Proposition \ref{prop:efficiency-exit-gen}) and numerical illustrations complete the study  in Section \ref{subsec:DFET-num} 
%
%
%
%
%
%
%
\mathversion{bold}
\section{Brownian motion constrained to stay in an interval $[a,b]$}
\label{sec:BMconstrained}
\mathversion{normal}
Let us first consider the Brownian motion in the space interval $[a,b]$. We need to describe the distribution of $B_t$ for a given $t>0$, given that the first exit of the interval $\tau_{a,b}(B)$ is larger than $t$. We fix the starting position of the Brownian motion $B_0=y\in(a,b)$. Due to both the translation invariance and the scaling property of the Brownian motion, we get 
\begin{align*}
\mathbb{E}[F(B_t,\,t\le \tau_{a,b}(B))| B_0=y]=\mathbb{E}\Big[F\Big(\frac{b-a}{2}\ B_t+\frac{a+b}{2},\,t\le \tau_{-1,1}(B)\Big)\Big| B_0=x\Big],
\end{align*}
where $x=\frac{2y-a-b}{b-a}$. So if we denote $B^x$ the Brownian motion starting in $x$, we deduce easily the following distribution identity
\begin{equation}\label{eq:scaling}
\Big(\tau_{a,b}(B^{y}),B_{\tau_{a,b}(B^{y})}^{y}\Big)\stackrel{(d)}{=}\Big(\frac{(b-a)^2}{4}\,\tau_{-1,1}(B^x),  \frac{b-a}{2} B^x_{\tau_{-1,1}(B^x)}+\frac{a+b}{2}\Big).
\end{equation}
In other words, it suffices to study precisely the first Brownian exit problem from the normalized interval $[-1,1]$ with the initial condition $B_0=x\in[-1,1]$. For notational convenience, we denote by 
\begin{equation}
\label{def:tau-norm}
\tau^x=\inf\{t> 0:\ B_t^x\notin I\}\quad \mbox{with}\ I=[-1,1].
\end{equation}
Let us first emphasize a classical result on the exit position: since $(B_t,\ t\ge 0)$ is a martingale, the identity function $\mathfrak{s}(x)=x$ corresponds to the scale function of the Brownian motion. The optimal stopping theorem permits therefore to describe the probability to exit from the interval on one particular side:
\begin{equation}
\label{eq:exitloc}
\mathbb{P}(B^x_\tau=1)=\frac{\mathfrak{s}(-1)-\mathfrak{s}(x)}{\mathfrak{s}(-1)-\mathfrak{s}(1)}=\frac{x+1}{2},\quad \forall x\in[-1,1]. 
\end{equation}
%
%
%
%
%
%
%
\mathversion{bold}
\subsection{Distribution of $B_t$ given $\tau>t$}
\label{subsec:BMconstrained}
\mathversion{normal}
Let us now focus our attention to the distribution of the constrained Brownian motion. We introduce the killed Brownian motion: as soon as the process hits the boundary of the interval, it jumps to a cemetery point $x^*\notin[-1,1]$. Indeed let us denote by $p(t,x,dy):=\mathbb{P}(B^x_t\in dy,\tau>t)$ the transition probability of the Brownian motion joint with $\tau>t$. It can be written as
$p(t,x,y)=q(t,x,y)\mathbb{P}(\tau^x>t)$ where $q(t,x,y)$ is the transition probability of the Brownian motion conditioned to stay in the interval $[-1,1]$ till time $t$. For any non negative or bounded measurable function $\psi$, we get
\[
\mathbb{E}[\psi(B^x_t),\tau^x>t]=\int_{-1}^1\psi(y)p(t,x,y)\dint y,\quad x\in[-1,1].
\]
It is well known (see, for instance \cite{Ito-McKean} Section 4.11 or \cite{Cox-Miller} Section 5.7) that $(t,y)\mapsto p(t,x,y)$ satisfies an Initial-Boundary value problem associated to the heat equation:
\begin{equation}
\label{eq:IBVP}
\left\{
\displaystyle \begin{array}{l}\frac{\partial p}{\partial t}(t,x,y)=\frac{1}{2}\, \frac{\partial^2 p}{\partial y^2}(t,x,y)\\[5pt]
p(t,x,y)\to\delta_{x}(y)\quad\mbox{as}\ t\to 0,\\[5pt]
p(t,x,y)=0\quad\mbox{for}\ y\in\{-1,1\}, \ t>0.
\end{array}\right.
\end{equation}
Milstein and Tretyakov \cite{Milstein-Tretyakov} recall that the solution of \eqref{eq:IBVP} takes two different expressions. The first one is obtained by the method of images and the second one is based on a spectral decomposition of the heat equation. We introduce the standard gaussian \emph{pdf} and \emph{cdf}: $\phi(x)=\frac{1}{\sqrt{2\pi}}\, e^{-x^2/2}$ and $\Phi(x)=\int_{-\infty}^x \phi(y)\dint y$. Then
\begin{equation}
\label{eq:first kind}
p(t,x,y)=\frac{1}{\sqrt{t}}\sum_{n=-\infty}^{+\infty}\Big\{ \phi((x-y-4n)/\sqrt{t})-\phi((x+y-2-4n)/\sqrt{t}) \Big\}
\end{equation}
which is a convenient formula for small values of $t$. For large times $t$,  we would prefer the second formula:
\begin{equation}
\label{eq:second kind}
p(t,x,y)=\sum_{n\ge 1}\exp\Big(-\frac{n^2\pi^2t}{8}\Big)\sin\Big(\frac{n\pi}{2}\ (x+1)\Big)\sin\Big(\frac{n\pi}{2}\ (y+1)\Big).
\end{equation}
Let us fix $(t,x)$: since we get an explicit expression of $p(t,x,y)$, we can point out exact simulation algorithms based on the acceptance/rejection method even if $p(t,x,y)$ is not a probability density function. Indeed it suffices to divide by $\mathbb{P}(\tau^x>t)$ in order to obtain such a density (cf. Section \ref{subsec:BFETdistribution}). The method applied here is a particular acceptance/rejection method presented in \cite{Devroye-1986} as the \emph{convergent series method}.
%
%
%
%
%
%
%
\subsection{Inequalities related to the series expansion}
\label{subsec:BMconstrained-series}
Let us consider $f$ a non negative function satisfying a series expansion:
\[
f(y)=\sum_{n\ge 0}f_n(y).
\]  
We assume that $I(f):=\int_\mathbb{R}f(y)\,dy<\infty$. The \emph{convergent series method} permits to simulate a random variable with the corresponding probability distribution function $f/I(f)$. It requires two important features: 
\begin{itemize}
\item the existence of both a \emph{pdf} $h$ and a constant $\kappa>0$ such that
\begin{equation}
\label{eq:upper-h}
f(y)\le \kappa h(y),\quad y\in\mathbb{R}.
\end{equation}
\item the existence of a positive sequence $(r_n)_{n\ge 0}$ which converges toward $0$ and satisfies the following reminder upper-bound 
\begin{equation}
\label{eq:reminder}
|R_n(y)|\le r_n,\quad \forall y\in\mathbb{R}\quad\mbox{where}\quad R_n(y):=\sum_{k\ge n+1}f_k(y).
\end{equation}
\end{itemize}

\begin{framed}
\centerline{CONVERGENT SERIES METHOD}

\vspace*{0.5cm}

\noindent
{\bf  Initialization.} \emph{$\mathcal{N}_c=0$.} \\[5pt]
{\bf Step 1.} Generate a random variable $Y$ with density $h$.\\[5pt]
{\bf Step 2.} Generate a random variable $U$ uniformly distributed on $[0,1]$ and define $W=\kappa U h(Y)$.\\[5pt]
{\bf Step 3. Initialization.} \emph{$\mathcal{N}_l=0$, $S=0$, ${\rm Test}=0$.} \\[5pt]
{\bf Step 4.} While $ ({\rm Test}=0)$ do:
\begin{itemize}
\item  $S\leftarrow S+f_{\mathcal{N}_l}(Y) $
\item ${\rm Test}=1_{\{ |S-W|>r_{\mathcal{N}_l} \}}$
\item $\mathcal{N}_l\leftarrow \mathcal{N}_l+1$ and $\mathcal{N}_c\leftarrow \mathcal{N}_c+1$.
\end{itemize}
{\bf Step 5.} If $W\le S$ then $X=Y$ otherwise go to Step 1.\\[5pt]
{\bf Outcome:} the random variable $X$ with density $f/I(f)$ and the global number of terms of the series expansions $\mathcal{N}_c$ used.
\end{framed}

In order to simulate $q(t,x,y)$, i.e. the conditional distribution of $B_t$ given $t<\tau$, we need to choose the density $h$, to point out some constant $\kappa$ and the sequence $(r_n)_{n\ge 0}$ for each explicit expression of $f$ given by \eqref{eq:first kind} and \eqref{eq:second kind}. 
Let us just note that it is challenging to find the smallest constant $\kappa$ as possible since  the number of iterations of Step 1 in order to get an outcome is geometrically distributed with average $\kappa/I(f)$. 

\subsubsection*{Some comments on the series expansion (\ref{eq:first kind})}
We note that (\ref{eq:first kind}) is not an alternating series and
 can be rewritten as
\begin{align*}
p(t,x,y)&=a_0(t,x,y)+\sum_{n=1}^\infty (a_n(t,x,y)+a_{-n}(t,x,y)),
\end{align*}
where $a_n(t,x,y)=\frac{1}{\sqrt{t}}(\phi((x-y-4n)/\sqrt{t})-\phi((x+y-2-4n)/\sqrt{t}))$. 
We first observe that 
\begin{equation}\label{eq:incr_pos}
a_{-n}(t,x,y)<0\quad\mbox{and}\quad  a_{n}(t,x,y)>0,\quad \forall n\ge 1, \ \forall (x,y)\in[-1,1]^2.
\end{equation} 
Moreover both increments of the function $\phi$ are computed on intervals whose length does not depend on the variables $x$ and $n$:
\[
(x+y-2-4n)-(x-y-4n)=(x+y-2+4n)-(x-y+4n)
=2y-2.
\]
For $n$ large enough, the increments are decreasing and therefore $a_n(t,x,y)+a_{-n}(t,x,y)$ becomes negative.  Indeed by considering the four terms:
\[
(x+y-2-4n)<(x-y-4n)<0<(x+y-2+4n)<(x-y+4n),
\]
it is straightforward to see that the smallest one in absolute value is $|x+y-2+4n|$. Since the change of convexity of the curve $x\mapsto\phi(x)$ appears for $x=1$, the sum $a_n(t,x,y)+a_{-n}(t,x,y)$ is negative as soon as 
\(
x+y-2+4n\ge\sqrt{t},
\)
which is satisfied in particular if $n\ge n_0$ with
\begin{equation}
\label{eq:cond_n0}
n_0=\left\lfloor \sqrt{t}/4\right\rfloor+2.
\end{equation}
Since the terms of the series are negative for $n\ge n_0$, the series (\ref{eq:first kind}) is obviously not an alternating series.
\mathversion{bold}
\subsubsection*{Bound of the series reminder in (\ref{eq:first kind})}
\mathversion{normal}
Let us fix $(t,x)\in\mathbb{R}_+^*\times [-1,1]$. We introduce the reminder of the series  \eqref{eq:first kind}:
\[
R_n(y):=\sum_{k\ge n+1}(a_k(t,x,y)+a_{-k}(t,x,y)),\quad n\ge 0.
\]
In order to apply the convergent series method, we need to bound this reminder.
Since $\phi$ is a decreasing function, we obtain the following bound:
\[
\frac{1}{\sqrt{t}}\sum_{k\ge n+1}\phi((\alpha+4k)/\sqrt{t})\le \frac{1}{\sqrt{t}}\int_{n}^\infty\phi((\alpha+4z)/\sqrt{t})\dint z,
\]
as soon as $\alpha+4n\ge 0.$ Moreover, by symmetry,
\[
\frac{1}{\sqrt{t}}\sum_{k\ge n+1}\phi((\alpha-4k)/\sqrt{t})\le \frac{1}{\sqrt{t}}\int_{n}^\infty\phi((\alpha-4z)/\sqrt{t})\dint z,
\]
as soon as $\alpha-4n\le 0.$ If $\alpha=x-y$ or $\alpha=x+y-2$, then the inequalities are satisfied for any $n\ge 1$. 
Finally, for $n\ge 1$, we get
\begin{align*}
|R_n(y)|&\le \frac{1}{\sqrt{t}}  \sum_{k\ge n+1} \phi((x-y-4k)/\sqrt{t})+\phi((x+y-2+4k)/\sqrt{t})\\
&\le \frac{1}{\sqrt{t}}\int_{n}^\infty\phi((x-y-4z)/\sqrt{t})\dint z+ \frac{1}{\sqrt{t}}\int_{n}^\infty\phi((x+y-2+4z)/\sqrt{t})\dint z.
\end{align*}
Let us now observe that we can obtain a bound which is uniform with respect to both the variable $x$ and $y$.
\begin{align*}
|R_n(y)|&\le \frac{1}{8}\Big(1-{\rm erf}\Big(\frac{4n-2}{\sqrt{2t}} \Big)\Big)+\frac{1}{8}\Big(1-{\rm erf}\Big(\frac{4n+4}{\sqrt{2t}} \Big)\Big)\\
&\le r_n(t):=\frac{1}{4}\Big(1-{\rm erf}\Big(\frac{4n-2}{\sqrt{2t}} \Big)\Big).
\end{align*}
Let us note that this upper-bound is efficient for small values of $t$.
\mathversion{bold}
\subsubsection*{Proposal distribution for the rejection method using (\ref{eq:first kind})}
\mathversion{normal}
The aim is to use the acceptance/rejection algorithm in order to simulate a random variable with the target density $q(t,x,y)=p(t,x,y)/\mathbb{P}(\tau^x>t)$ where $p(t,x,y)$ is given by \eqref{eq:first kind}. That's why we are looking for a proposal \emph{pdf} $h$ and a constant $\kappa(t,x)$ independent of $y$ satisfying
\[
p(t,x,y)\le \kappa(t,x)h_{t,x}(y),\quad \forall y\in[-1,1],
\]
as explained in \eqref{eq:upper-h}.
We suggest here the following choice of the proposal distribution:
\[
h_{t,x}(y)=\frac{1}{\sqrt{t}}\ \phi((x-y)/\sqrt{t}),
\]
that means that we choose a gaussian distribution centered in $x$ with variance $t$: it corresponds to the distribution of $B_t$ without any conditioning. The rejection method shall permit to go from this initial distribution to the conditional distribution with respect to the event $\{\tau>t\}$. The algorithm is therefore particularly efficient if the considered event is satisfied with a large probability, this is namely the case for small values of $t$.

Let us now focus our attention to the upper-bound
\begin{align*}
\frac{p(t,x,y)}{h_{t,x}(y)}\le 1+\sum_{n\ge 1}\frac{a_{-n}(t,x,y)}{h_{t,x}(y)}+\frac{a_n(t,x,y)}{h_{t,x}(y)}.
\end{align*}
Since the sum $a_{-n}(t,x,y)+a_{n}(t,x,y)$ is negative as soon as $n\ge n_0$ (see the definition of $n_0$ in \eqref{eq:cond_n0}), we obtain
\begin{align*}
\frac{p(t,x,y)}{h_{t,x}(y)}&\le 1+\sum_{n= 1}^{n_0}\frac{a_{-n}(t,x,y)}{h_{t,x}(y)}+\frac{a_n(t,x,y)}{h_{t,x}(y)}.
\end{align*}
Moreover $a_n(t,x,y)$ is negative, and therefore
\begin{align}
\label{eq:bound-1}
\frac{p(t,x,y)}{h_{t,x}(y)}\le 1+\sum_{n= 1}^{n_0}\frac{\phi((x-y-4n)/\sqrt{t})}{\phi((x-y)/\sqrt{t})}.
\end{align}
In fact all the ratios considered just above have the same nice property: they are smaller than $1$. It suffices to notice that $(x-y)\in]-2,2[$ implies $\phi((x-y)/\sqrt{t})\ge \phi(\alpha/\sqrt{t})$ for any $\alpha\notin]-2,2[$. Hence, for $\alpha=(x-y-4n)$ with $n\ge 1$, the ratios in the r.h.s. of (\ref{eq:bound-1}) are smaller than 1. Using the definition of $n_0$ in \eqref{eq:cond_n0}, \eqref{eq:bound-1} becomes
\begin{equation}\label{eq:k smallt}
\frac{p(t,x,y)}{h_{t,x}(y)}\le 1+n_0=3+\left\lfloor \sqrt{t}/4\right\rfloor=:\kappa(t,x).
\end{equation}
Let us note that this bound is uniform with respect to the variable $x$ (starting value of the Brownian paths).
\subsubsection*{Some comments on the series expansion (\ref{eq:second kind})}
Since $q(t,x,y)$ 
can be characterized by two different series expansions, it is useful to understand which kind of acceptance/rejection algorithm each series produces. Let us now focus on the second series \eqref{eq:second kind}: 
\[
p(t,x,y)=\sum_{n\ge 1}\beta_n(t)a_n(x,y),\quad t>0,\quad (x,y)\in[-1,1]^2,
\]
where $\beta_n(t)=e^{-n^2\pi^2t/8}$ and $a_n(x,y)=
 \sin\Big(\frac{n\pi}{2}\ (x+1)\Big)\sin\Big(\frac{n\pi}{2}\ (y+1)\Big)$. Let us first notice that, for fixed $t>0$, $(\beta_n(t))_{n\ge 1}$ is a decreasing sequence of positive real numbers which converges towards $0$. Since $a_n$ is defined as the product of two sine functions, one idea to bound the reminder of the series is to use the Abel transform and to prove that $\sum_{n=1}^N a_n(x,y)$ is bounded as $N$ tends to infinity. It is quite easy to obtain some bound when $x\neq y$ but for the particular case $x=y$, the sequence of partial sums is increasing (sum of squares) and tends to infinity. In conclusion, we cannot apply technics based on the Abel transform for the series \eqref{eq:second kind}. We shall therefore present an other approach.

\subsubsection*{Bound of the series reminder in (\ref{eq:second kind})}
Let us define the reminder of the series \eqref{eq:second kind}:
\[
\hat{R}_n(y):=\sum_{k\ge n+ 1}\beta_k(t)a_k(x,y),\quad t>0,\quad (x,y)\in[-1,1]^2.
\]
Let us find a bound of this reminder. Since $|a_k(x,y)|\le 1$ for any $k\in\mathbb{N}$ and any $(x,y)\in]-1,1[^2$, we obtain
\[
|\hat{R}_n|\le \sum_{k\ge n+1}\beta_k(t),\quad n\ge 0.
\]
Moreover, $u\mapsto e^{-u^2\pi^2t/8}$ is a decreasing function on $\mathbb{R}_+$, and therefore we get the following bound 
\begin{equation}
\label{eq:reminder-second-kind}
|\hat{R}_n|\le \int_{n}^\infty e^{-u^2\pi^2t/8}\dint u=\sqrt{\frac{2}{\pi t}}\left(1-{\rm erf}\left(\frac{n\pi \sqrt{t}}{2\sqrt{2}}\right)\right):=\hat{r}_n(t),\quad n\ge 0.
\end{equation}
Note that this bound is sharp for large value of $t$.
\subsubsection*{Proposal distribution for the rejection method using (\ref{eq:second kind})}
The aim is to use an acceptance /rejection method in order to simulate a random variable with the target density $p(t,x,y)/\mathbb{P}(\tau>t)$ based on the equation \eqref{eq:second kind}, as already done for \eqref{eq:first kind}. We are looking for a probability distribution function $ h(y)$ and a constant $\kappa(t,x)>0$ independent of $y$ such that:
\[
p(t,x,y)\le \kappa(t,x)h(y),\quad \forall y\in[-1,1],
\]
as explained in \eqref{eq:upper-h}. Our particular choice for the function $h$ is
\[
h(y)=\frac{\pi}{4}\sin\Big( \frac{\pi}{2}(y+1) \Big) ,\quad y\in[-1,1],
\]
which corresponds to the invariant probability measure of the Brownian motion conditioned to stay in the interval $[-1,1]$.  Let us now determine a constant $\kappa(t,x)$. We have
\begin{align*}
p(t,x,y)=\frac{4}{\pi} h(y)\sin\Big(\frac{\pi}{2}\,(x+1)\Big)\sum_{n\ge 1}\beta_n(t)s_n(x)s_n(y),
\end{align*}
where
\begin{equation}\label{eq:def:sn}
s_n(x)=\frac{\sin(\frac{n\pi}{2}(x+1))}{\sin(\frac{\pi}{2}\,(x+1))}.
\end{equation}
Using the formula $a^n-b^n=(a-b)(a^{n-1}+a^{n-2}b+\ldots+b^{n-1})$ applied to the sine function, we observe that $|s_n(x)|\le n$ for all $x\in[-1,1]$.
Hence
\[
p(t,x,y)\le \frac{4}{\pi} h(y) \,\sin\left(\frac{\pi}{2}\,(x+1)\right)\sum_{n\ge 1}n^2\beta_n(t),\quad t>0.
\]
Moreover the function $x\mapsto x^2\, e^{-x^2 \pi^2t/8}$ is decreasing as soon as 
$x\ge \frac{2\sqrt{2}}{\pi\sqrt{t}}$,  we obtain the following upper bound:
\begin{align}
\label{eq:rest}\sum_{n\ge n_0+1}  n^2\beta_n(t)&\le \int_{n_0}^\infty z^2 e^{-z^2\pi^2t/8}\dint z,
\end{align}
where $ n_0=\left\lfloor   \frac{2\sqrt{2}}{\pi\sqrt{t}} \right\rfloor+1$.

\noindent The integration by parts formula and a change of variables lead to
\begin{align*}
\int_{n_0}^\infty z^2 e^{-z^2\pi^2t/8}\dint z &=\frac{4n_0}{\pi^2t}\,e^{-n_0^2\pi^2t/8}+\frac{4}{\pi^2t}\int_{n_0}^\infty e^{-z^2\pi^2t/8}\,\dint z\\
&=\frac{4}{\pi^2t}\,e^{-n_0^2\pi^2t/8}\left(n_0+\int_{0}^\infty e^{-u^2\pi^2t/8-n_0u\pi^2t/4}\,\dint u\right)\\
&\le \frac{4}{\pi^2t}\,e^{-n_0^2\pi^2t/8}\Big(n_0+\sqrt{\frac{2}{\pi t}}\Big)\le C_1(t):=\frac{8n_0}{\pi^2t}\,e^{-n_0^2\pi^2t/8},
\end{align*}
since $\sqrt{\frac{2}{\pi t}}\le \sqrt{\pi}n_0/2\le n_0$.  We just note that 
\[
C_1(t)=\frac{8}{\pi^2 t} \left(\left\lfloor   \frac{2\sqrt{2}}{\pi\sqrt{t}} \right\rfloor+1\right)\, e^{-n_0^2\pi^2t/8}\sim \frac{8}{\pi^2t}\,e^{-\pi^2 t/8},
\]
as $t$ becomes large. Moreover
\begin{equation}\label{eq:rest1}
\sum_{n=1}^{n_0}n^2\beta_n(t)\le  C_2(t):=n_0^3\,\beta_1(t)=\left(\left\lfloor \frac{2\sqrt{2}}{\pi\sqrt{t}}\right\rfloor+1\right)^3e^{-\pi^2 t/8}.
\end{equation}
%
Finally we can choose the constant 
\(
\kappa(t,x):=\frac{4}{\pi}(C_1(t)+C_2(t))\ \sin(\frac{\pi}{2}\,(x+1)).
\)
Let us note that $\kappa(t,x)\sim \frac{4}{\pi}\sin(\frac{\pi}{2}\,(x+1)) e^{-\pi^2t/8}$ as $t$ tends to infinity. Since the averaged number of iterations in the acceptance/rejection method corresponds to $\kappa(t,x)/\mathbb{P}_x(\tau>t)$ and since both $\kappa(t,x)$ and $\mathbb{P}_x(\tau>t)$ are of the same order in the large time limit, the efficiency of the algorithm using \eqref{eq:second kind} still remains strong when the time variable enlarges. 
Indeed let us decompose the following probability:
\begin{align*}
\mathbb{P}_x(\tau> t)&=\sum_{n\ge 0}\mathcal{R}_n(t,x)
\end{align*}
where
\begin{align*}
\mathcal{R}_n(t,x):=\frac{4}{\pi}\frac{1}{2n+1}\exp\Big(-\frac{(2n+1)^2\pi^2}{8}\ t\Big)\sin\Big((2n+1)(x+1)\frac{\pi}{2}\Big).
\end{align*}
Using the definition \eqref{eq:def:sn}, we obtain
\begin{align*}
\frac{\mathbb{P}_x(\tau> t)}{\mathcal{R}_0(t,x)}&=1+\sum_{n\ge 1}\frac{s_{2n+1}(x)}{2n+1}\ \exp\Big(-\frac{n(n+1)}{2}\,\pi^2t\Big).
\end{align*}
The inequality $|s_n(x)|\le n$ for any $x\in[-1,1]$ and $n\ge 1$ leads to
\begin{align*}
\frac{\mathbb{P}_x(\tau> t)}{\mathcal{R}_0(t,x)}&\ge 1-\sum_{n\ge 1}\exp\Big(-\frac{n(n+1)}{2}\,\pi^2t\Big).
\end{align*}
We observe that each term of the series converges in a monotonous way towards $0$ as $t$ tends to $\infty$ which implies that the ratio tends towards $1$ by the Lebesgue theorem. We therefore deduce that for any $\theta>1$, there exists  $t_\theta>0$ such that the average number of iterations in the acceptance/rejection algorithm is smaller than $\theta$ as soon as $t\ge t_\theta$. 
%
%
%
%
%
%
%
\subsection{Algorithm and numerics}
\label{subsec:BMconstrained-num}
Let us now describe the algorithms used in the following in order to simulate $q(t,x,dy)$, the conditional distribution of the Brownian motion at time $t$ given $\tau>t$. We know that $p(t,x,dy)$ admits two different series expansion presented in \eqref{eq:first kind} and \eqref{eq:second kind}. As introduced in Section \ref{subsec:BMconstrained-series}, the classical convergent series method requires both a proposal distribution $h$, satisfying the inequality \eqref{eq:upper-h} that is $p(t,x,y)\le \kappa(t,x)h_{t,x}(y)$ (also necessary for classical acceptance/rejection methods), and a precise description of the series reminder characterized by the sequence $(r_n(t))_n$ and $(\hat{r}_n(t))_n$, see \eqref{eq:reminder}. 

We just recall the results obtained in the previous section:
\begin{table}[ht]
\caption{Series expansion \eqref{eq:first kind}}
\vspace*{0.2cm}
\renewcommand{\arraystretch}{2}
\centerline{\begin{tabular}{|l|p{6.5cm}|}
\hline
Proposal distribution: $h_{t,x}(y)$ & \centerline{$\frac{1}{\sqrt{t}}\ \phi((x-y)/\sqrt{t})$} \\
Constant: $\kappa(t,x)$ &  \centerline{$3+\left\lfloor \sqrt{t}/4\right\rfloor$} \\
Reminder bounds: $r_n(t)$ &  \centerline{$\frac{1}{4}\Big(1-{\rm erf}\Big(\frac{4n-2}{\sqrt{2t}} \Big)\Big)$}  \\
\hline
\end{tabular}}
\label{table1}
\end{table}
\begin{table}[ht]
\caption{Series expansion \eqref{eq:second kind}}
\vspace*{0.2cm}
\centerline{\begin{tabular}{|l|p{9.5cm}|}
\hline
$h_{t,x}(y)$ & \centerline{$\frac{\pi}{4}\sin\Big( \frac{\pi}{2}(y+1) \Big)$} \\
$\kappa(t,x)$ & \hspace*{\fill}$ \frac{4}{\pi}\sin\Big( \frac{\pi}{2}(x+1)\Big) \left\{ \frac{8n_0}{\pi^2t}\,e^{-n_0^2\pi^2t/8} +n_0^3\, e^{-\pi^2 t/8} \right\}$\hspace*{\fill}\\
 & \hspace*{\fill}where $n_0= \left\lfloor \frac{2\sqrt{2}}{\pi\sqrt{t}} \right\rfloor+1$\hspace*{\fill} \\
$\hat{r}_n(t)$ &  \centerline{$\sqrt{\frac{2}{\pi t}}\Big(1-{\rm erf}(\frac{n\pi \sqrt{t}}{2\sqrt{2}})\Big)$}  \\
\hline
\end{tabular}}
\renewcommand{\arraystretch}{1}
\label{table2}
\end{table}
\vspace*{0.2cm}

It is straightforward that the convergent series method using \eqref{eq:first kind} is convenient for small values of $t$ while \eqref{eq:second kind} is rather convenient for large $t$. That's why we choose a threshold $\tcond>0$ (threshold for the conditional distribution) such that \eqref{eq:first kind} is used for $t\le \tcond$ and \eqref{eq:second kind} otherwise. For practical purposes, we fix $\tcond=0.7$, this choice is motivated by the curves in Fig. \ref{fig:curves} and will be held for all numerical illustrations presented in this study. Applying the convergence series algorithm either for small times or large times permits to obtain the simulations presented in Fig. \ref{fig:cond2}. We observe that, even if $x\neq 0$, the condition $t<\tau$ leads to a distribution which looks like symmetric as $t$ becomes large and which converges towards $h_{t,x}(y)$ the invariant measure of the diffusion conditioned to stay in the interval $[-1,1]$. 
\begin{figure}[h]
\centerline{\includegraphics[height=8cm]{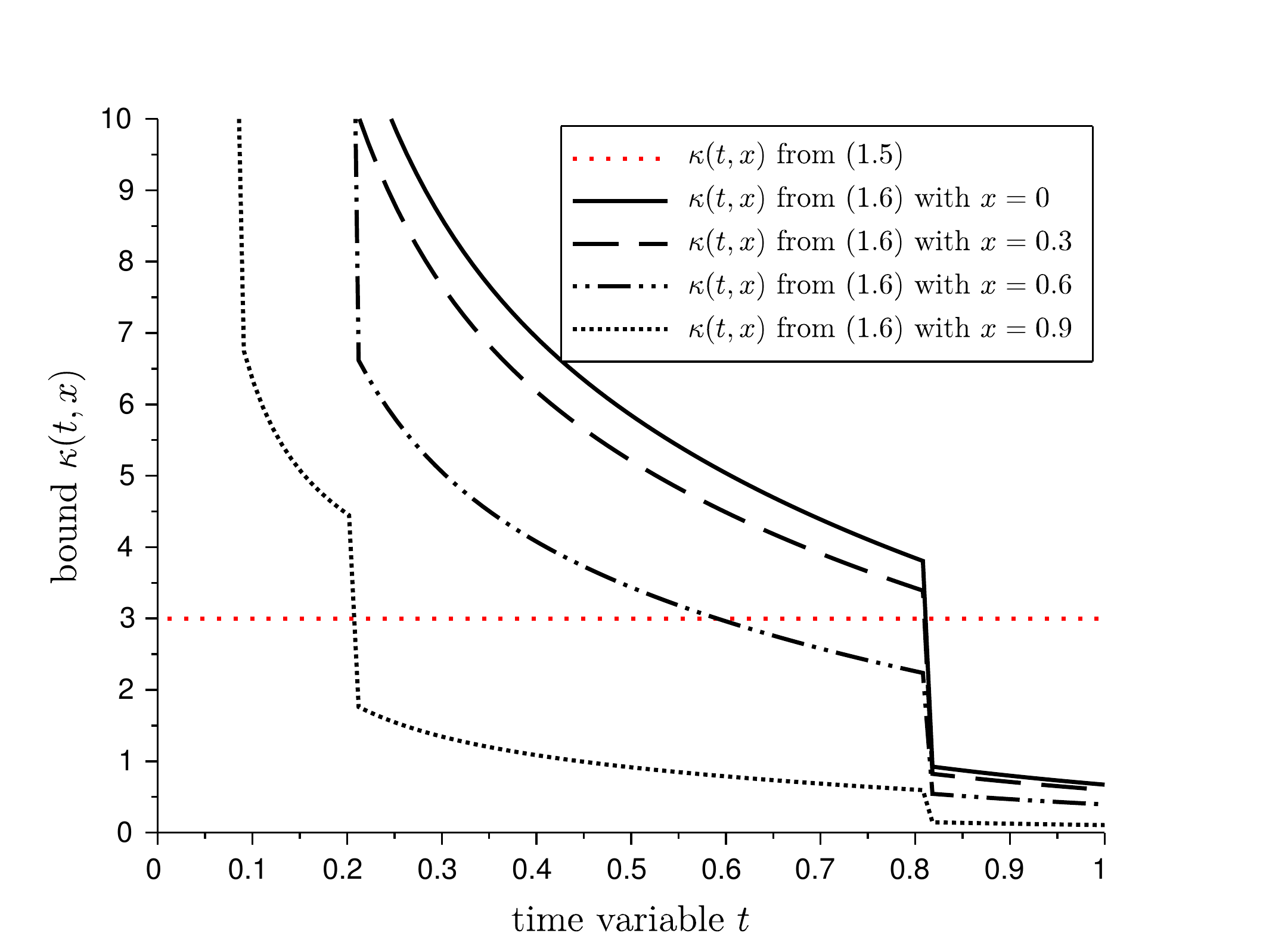}}
\caption{\small Comparison between the constants $\kappa(t,x)$ obtained either from the series expansion \eqref{eq:first kind} or \eqref{eq:second kind} as time elapses ad for different initial values $x\in[0,1)$. }
\label{fig:curves}
\end{figure}
\begin{figure}[h]
\centerline{\includegraphics[height=4cm]{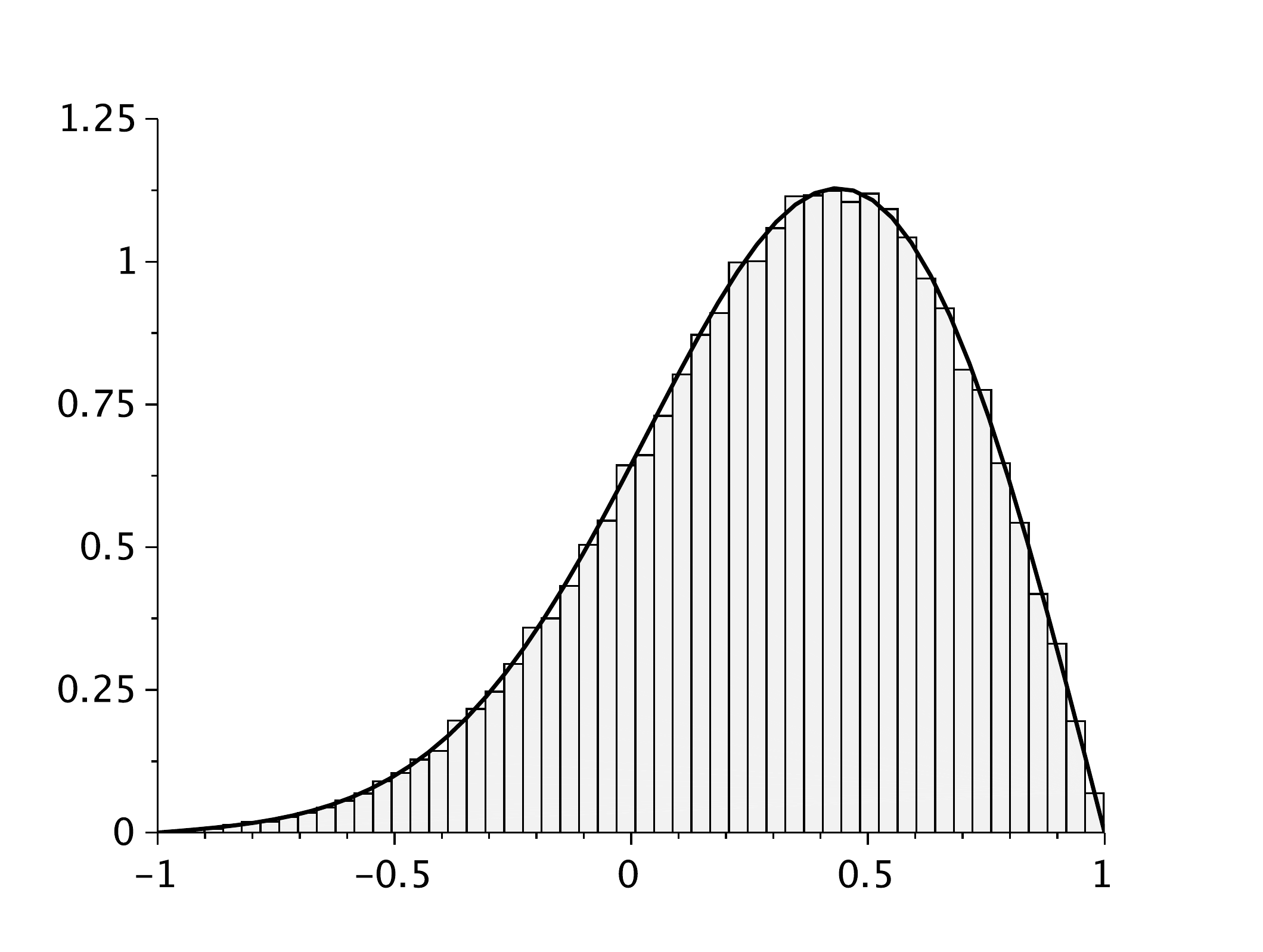}\includegraphics[height=4cm]{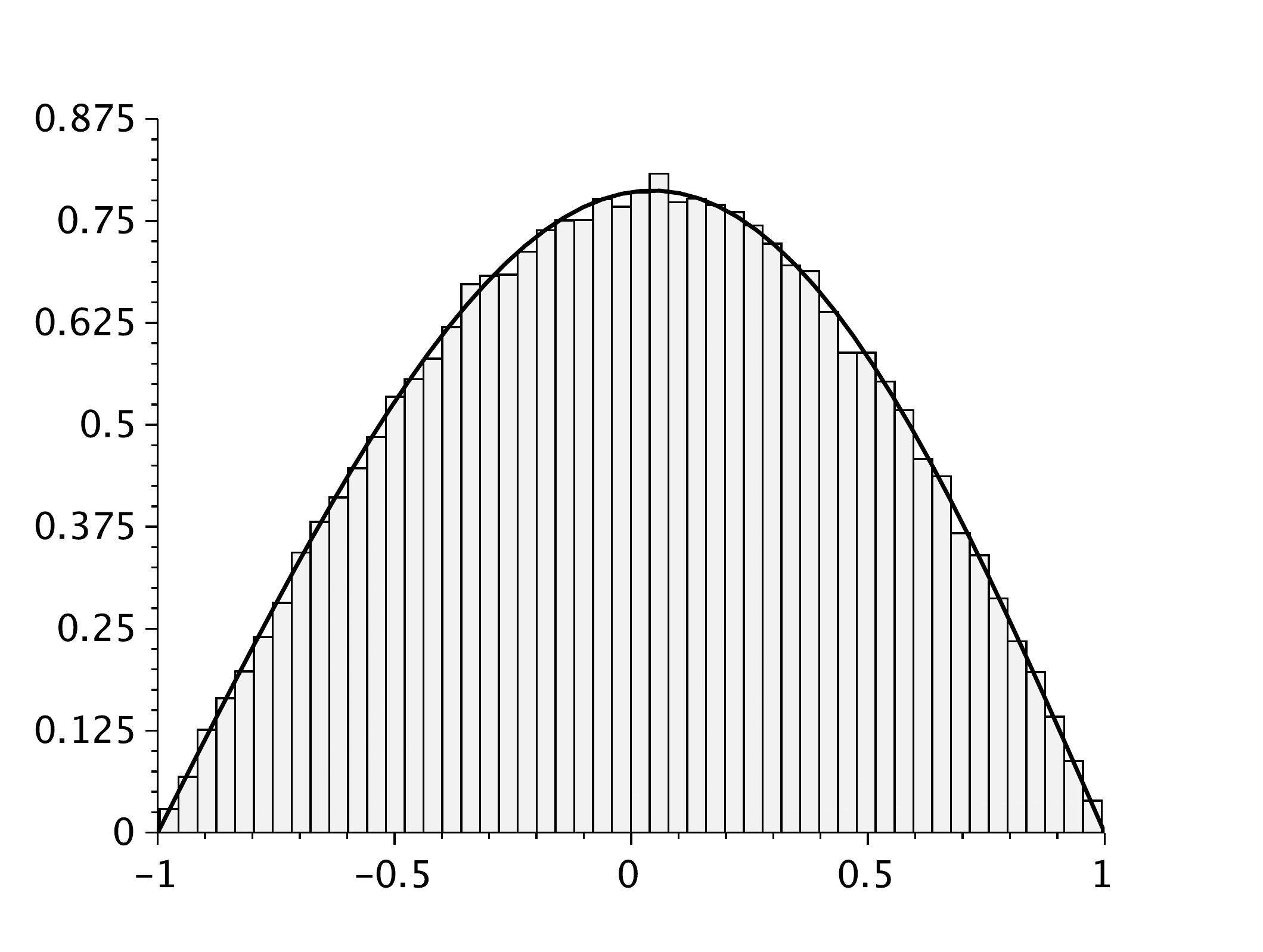}}
\caption{\small Conditional distribution (probability density function and normalized histogram) for $x=0.5$, $t=0.2$ (left) or $t=1$ (right) and $100\,000$ simulations.}
\label{fig:cond2}
\end{figure}
\begin{figure}[h]
\centerline{\includegraphics[height=6cm]{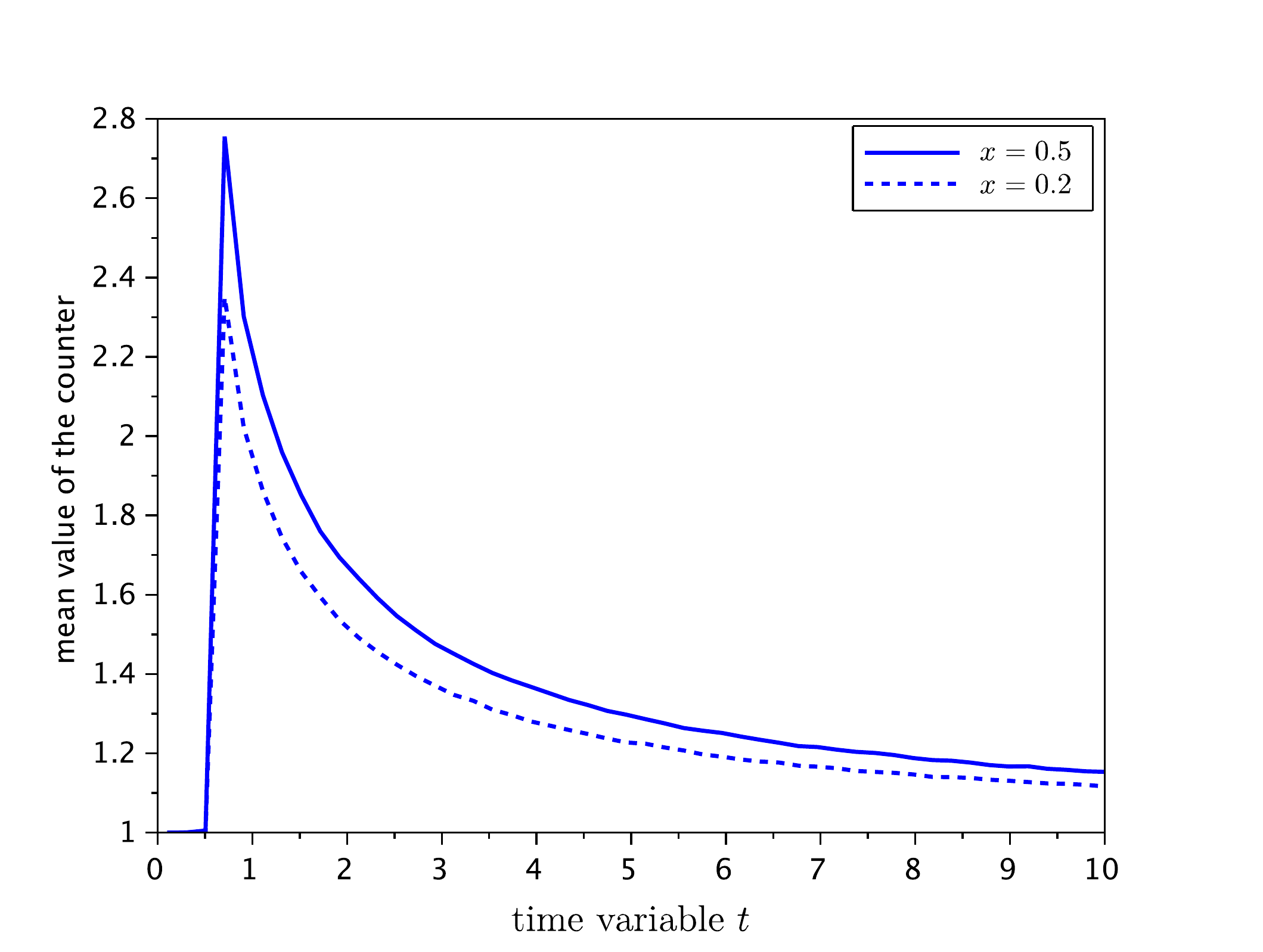}}
\caption{\small Average value of the algorithm counter $\mathcal{N}_c$ versus the time variable,  for the starting position $x=0.5$ (solid line) and $x=0.2$ (dashed line) with $\tcond=0.7$ and $100\,000$ simulations.}
\label{fig:cond2-bis}
\end{figure}
%
%
%
%
%
%
%
\subsection*{Efficiency of the algorithms}
Let us just recall classical results concerning the efficiency of the \emph{convergent series method} introduced in Section \ref{subsec:BMconstrained-series}.  We introduce $\mathcal{N}_c$ the random number of computations of terms $f_n$ used in order to simulate just one random variable $X$. In fact $\mathcal{N}_c$ also corresponds to the number of random variable $S$ generated before the algorithm halts. We also define $\mathcal{N}_l$ the local counter depending on $Y$ which represents the number of terms $f_n$ used till the decision of acceptance or rejection of the variable $Y$ can be taken. Theorem IV.5.2 in Devroye \cite{Devroye-1986} emphasizes the following upper-bound:
\begin{equation}\label{eq:aver-first-1}
\mathbb{E}[\mathcal{N}_l|Y]\leq \frac{2}{\kappa\, h(Y)}\sum_{n=0}^{\infty} R_n(Y)
\end{equation}
where $R_n(y)$ is the reminder of the series expansion defined in the general framework \eqref{eq:reminder}. Let us also notice that the starting idea behind this bound is quite classical: it suffices to use the classical expansion: $\mathbb{E}[\mathcal{N}_l|Y]\le \sum_{n=0}^\infty \mathbb{P}(\mathcal{N}_l>n|Y)$.
 In the study developed in the previous section, we obtained precise bounds for the reminder terms only for $n\ge 1$. So we shall modify the bound isolating the term $n=0$ - instead of \eqref{eq:aver-first-1} - which plays a crucial role for the description of the algorithm efficiency:
\[
\mathbb{E}[\mathcal{N}_l|Y]\leq  1+\frac{2}{\kappa\, h(Y)}\sum_{n=1}^{\infty} R_n(Y).
\]
Since the uniform bound $|R_n(y)|\le r_n$ holds for $n\ge 1$, and since the average number of iterations is $\kappa/I(f)$, Wald's inequality immediately implies that the total number $\mathcal{N}_c$ satisfies for general positive integrable functions $f$:
\begin{align}\label{eq:aver-first-2}
\mathbb{E}[\mathcal{N}_c]\le \frac{\kappa}{I(f)}\int_{-1}^1 \mathbb{E}[\mathcal{N}_l|Y=y]h(y)\,dy\le \frac{1}{I(f)}  \Big(\kappa + 4\sum_{n=1}^\infty r_n\Big).
\end{align}
Let us describe the consequences of this general statement to our particular algorithms.
\mathversion{bold}
\subsection*{For small values of the time variable $t$.}
\mathversion{normal}
For small $t$, it is useful to use the series expansion \eqref{eq:first kind} and the associated reminder bounds presented in Table \ref{table1}.
\begin{prop} \label{prop:average_number_first_kind}
For $t>0$ and $x\in[-1,1]$, we observe the following bound for the number of computations in the convergent series algorithm associated to the expansion \eqref{eq:first kind}:
\begin{equation}
\label{eq:prop:average_number_first_kind}
\mathbb{E}[\mathcal{N}_c]\,\mathbb{P}_x(\tau>t)\le \mathcal{U}_1(t):= 3+\lfloor \sqrt{t}/4\rfloor +\frac{\sqrt{t}}{2\sqrt{2\pi}}\, \vartheta_2(0,e^{-8/t}),
\end{equation}
where $\vartheta_2$ stands for the Jacobi theta function given by
\[
\vartheta_2(z,q):=2q^{1/4}\sum_{n=0}^{\infty}q^{n(n+1)}\cos((2n+1)z),
\]
see \cite{Whittaker} for instance. 
In particular the r.h.s of the previous inequality tends to $3$ as $t$ tends to $0$.
\end{prop}
\begin{proof} The arguments are based on the upper-bound \eqref{eq:aver-first-2}. Let us recall that  
\[
r_n(t,x)=\frac{1}{4}\left( 1-{\rm erf}\Big( \frac{4n-2}{\sqrt{2t}} \Big) \right), \quad \mbox{for}\quad n\ge 1.
\]
Using the bound 
\begin{equation}\label{eq:bound1}
1-{\rm erf}(x)\leq \frac{1}{x\sqrt{\pi}}\ e^{-x^2},
\end{equation}
we obtain
\begin{align*}
I(f)\mathbb{E}[\mathcal{N}_c]&\leq \kappa+\frac{\sqrt{2t}}{\sqrt{\pi}}\sum_{n=1}^{\infty}\frac{e^{-\frac{\left({4n-2}\right)^2}{2t}}}{4n-2}\le \kappa +\frac{\sqrt{t}}{\sqrt{2\pi}}e^{-2/t}\sum_{n=1}^{\infty} e^{-\frac{\left({4n-2}\right)^2-4}{2t}}\\
&\leq  \kappa +\frac{\sqrt{t}}{\sqrt{2\pi}}e^{-2/t}\sum_{n=0}^{\infty}\Big( e^{-\frac{8}{t}}\Big)^{n(n+1)} \le \kappa + \frac{\sqrt{t}}{2\sqrt{2\pi}}\, \vartheta_2(0,e^{-8/t}).
\end{align*}
In order to conclude, it suffices to replace $\kappa$ by (\ref{eq:k smallt}) and to notice that $I(f)=\mathbb{P}_x(\tau>t)$.
\end{proof}
\mathversion{bold}
\subsection*{For large values of the time variable $t$.}
\mathversion{normal}
For large values of $t$ it is more convenient to use the series expansion \eqref{eq:second kind}. A similar approach to Proposition \ref{prop:average_number_first_kind} easily leads to the upper bound:
\begin{prop}
\label{prop:average_number_second_kind}
Let $t\ge 0$ and $x\in[-1,1]$. The number of computations $\mathcal{N}_c$ for the convergent series algorithm associated to the expansion \eqref{eq:second kind} satisfies
\begin{equation}
\label{eq:prop:average_number_second_kind}
\mathbb{E}[\mathcal{N}_c]\,\mathbb{P}_x(\tau>t)\le  \mathcal{U}_2(t):= \kappa(t,x)+\frac{8}{\pi^2t}\, \Big(\vartheta_3(0,e^{-\pi^2t/8})-1\Big),
\end{equation}
where the constant $\kappa(t,x)$ associated to \eqref{eq:second kind} is described in Section \ref{subsec:BMconstrained-num} and $\vartheta_3$ stands for the Jacobi theta function given by
\[
\vartheta_3(z,q):=1+2\sum_{n=1}^{\infty}q^{n^2}\cos(2nz).
\]
\end{prop}
Since the proofs of Proposition \ref{prop:average_number_second_kind} and Proposition \ref{prop:average_number_first_kind} are similar, we let the details of the proof to the reader.
%
%
%
%
%
%
%
\section{Distribution of the Brownian first exit time}
\label{sec:BFET}
In this section, we focus our attention to the distribution of the first Brownian exit $(\tau,B^x_{\tau})$, where $B^x$ stands for the one-dimensional Brownian motion starting in $x$ and $\tau$ the exit time of the interval $[-1,1]$ as defined in \eqref{def:tau-norm}.
%
%
%
%
%
%
%
\mathversion{bold}
\subsection{Distribution of the first exit time $\tau$}
\label{subsec:BFETdistribution}
\mathversion{normal}
 Since $p(t,x,y)$ satisfies two different series expansions \eqref{eq:first kind} and \eqref{eq:second kind}, we can also obtain two series for the cumulative distribution of the exit time $\tau$: 
\begin{align*}
\mathbb{P}_x(\tau\le t)=1-\mathbb{P}_x(\tau>t)=1-\int_{-1}^1p(t,x,y)\dint y.
\end{align*}
Therefore \eqref{eq:first kind} becomes
\begin{align*}
\mathbb{P}_x(\tau\le t)&=1-\sum_{n=-\infty}^{+\infty}\Big\{ -\Phi((x-1-4n)/\sqrt{t})+\Phi((x+1-4n)/\sqrt{t})\\
&-\Phi((x+1-2-4n)/\sqrt{t})+\Phi((x-1-2-4n)/\sqrt{t}) \Big\}\\
&=1-2\sum_{n=-\infty}^{+\infty}\Big\{ -\Phi((x-1-4n)/\sqrt{t})+\Phi((x+1-4n)/\sqrt{t})\Big\}\\
&=1+2\sum_{n=-\infty}^{+\infty}(-1)^n\Phi((x-(2n+1))/\sqrt{t}).
\end{align*}
We deduce easily the expression of the \emph{pdf}:
\begin{equation}
\label{eq:pdf_tau}
p_\tau(t)=\sum_{n=-\infty}^{+\infty}(-1)^{n+1}\frac{ (x-(2n+1))}{t^{3/2}}\ \ \phi((x-(2n+1))/\sqrt{t}).
\end{equation}
Of course the particular case $x=0$ ensures simplifications:  the symmetry of the function $\phi$ implies $\phi(-(2n+1)/\sqrt{t})=\phi(-(2\times(-n-1)+1)/\sqrt{t})$ and therefore
\begin{align}
\label{eq:pdf_tau_1}
p_\tau(t)&=\sum_{n=0}^{+\infty}(-1)^{n}R_1(2n+1,t)\quad \mbox{with}\ R_1(n,t):=\frac{ 2n}{t^{3/2}}\ \ \phi\Big(\frac{n}{\sqrt{t}}\Big).
\end{align}
Such an expression for the \emph{pdf} of $\tau$ is of prime interest for simulation purposes. Let us note that the reminder of the series is small for small values of $t$. Indeed the sequence $(R_1(2n+1,t))_{n\ge 1}$ is a decreasing sequence under the assumption $t\le 9$. For large values $t$, it is more convenient to consider \eqref{eq:second kind} which leads to
\begin{align*}
\mathbb{P}_x(\tau> t)&=\frac{4}{\pi}\sum_{n\ge 0}\frac{1}{2n+1}\exp\Big(-\frac{(2n+1)^2\pi^2}{8}\ t\Big)\sin\Big((2n+1)(x+1)\pi/2\Big), 
\end{align*}
(see also (3.3) in Milstein and Tretyakov \cite{Milstein-Tretyakov}). We deduce the expression
\begin{equation}\label{eq:pdf_gen}
p_\tau(t)=\frac{\pi}{2} \sum_{n=0}^{+\infty}(2n+1)\exp\Big(-\frac{(2n+1)^2\pi^2}{8}\ t\Big)\sin\Big((2n+1)(x+1)\pi/2\Big).
\end{equation}
Here also the case $x=0$ plays a crucial role due to its simple expression:
\begin{equation}\label{eq:pdf_tau_2}
p_\tau(t)=\sum_{n=0}^{+\infty}(-1)^{n}R_2(2n+1,t)\quad \mbox{with}\ R_2(n,t):=\frac{\pi n}{2}\exp\Big(-\frac{n^2\pi^2}{8}\ t\Big).
\end{equation}
This expression is of prime interest as soon as  the variable $t$ is large. Indeed the series becomes alternating with a decreasing sequence $(R_2(2n+1))_{n\ge 1}$ as soon as $t\ge 4/(9\pi^2)$. 
%
%
%
%
%
%
%
\subsection{Algorithms and numerics}
\label{subsec:BFET-num}

The aim is to describe an algorithm which permits to simulate both the Brownian exit time $\tau$ and the exit position $B^x_\tau$ of the interval $[-1,1]$. In the previous section, we have pointed out two different series \eqref{eq:pdf_tau} and \eqref{eq:pdf_gen} corresponding to the  \emph{pdf} of $\tau$. We could therefore apply the classical convergent series method for simulation purposes. But here, we prefer to introduce another method based on an iterative procedure, the advantage of our approach is to deal with alternative series rather than general convergent series and therefore the reminder of the series is easier to bound. 

First of all, we shall focus our attention to the case $x=0$. In this case the interval is symmetric and the expression of the \emph{pdf} of $\tau$ is simplified, see \eqref{eq:pdf_tau_1} and \eqref{eq:pdf_tau_2}. Moreover these series are alternating for suitable conditions on the time variable $t$. So we can apply the \emph{alternating series method} for the simulation of the exit time $\tau$, the exit location being just uniformly distributed in $\{-1,1\}$. The alternating series method is an acceptance/rejection method. We need therefore a proposal distribution and an acceptance procedure.
For the proposal distribution, let us first fix $\texit>0$ (threshold for the exit distribution). Let us consider a standard gaussian random variable $G$. We define a new variable $Y$ as follows: if $1/G^2\le \texit$ then $Y=1/G^2$ otherwise $Y=\texit-\frac{8}{\pi^2}\log(U)$ where $U$ is uniformly distributed on $[0,1]$ and independent of $G$. The density function of $Y$ which corresponds to the proposal distribution of the algorithm satisfies:
\begin{equation}\label{eq:def:hhat}
\hat{h}(t)=\frac{1}{t^{3/2}}\,\phi\Big(\frac{1}{\sqrt{t}}\Big)1_{\{t\le \texit\}}+\frac{\pi}{4\kappa}e^{-\frac{\pi^2}{8}\,t}1_{\{t>\texit\}},
\end{equation}
where $\kappa^{-1}=\pi\ {\rm erf}(\sqrt{1/(2\texit)})\ e^{\pi^2\texit/8}/2$. Using \eqref{eq:pdf_tau_1} and \eqref{eq:pdf_tau_2}, we deduce the following link between the proposal distribution and the target distribution:
\begin{align}
\label{eq:double}
p_\tau(t)=\left\{
\begin{array}{ll}
2\,\hat{h}(t)\Big(1-\frac{R_1(3,t)}{R_1(1,t)}+\frac{R_1(5,t)}{R_1(1,t)}-\ldots  \Big) & \mbox{for}\ t\le  \texit,\\
2\kappa \,\hat{h}(t)\Big( 1-\frac{R_2(3,t)}{R_2(1,t)}+\frac{R_2(5,t)}{R_2(1,t)}-\ldots \Big) & \mbox{for}\ t >  \texit.
\end{array}
\right.
\end{align}
Let us present now the acceptance/rejection method applied to this particular situation which permits to simulate the exit time $\tau$ in the symmetric case.
\begin{framed}
\centerline{BROWNIAN EXIT TIME FOR SYMM. INTERVALS (with parameter $\texit$)}

\vspace*{0.2cm}
\centerline{
\framebox[1.1\width]{\scriptsize BROWNIAN\_EXIT\_SYMMETRIC}}
\vspace*{0.5cm}

\noindent
{\bf First initialization:} $\mathcal{N}_s=0$.\\[5pt]
{\bf Step 0: Second initialization.} \emph{$n=0$, ${\rm Test}=0$, $L_0=0$, $U_0=1$. \\[5pt]
{\bf Step 1.} Generate a random variable $Y$ with \emph{pdf} $\hat{h}$ given by \eqref{eq:def:hhat}. If $Y\le \texit$ then set $i=1$ and $C=1$ else set $i=2$ and $C=\kappa$.\\[5pt]
{\bf Step 2.} Generate a random variable $V$ uniformly distributed on $[0,1]$.\\[5pt]
{\bf Step 3.} While $(V<C U_n)\ \&\ ({\rm Test}=0)$ do:
\begin{itemize}
\item $n\leftarrow n+1$ and $\mathcal{N}_s\leftarrow \mathcal{N}_s+1$,
\item  $L_n=U_{n-1}-\frac{R_i(4n-1,Y)}{R_i(1,Y)}$, $U_n=L_n+\frac{R_i(4n+1,Y)}{R_i(1,Y)}$
\item ${\rm Test}=1_{\{ V\le C L_n \}}$
\end{itemize}
{\bf Step 4.} If ${\rm Test}=1$ then $Z=Y$ otherwise go to Step 0.\\[5pt]
{\bf Outcome:} the random variable $Z$ with density $p_\tau$ and the number of incrementations needed $\mathcal{N}_s$.
}

\end{framed}
The outcome variable distribution $p_\tau$ corresponds to the target one as soon as the sequences $(R_i(2n+1,t)/R_i(1,t))_{n\ge 1}$ appearing in  
\eqref{eq:double} are decreasing. This is an easy adaptation of the classical alternating series method (proof left to the reader). Such a property leads to the condition:
\begin{equation}\label{eq:interval}
\frac{4}{9\pi^2}\le \texit\le 1.
\end{equation}
Let us just note that the probability of acceptance $A=\{{\rm Accept}\ Y\}$ in the acceptance/rejection algorithm satisfies:
\begin{align}\label{predev}
	\mathbb{P}(A)&=\mathbb{P}(A,Y\le \texit )+\mathbb{P}(A,Y> \texit )\nonumber\\
	&=\mathbb{P}\Big(V\le C\frac{p_\tau(Y)}{2C\hat{h}(Y)},Y\le \texit \Big)+\mathbb{P}\Big(V\le C\frac{p_\tau(Y)}{2C\hat{h}(Y)},Y >\texit \Big)=\frac{1}{2}.
\end{align}
We deduce that the number of random variables $Y$ simulated in order to obtain $Z$ is geometrically distributed with average $2$ and does not depend on the choice of $\texit $. Nevertheless the parameter $\texit $ has an influence on the efficiency of the algorithm as illustrated in Fig. \ref{fig:time-bis}.
 This figure represents the averaged value of the number of iterations of Step 3 needed by the algorithm in order to simulate one r.v. with the \emph{p.d.f.} $p_\tau$. This figure suggest to choose a parameter $\texit $ of the order of $1/2$. For this parameter, we obtain:
\[
\kappa^{-1}\approx 2.4529458
\]
\begin{figure}
\centerline{\includegraphics[height=6cm]{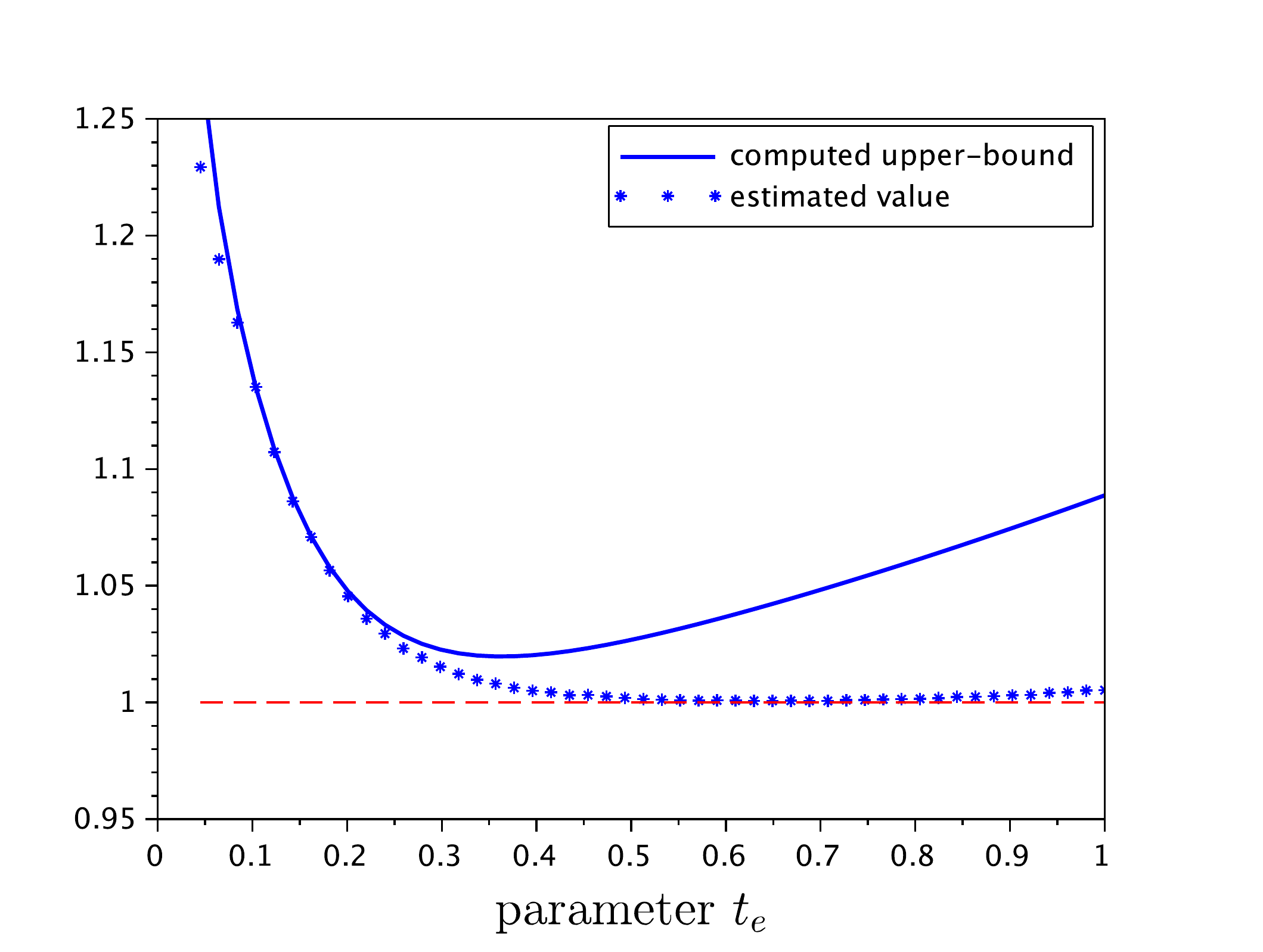}
}
\caption{\small The estimated average value of the counter $\mathcal{N}_s$ versus the barrier $\texit $ with $100\,000$ simulations for the estimation (stars) and the upper-bound function given by \eqref{eq:prop:efficient_sym} (solid line).}
\label{fig:time-bis}
\end{figure}
\begin{prop}
\label{prop:efficient_sym}
Let us note that $\mathcal{N}_s$ is the random number of iterations of Step 3 used in the previous algorithm called {\scriptsize BROWNIAN\_EXIT\_SYMMETRIC} in order to simulate just one random variable $Z$ with density $p_\tau$. Then
\begin{equation}\label{eq:prop:efficient_sym}
\mathbb{E}[\mathcal{N}_s]\le \sqrt{\frac{\texit }{2\pi}}\, e^{-\frac{1}{2\texit }}+\frac{3}{2}\,{\rm erfc}\Big(\frac{1}{\sqrt{2\texit }}\Big)+\frac{4}{\pi}\, e^{-\frac{\pi^2\texit }{8}}+\frac{4}{5\pi}\, \frac{e^{-25\frac{\pi^2\texit }{8}}}{1-e^{-5\pi^2 \texit }},
\end{equation}
where $x\mapsto{\rm erfc}(x)$ is the complementary error function and $\texit $ is the parameter appearing in the algorithm and satisfying  \eqref{eq:interval}. 
\end{prop}
In particular, for $\texit =1/2$, the value of the upper bound is approximatively $1.027$ 
which emphasizes that the inequality is quite sharp in comparison with the estimated number of iterations in Figure \ref{fig:time-bis} (solid line/stars). Moreover the bound is reasonably small which permits to confirm the efficiency of the proposed algorithm.
\begin{proof}
The arguments are quite similar to those developed in Theorem 5.1 in Devroye \cite{Devroye-1986}  for the efficiency of the alternating series method. Let us denote by $\mathcal{N}_s^{\rm loop}$ the number of steps of type 3 used in order to go from Step 1 to Step 4, in other words: the number of increments $n\leftarrow n+1$ during one loop. 
By Wald's equation and since the number of Steps 0 in this algorithm is geometrically distributed with parameter $1/2$ (see \eqref{predev}), we get
\[
\mathbb{E}[\mathcal{N}_s]=2\mathbb{E}[ \mathcal{N}_s^{\rm loop}].
\]
Let us first compute the probability of the following event $\{\mathcal{N}_s^{\rm loop}> 0\}$. This event corresponds to $\{V<C\}$. In the following, we shall just recall that the constant $C$ so as $L_1$, $U_1$... depend on $Y$, that's why we use from now on the notation $C=C_Y$. We obtain:
\begin{align*}
\mathbb{P}(\mathcal{N}_s^{\rm loop}> 0\vert Y)=\mathbb{P}(V<C_Y|Y)=C_Y,
\end{align*}
since $0\le C_Y\le 1$ a.s. Let us now consider the event $\{\mathcal{N}_s^{\rm loop}> 1\}$. This event corresponds to $\{CL_1<V< CU_1\}$ and therefore
\begin{align*}
\mathbb{P}(\mathcal{N}_s^{\rm loop}> 1\vert Y)=C_Y (U_1-\max(L_1,0))\le C_Y(U_1-L_1).
\end{align*}
The definition of $L_1$ and $U_1$ leads to
\begin{align*}
\mathbb{P}(\mathcal{N}_s^{\rm loop}> 1\vert Y)&\le C_Y\cdot\frac{R_i(5,Y)}{R_i(1,Y)},
\end{align*}
where $i=1$ for $Y\le \texit $ and $i=2$ otherwise. The same arguments lead to
\begin{align*}
\mathbb{P}(\mathcal{N}_s^{\rm loop}> k\vert Y)&=\mathbb{P}(C_Y L_{k}<V< C_Y U_k|Y)\\
&=C_Y(U_k-L_k)=C_Y\cdot\frac{R_i(4k+1,Y)}{R_i(1,Y)}.
\end{align*}
We deduce
\[
\mathbb{E}[\mathcal{N}_s^{\rm loop}|Y]=\sum_{k\ge 0}\mathbb{P}(\mathcal{N}_s^{\rm loop}> k\vert Y)= C_Y\sum_{k\ge 0}\frac{R_i(4k+1,Y)}{R_i(1,Y)}.
\]
Using the distribution of the random variable $Y$ with density $\hat{h}$ in \eqref{eq:def:hhat}, we obtain
\begin{align}\label{eq:loop}
\mathbb{E}[\mathcal{N}_s^{\rm loop}]&=\sum_{k\ge 0}\int_0^{\texit }\frac{R_1(4k+1,y)}{R_1(1,y)}\ \hat{h}(y)dy+\kappa\sum_{k\ge 0}\int_{\texit }^\infty\frac{R_2(4k+1,y)}{R_2(1,y)}\ \hat{h}(y)dy\nonumber\\
&=\frac{1}{2}\sum_{k\ge 0}\int_0^{\texit }R_1(4k+1,y)\,dy+\frac{1}{2}\sum_{k\ge 0}\int_{\texit }^\infty R_2(4k+1,y)\,dy\nonumber\\
&=:\frac{1}{2}\, \mathcal{A}_1(\texit )+\frac{1}{2}\, \mathcal{A}_2(\texit ).
\end{align}
Combining the definition of $R_1$ in \eqref{eq:pdf_tau_1} and the change of variable $w=\frac{y}{(4k+1)^2}$ leads to
\begin{align*}
\int_0^{\texit }R_1(4k+1,y)\, dy=\frac{2(4k+1)}{\sqrt{2\pi}}\int_0^{\texit }\frac{e^{-\frac{(4k+1)^2}{2y}}}{y^{3/2}}\, dy=\sqrt{\frac{2}{\pi}}\int_0^{\frac{\texit }{(4k+1)^2}}\frac{e^{-\frac{1}{2w}}}{w^{3/2}}\, dw.
\end{align*}
By Fubini's theorem, we obtain
\begin{align*}
\mathcal{A}_1(\texit )&=\sqrt{\frac{2}{\pi}}\int_{0}^\infty \Big(\sum_{k\ge 0}1_{\{  w\le \frac{\texit }{(4k+1)^2}\}}\Big) \frac{e^{-\frac{1}{2w}}}{w^{3/2}}\, dw\\
&=\sqrt{\frac{2}{\pi}}\int_{0}^\infty \Big\lfloor \frac{1}{4} \sqrt{\frac{\texit }{w}}+\frac{3}{4} \Big\rfloor \frac{e^{-\frac{1}{2w}}}{w^{3/2}}\, dw\\
&\le\sqrt{\frac{\texit }{8\pi}}\int_{0}^{\texit } \frac{1}{\sqrt{w}} \ \frac{e^{-\frac{1}{2w}}}{w^{3/2}}\, dw+\frac{3}{\sqrt{8\pi}}\,\int_0^{\texit }\frac{e^{-\frac{1}{2w}}}{w^{3/2}}\, dw.
\end{align*}
By the change of variable $z=1/w$, we get
\begin{equation}\label{eq:upper_A1}
 \sqrt{\frac{\texit }{8\pi}}\int_{0}^{\texit } \frac{1}{\sqrt{w}} \ \frac{e^{-\frac{1}{2w}}}{w^{3/2}}\, dw= \sqrt{\frac{\texit }{8\pi}} \int_{1/\texit }^\infty e^{-z/2}\, dz=\sqrt{\frac{\texit }{2\pi}}\, e^{-\frac{1}{2\texit }}.
\end{equation}
Moreover, using the change of variable $w=r^{-2}/2$, we have
\begin{equation}\label{eq:upper_A1-2}
 \frac{3}{\sqrt{8\pi}}\,\int_0^{\texit }\frac{e^{-\frac{1}{2w}}}{w^{3/2}}\, dw=\frac{3}{\sqrt{\pi}}\int_{1/\sqrt{2\texit }}^\infty e^{-r^2}\,dr=\frac{3}{2}\,{\rm erfc}\Big(\frac{1}{\sqrt{2\texit }}\Big).
\end{equation}
Combining \eqref{eq:upper_A1} and \eqref{eq:upper_A1-2} we obtain the following upper-bound:
\begin{equation}
\label{eq:upper-A1-fin}
\mathcal{A}_1(\texit )\le \sqrt{\frac{\texit }{2\pi}}\, e^{-\frac{1}{2\texit }}+\frac{3}{2}\,{\rm erfc}\Big(\frac{1}{\sqrt{2\texit }}\Big).
\end{equation}
Let us note that $\mathcal{A}_1(\texit )$ becomes small as $\texit $ becomes small.
 Let us now focus our attention to $\mathcal{A}_2$ defined in \eqref{eq:loop}. \begin{align}\label{eq:upper_A2}
\mathcal{A}_2(\texit )&=4\sum_{k\ge 0}\frac{\exp(-(4k+1)^2\frac{\pi^2\texit  }{8})}{(4k+1)\pi}\nonumber\\
&\le \frac{4}{\pi}\, e^{-\frac{\pi^2\texit }{8}}+\frac{4}{5\pi}\,e^{-25\frac{\pi^2 \texit }{8}}\sum_{k\ge 1}\exp\Big( -((4k+1)^2-5^2)\, \frac{\pi^2\texit  }{8} \Big)\nonumber\\
&\le \frac{4}{\pi}\, e^{-\frac{\pi^2\texit }{8}}+\frac{4}{5\pi}\,e^{-25\frac{\pi^2 \texit }{8}}\sum_{k\ge 0}e^{-5k\pi^2 \texit }\nonumber\\
&= \frac{4}{\pi}\, e^{-\frac{\pi^2\texit }{8}}+\frac{4}{5\pi}\,\frac{e^{-25\frac{\pi^2 \texit }{8}}}{1-e^{-5\pi^2 \texit }}.
\end{align}
Combining \eqref{eq:upper-A1-fin} and \eqref{eq:upper_A2} permits to obtain the announced bound.
\end{proof}
Let us now come to the exit simulation in the asymmetric case. Since we know how to handle with symmetric intervals, we shall use this first algorithm {\scriptsize BROWNIAN\_EXIT\_SYMMETRIC} in an iteration procedure in order to solve the asymmetric case. The main idea consists in the following :
\begin{itemize}
\item starting in $x$, we consider the largest  interval centered in $x$ of the type $[x-\delta,x+\delta]$ and included in $[-1,1]$. 
\item We simulate the exit time of this interval denoted by $T_1$ and the exit position will be uniformly distributed in $\{x-\delta,x+\delta\}$. 
\item If $B^x_{T_1}\in\{-1,1\}$ then we set $T=T_1$ and $X=B^x_{T_1}$ else we start a new simulation for the exit time and position of a Brownian motion $\tilde{B}^x$ starting in $x=B_{T_1}$ from the largest symmetric interval centered in $x$ and included in $[-1,1]$. This exit time is denoted by $T_2$. 
\item As above, if $\tilde{B}^x_{T_2}\in\{-1,1\}$ then we set $X=\tilde{B}^x_{T_2}$ and $T=T_1+T_2$ else we start again with a new initial position in the interval $[-1,1]$. 
\end{itemize}
\begin{framed}
\centerline{BROWNIAN EXIT TIME FOR ASYMMETRIC INTERVALS $[a,b]$}

\vspace*{0.2cm}
\centerline{
\framebox[1.1\width]{\scriptsize BROWNIAN\_EXIT\_ASYMM}}
\vspace*{0.5cm}

\noindent
{\bf Input:} $x$ (initial value of the Brownian paths) and $[a,b]$.\\[5pt]
{\bf Step 0: Initialization.}\\ \emph{$X=x$, $T=0$, ${\rm test}=0$, $L=a$, $U=b$ and $\mathcal{N}_{\rm as}=0$. \\[5pt]
 {\bf While} {\rm (test=0)} {\bf do:}\\[5pt]
{\bf Step 1.} Set $D=\min(X-L,U-X)$. Generate a random variable $\tau$ (and the number of iterations $\mathcal{N}_s$) using the algorithm {\scriptsize BROWNIAN\_EXIT\_SYMMETRIC} and define $S=D^2\tau$. Set $T\leftarrow T+S$ and $\mathcal{N}_{\rm as}\leftarrow \mathcal{N}_{\rm as}+\mathcal{N}_{\rm s}$.\\[5pt]
{\bf Step 2.} Generate a random variable $V$ uniformly distributed on $\{-D,D\}$. Set $X\leftarrow X+V$\\[5pt]
{\bf Step 3.} If $X\in \{a,b\}$ then ${\rm test}=1$ else set $L\leftarrow X-L$ and $U\leftarrow U-X$.\\[5pt]
{\bf End While}\\[5pt]
{\bf Outcome:} the exit time $T$ and the exit location $X$ of the interval $[a,b]$  and the total number of iterations $\mathcal{N}_{\rm as}$.
}

\end{framed}
The number of iterations is  stochastically upper-bounded by a geometrically distribution with parameter $1/2$ since we only deal with symmetric intervals. Moreover this random number is independent of the generation cost of any Brownian symmetric exit time and position. Let us finally note that $(T,X)$ and $(\tau, B^x_\tau)$ are obviously identically distributed.

This algorithm can be illustrated by Fig. \ref{fig:exit-nonsymm} for the standard Brownian exit time of the asymmetric interval $[-1.5,2]$.  Of course the algorithm is not restricted to the standard Brownian case, it takes into account any initial position $x$ belonging to the interval $[a,b]$.
\begin{figure}
\centerline{\includegraphics[height=5.5cm]{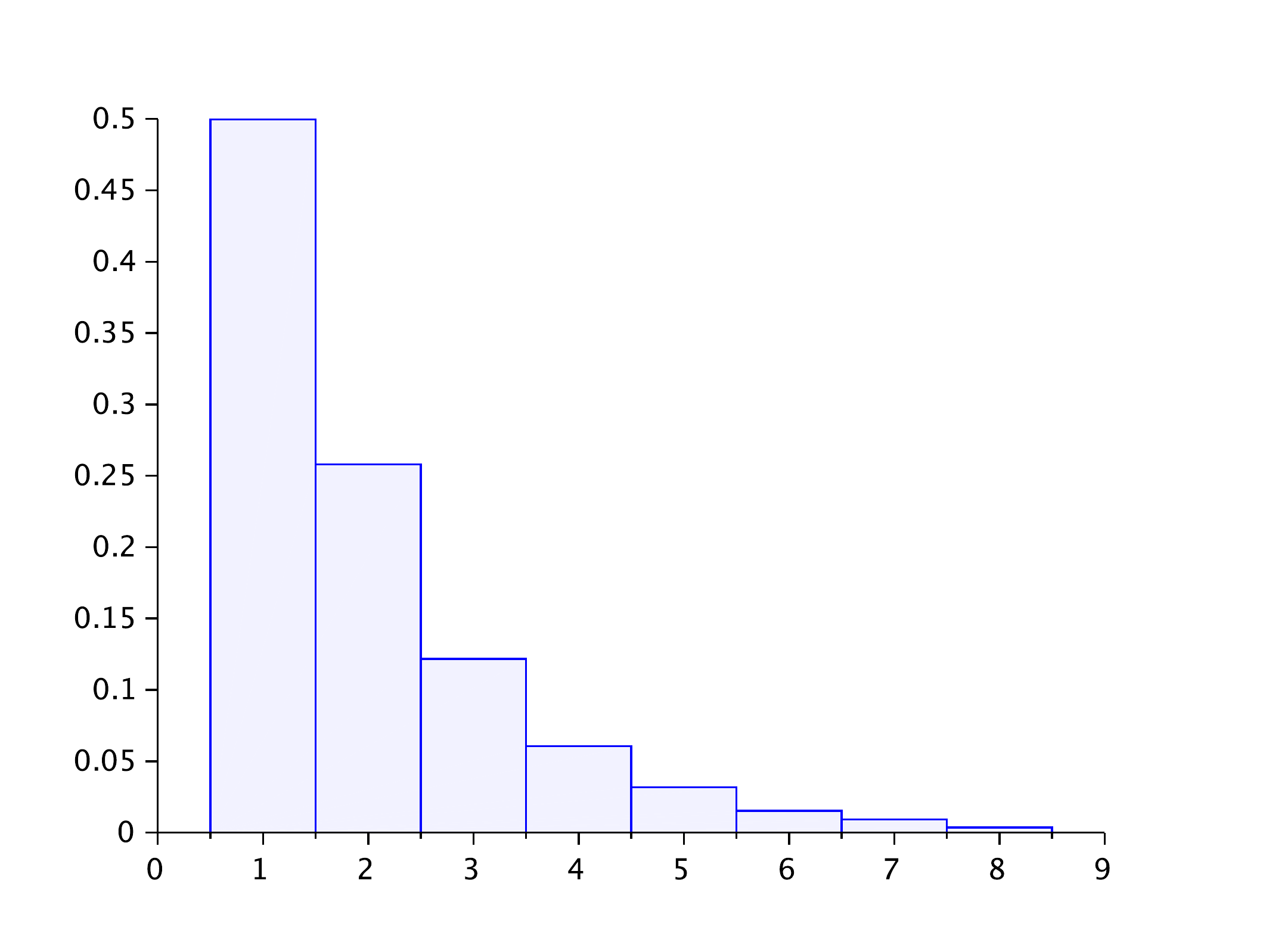}\includegraphics[height=5.5cm]{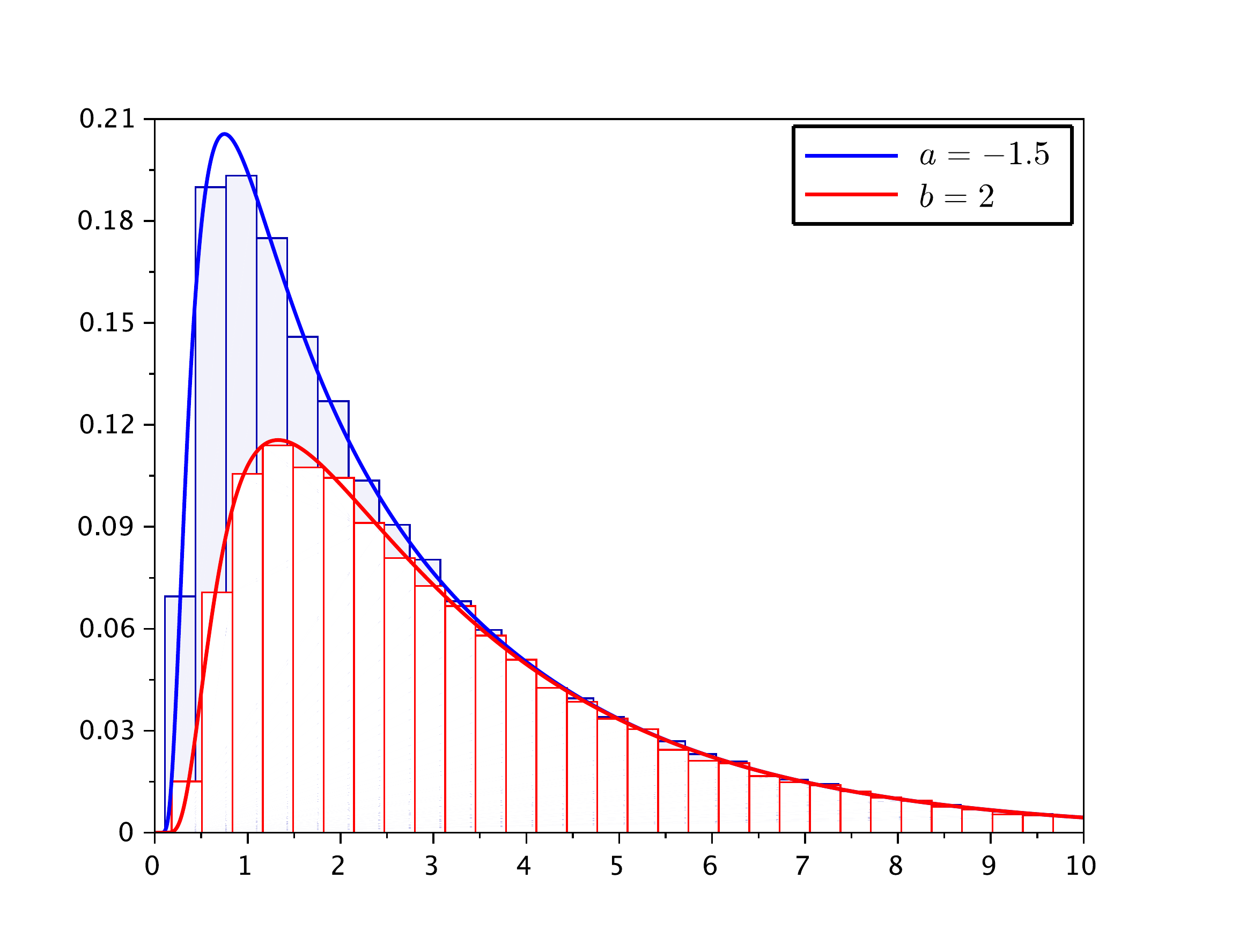}
}
\caption{\small Histogram of the algorithm counter corresponding to the Brownian exit from the interval $[-1.5,2]$ (left), p.d.f. and histograms of the exit time when exiting at the top or at the bottom of this interval (right). A sample of $100\,000$ simulations has been used for these figures and $t_e=0.5$.}
\label{fig:exit-nonsymm}
\end{figure}
\begin{corollary}\label{cor:asymm} The algorithm {\scriptsize BROWNIAN\_EXIT\_ASYMM} involves a random number of calls to {\scriptsize BROWNIAN\_EXIT\_SYMMETRIC} whose efficiency is characterized by the number of iterations $\mathcal{N}_s$. That's why the total number of iterations $\mathcal{N}_{\rm as}$ is directly linked to the efficiency of {\scriptsize BROWNIAN\_EXIT\_ASYMM}. We have
\[
\mathbb{E}[\mathcal{N}_{\rm as}]\le 2\mathbb{E}[\mathcal{N}_s]\le\sqrt{\frac{2\texit }{\pi}}\, e^{-\frac{1}{2\texit }}+3\,{\rm erfc}\Big(\frac{1}{\sqrt{2\texit }}\Big)+\frac{8}{\pi}\, e^{-\frac{\pi^2\texit }{8}}+\frac{8}{5\pi}\, \frac{e^{-25\frac{\pi^2\texit }{8}}}{1-e^{-5\pi^2 \texit }},
\]
where $x\mapsto{\rm erfc}(x)$ is the complementary error function and $\texit $ is the parameter appearing in {\scriptsize BROWNIAN\_EXIT\_SYMMETRIC}. 
\end{corollary}
The statement is a direct consequence of the geometrical distributed upper-bound of the number of calls to the symmetric case algorithm on one hand (Wald's identity therefore leads to $\mathbb{E}[\mathcal{N}_{\rm as}]\le 2\mathbb{E}[\mathcal{N}_s]$) and of Proposition \ref{prop:efficient_sym} on the other hand.
%
%
%
%
%
%
%
\section{First exit time for one-dimensional diffusion}
\label{sec:DFET}
This section is concerned with the exit problem for a one-dimensional diffusion. Let us first consider $(X_t,\ t\ge 0)$ the solution of the following stochastic differential equation:
\begin{equation}\label{eq:0}
dX_t=\mu(X_t)dt+\sigma(X_t)dB_t,\quad X_0=x\in[a,b],
\end{equation}
where $(B_t,\ t\ge 0)$ stands for the standard one-dimensional Brownian motion. As already introduced in \eqref{eq:def:tau}, we denote by $\tau_{a,b}(X)$ the first exit time of the interval $[a,b]$. It suffices to assume the existence of a unique weak solution to the equation, see for instance \cite{K-S} for the corresponding conditions.  In order to simplify the presentation of all algorithms, we restrict our study to the constant diffusion case: $\sigma(x)\equiv 1$. Using the classical \emph{Lamperti transform},  we observe that this restriction is not sharp at all. That is why, from now on, $X$ stands for the unique solution of 
\begin{equation}
\label{eq:simple-sde}
dX_t=\mu(X_t)dt+dB_t, \quad t\ge 0, \quad X_0=x\in[a,b].
\end{equation}
Let us note that, for the particular Brownian case: $\mu(x)\equiv 0$, the first exit time has already been presented in Section \ref{sec:BFET}. The aim is to use the results developed in the previous sections when considering the general diffusion case and the important tool for such a strategy is Girsanov's formula. 
%
%
%
%
%
%
%
%
%
%
%
%
%
%
\subsection{First Exit Time Algorithm (DET)}
\label{subsec:DFET-algo}
Let us now consider the exact simulation algorithm which permits to handle with the diffusion exit problem. 
\begin{framed}
\centerline{DIFFUSION EXIT TIME (DET)} 
\centerline{Parameter: $\gamma_0$, input functions $\gamma(\cdot)$ and $\beta(\cdot)$}

\vspace*{0.5cm}

\noindent
 {\bf First initialization:} $\mathcal{N}_{\rm tot}=0$.\\[5pt]
{\bf Step 0: Initialization.} \emph{$Z=x$, $T=0$, ${\rm test}=0$. Here $x$ stands for the initial value of the diffusion. \\[5pt]
{\bf While} {\rm (test=0)} {\bf do:}\\[5pt]
{\bf Step 1.} Generate an expon. distr. random variable $E$ with parameter $\gamma_0$ and   $U$ and $V$ two random variables uniformly distributed on $[0,1]$, the variables $E$, $U$ and $V$ being independent.\\[5pt]
{\bf Step 2.} Simulate the Brownian exit time and location 
\[
(S,Y,\mathcal{N}_{\rm as})=\mbox{\scriptsize BROWNIAN\_EXIT\_ASYMM\, (Z,[a,b]) }
\]
 and set $\mathcal{N}_{\rm tot}\leftarrow \mathcal{N}_{\rm tot}+\mathcal{N}_{\rm as}$.\\[5pt]
{\bf Step 3.} {\bf If} $S<E$ {\bf then} 
\begin{itemize}
\item {\bf if} $U\le \beta(Y)$ {\bf then} set ${\rm test}=1$, $Z\leftarrow Y$ and $T\leftarrow T+S$\\  
\hspace*{1.9cm} {\bf else} go to Step 0 {\bf end if}, 
\end{itemize}
{\bf  else}
\begin{itemize}
\item simulate the Brownian location at time $E$ given that the exit time is larger than $E$:
\(
(Y_c,\mathcal{N}_c)=\mbox{\scriptsize CONDITIONAL\_DISTR\,(Z,[a,b],E)}
\),  and set $\mathcal{N}_{\rm tot}\leftarrow \mathcal{N}_{\rm tot}+\mathcal{N}_c$.
\item if $\gamma_0 V\le \gamma (Y_c)$ then go to Step 0 else set $Z\leftarrow Y_c$ and $T\leftarrow T+E$.
\end{itemize}
{\bf end if.}\\[5pt]
{\bf End While}\\[5pt]
{\bf Outcome:} the exit location $Z$, the exit time $T$ of the interval $[a,b]$  and the efficiency index $\mathcal{N}_{\rm tot}$.
}

\end{framed}
Let $\rho\ge 0$ such that $\gamma(x):=\frac{\mu^2(x)+\mu'(x)}{2}+\rho$ is a non negative function on the interval $[a,b]$. We introduce $\gamma_+=\sup_{x\in[a,b]}\gamma(x)$ and define $\Delta=\int_a^b \mu(y)dy$ and 
\begin{equation}\label{def:beta}
\beta(x)=\min\Big(1,e^{-\Delta}\Big)\, \exp\int_a^x \mu(y)\,dy.
\end{equation}
Let us just notice that $0\le \beta(a)\le 1$ and $0\le \beta(b)\le 1$.
\begin{thm}
\label{thm:general_exit}
The outcome $(Z,T)$ of the Algorithm (DET) with parameter $\gamma_0=\gamma_+$ and input functions $\gamma(\cdot)$ and $\beta(\cdot)$  satisfies: for any non-negative measurable functions $f$ and $g$, 
\begin{equation}\label{eq:thm:exit}
\mathbb{E}[f(T)g(Z)]=\frac{\mathbb{E}_x\Big[f(\tau_{a,b}(X))e^{-\rho\,\tau_{a,b}(X)}g(X_{\tau_{a,b}(X)})\Big]}{\mathbb{E}_x[e^{-\rho\,\tau_{a,b}(X)}]},
\end{equation}
where $X$ stands for the diffusion defined by \eqref{eq:simple-sde}. In particular, if $\mu'+\mu^2$ is a non-negative function on the interval $[a,b]$ (we set $\rho=0$), then the outcome $(Z,T)$ has the same distribution as $(X_{\tau_{a,b}(X)},\tau_{a,b}(X))$.
\end{thm}
\begin{proof} The proof of Theorem \ref{thm:general_exit} is organized as follows: first we introduce a stochastic theoretical model which is based on the one-dimensional Brownian motion and on an independent Poisson process. This model is closely related to the exact simulation introduced by Beskos and Roberts in \cite{beskos2005exact}. Secondly we consider in details the outcome of Algorithm (DET) and point out the strong link between the outcome and random variables derived from the theoretical model of the first step. Finally we shall use the Girsanov transformation in order to conclude the proof.\\
\emph{Step 1. A theoretical model.}
Let us consider for any non-negative functions $f$ and $g$:
\begin{equation}
\label{eq:Girs}
I(x,f,g):=\mathbb{E}_x\Big[f(\tau_{B})g(B_{\tau_{B}}) e^{-\int_0^{\tau_{B}}\gamma(B_s)\,ds}\Big],
\end{equation}
where $\tau_B$ means $\tau_{a,b}(B)$ for notational simplicity and $B$ represents a Brownian motion starting in $x\in [a,b]$ and $\gamma$ a non negative continuous function. 
Let us now introduce a Poisson point process $N$ on the space $\mathbb{R}_+\times [0,\gamma_{+}]$ with Lebesgue intensity measure and independent of the Brownian paths. If we reorganize the points of the Poisson process with respect to the abscissa, we obtain a sequence of points defined by $(\xi_n,U_n)_{n\ge 1}$ where $(U_n)$ are independent uniformly distributed r.v. on $[0,\gamma_+]$ and $\xi_n:=\sum_{j=1}^n e_j$ with $(e_j)_{j\ge 1}$ a sequence of independent exponentially distributed r.v. with parameter $\gamma_+$. Both sequences $(U_n)_n$ and $(e_n)_n$ are independent. Therefore, for any subset $A$ of $ \mathbb{R}_+\times [0,\gamma_{+}]$, we get
\[
N(A)=\sum_{n\ge 1}1_{\{ (\xi_n,U_n)\in A \}}\quad \mbox{and}\quad \mathbb{P}(N(A)=0)=e^{-\lambda(A)},
\]
 where $\lambda(\cdot)$ is the Lebesgue measure. In particular, if $A$ is defined as follows:
 \[
 A:=\Big\{ (t,y)\in [0,\tau_B]\times [0,\gamma_{+}]:\ y\le \gamma(B_t) \Big\}
 \]
the independence of both the Brownian motion and the Poisson process implies
\begin{align*}
I(x,f,g)&=\mathbb{E}_x\Big[ f(\tau_B)g(B_{\tau_B})1_{\{ N(A)=0 \}} \Big]=\sum_{n\ge 0}I_n(x,f,g)
\end{align*}
where
\begin{align*}
I_n(x,f,g):=\left\{\begin{array}{l}
\mathbb{E}_x\Big[ f(\tau_B)g(B_{\tau_B})1_{\{ N(A)=0 \}} 1_{\{ \xi_n<\tau_B \le \xi_{n+1} \}}\Big],\quad\mbox{for}\ n\ge 1,\\[8pt]
\mathbb{E}_x\Big[ f(\tau_B)g(B_{\tau_B})1_{\{ N(A)=0 \}} 1_{\{ \tau_B \le \xi_{1} \}}\Big],\quad n=0.\end{array}\right.
\end{align*}
Let us first observe that $\{ \tau_B \le \xi_{1} \}\subset  \{N(A)=0 \}$ which implies that
\begin{equation}\label{eq:I0}
I_0(x,f,g)=\mathbb{E}_x\Big[ f(\tau_B)g(B_{\tau_B}) 1_{\{ \tau_B \le \xi_{1} \}}\Big].
\end{equation}
For $n\ge 1$, due to the Markov property of the Brownian path and since the Poisson process is independent of the Brownian motion, we have
\begin{align}\label{eq:recu0}
I_n(x,f,g)&=\mathbb{E}_x\Big[ f(\tau_B)g(B_{\tau_B})1_{\{ U_1>\gamma(B_{\xi_1}),\ldots, U_n> \gamma(B_{\xi_n}) \}} 1_{\{ \xi_n<\tau_B \le \xi_{n+1} \}}\Big]\nonumber\\
&=\mathbb{E}_x\Big[I_{n-1}(B_{\xi_1}, f(\xi_1+\cdot),g)1_{\{ U_1>\gamma(B_{\xi_1}),\, \xi_1<\tau_B\}}\Big].
\end{align}
Since $\xi_1$ is exponentially distributed and independent of the Brownian motion, we get
\begin{align}\label{eq:recur}
I_n(x,f,g)&=\int_{0}^\infty \mathbb{E}_x[I_{n-1}(B_t,f(t+\cdot),g)1_{\{ U_1>\gamma(B_t),t<\tau_B \}} ] \gamma_+e^{-\gamma_+ t}\,dt\nonumber \\
&=\int_{0}^\infty \mathbb{E}_x[I_{n-1}(B_t,f(t+\cdot),g)1_{\{ U_1>\gamma(B_t)\}}\vert t<\tau_B  ] p(t,x) \gamma_+e^{-\gamma_+ t}\,dt\nonumber\\
&=\int_{\mathbb{R}_+\times[a,b]} I_{n-1}(y,f(t+\cdot),g)\Gamma(y)p(t,x)\mathcal{M}(x,dt,dy),
\end{align}
where $\Gamma(y):=\mathbb{P}(U_1>\gamma(y))$ and $p(t,x)=\mathbb{P}_x(\tau_B>t)$. The probability measure $\mathcal{M}(x,dt,dy)$ represents the distribution of the couple $(T,Y)$ where $T$ is an exponentially distributed r.v. with parameter $\gamma_+$ and $Y$ represents the distribution of $B_T$ the value of the Brownian motion at time $T$ starting in $x$ and given $\{\tau_B>T\}$, $T$ and $(B_t)$ being independent. The recurrence relations \eqref{eq:recu0} and \eqref{eq:recur} are satisfied for any $n\ge 1$.\\[3pt]
\emph{Step 2. Relation between Algorithm (DET) and the Brownian-Poisson theoretical model presented in Step 1.}\\ If we denote by $\mathcal{N}_0$ the number of \emph{Step 0} used during the procedure which leads to the computation of the outcome $(Z,T)$ and $\mathcal{N}_1$ the number of exponentially distributed random variables generated (\emph{Step 1}), then we obviously obtain
\begin{align}\label{eq:I0eq}
&\mathbb{E}_x[f(T)g(Z)1_{\{\mathcal{N}_0=1,\ \mathcal{N}_1=1 \}}] =\mathbb{E}[f(S)g(Y)1_{\{U\le \beta(Y),\, S<E  \}}]\nonumber\\
&  =\mathbb{E}[f(S)g(Y)\beta(Y)1_{\{ S<E  \}}]=\mathbb{E}_x\Big[ f(\tau_B)g(B_{\tau_B})\beta(B_{\tau_B}) 1_{\{ \tau_B \le \xi_{1} \}}\Big]\nonumber\\
&=I_0(x,f,g\times\beta)
\end{align}
using \eqref{eq:recu0}.\\  Let us now consider that $\mathcal{N}_1>1$.  On the event $\{\mathcal{N}_0=1\}\cap \{\mathcal{N}_1>1\}$, we observe that
\begin{itemize}
\item the first exponentially distr. r.v. $E$ satisfies $E<S$ where $S$ is given by \emph{Step 2} and corresponds to the Brownian exit time of the interval $[a,b]$. 
\item $Y_c$ given that $E<S$ satisfies $V> \gamma(Y_c)$, where $V$ is uniformly distributed and $Y_c$ becomes the new starting point for the next use of the exit problem.
\end{itemize}
In other words, if we denote by $\mathcal{I}_n:=\mathbb{E}_x[f(T)g(Z)1_{\{\mathcal{N}_0=1,\ \mathcal{N}_1=n+1 \}}]$ then \eqref{eq:I0} leads to $\mathcal{I}_0(x,f,g)=I_0(x,f,g\cdot\beta)$. Moreover by \eqref{eq:recu0}
\begin{align*}
\mathcal{I}_1(x,f,g)&=\mathbb{E}_x[f(T)g(Z)1_{\{\mathcal{N}_0=1,\ \mathcal{N}_1=2 \}}]\\
&=\mathbb{E}_x[I_0(Y_c,f(E+\cdot),g\cdot\beta)1_{\{V>\gamma(Y_c), E<S \}}]\\
&=\mathbb{E}_x[I_0(B_{\xi_1},f(\xi_1+\cdot),g\times\beta)1_{\{U_1>\gamma(B_{\xi_1}), \xi_1<\tau_B \}}]=I_1(x,f,g\times\beta),
\end{align*}
and using the same arguments, we prove easily that  $\mathcal{I}_n$ satisfies the recurrence relations \eqref{eq:recu0} and \eqref{eq:recur}. We deduce that 
\[
\mathcal{I}_n(x,f,g)=I_n(x,f,g\times\beta),\quad\forall n\ge 0.
\]
Therefore 
\begin{equation}
\label{eq:ens}
I(x,f,g\times\beta)=\sum_{n\ge 0} I_n(x,f,g\times\beta)=\sum_{n\ge 0} \mathcal{I}_n(x,f,g)=\mathbb{E}_x[f(T)g(Z)1_{\{\mathcal{N}_0=1\}}].
\end{equation}
Since the Algorithm (DET) is an acceptance/rejection algorithm, \eqref{eq:ens} leads to
\begin{align*}
\mathbb{E}_x[f(T)g(Z)]&=\frac{\mathbb{E}_x[f(T)g(Z)1_{\{\mathcal{N}_0=1\}}]}{\mathbb{P}(\mathcal{N}_0=1)}=\frac{I(x,f,g\times\beta)}{I(x,1,\beta)}.
\end{align*}
By \eqref{eq:Girs} 
\begin{align}\label{eq:fract}
\mathbb{E}_x[f(T)g(Z)]&=\frac{\mathbb{E}_x\Big[f(\tau_{B})g(B_{\tau_{B}})\beta(B_{\tau_{B}}) e^{-\int_0^{\tau_{B}}\gamma(B_s)\,ds}\Big]}{\mathbb{E}_x\Big[\beta(B_{\tau_{B}})e^{-\int_0^{\tau_{B}}\gamma(B_s)\,ds}\Big]}.
\end{align}
\emph{Step 3. The Girsanov transformation.} In this last part of the proof, the aim is to link the distribution of $(T,Z)$ described in \eqref{eq:fract} to the distribution of $ (X_{\tau_{a,b}(X)},\tau_{a,b}(X))$ where $X$ is the diffusion defined by \eqref{eq:simple-sde}. Using the definition of the function $\beta$ and  It\^o's formula, we obtain that
\begin{align}\label{def:martingale}
M_t&:=\frac{\beta(B_t)}{\beta(B_0)}\ e^{\rho\, t-\int_0^{t}\gamma(B_s)\,ds}=\exp\Big\{\int_{B_0}^{B_t}\mu(y)\, dy-\frac{1}{2}\int_0^t \mu'(B_s)+\mu^2(B_s)   \,ds\Big\}\nonumber\\
&=\exp\Big\{\int_{0}^{t}\mu(B_s)\, dB_s-\frac{1}{2}\int_0^t \mu^2(B_s)   \,ds\Big\},
\end{align}
is an exponential martingale. Moreover the stopped martingale $(M_{t\wedge \tau_B})_{t\ge 0}$ is bounded. The stopping theorem therefore implies $\mathbb{E}_x[M_{\tau_B}]=1$ which means that 
\[
\mathbb{E}_x\Big[\beta(B_{\tau_{B}})e^{\rho\,\tau_B-\int_0^{\tau_{B}}\gamma(B_s)\,ds}\Big]=\beta(x).
\]
The expression \eqref{eq:fract} and Girsanov's transformation permits to obtain 
\begin{align*}
\mathbb{E}_x[f(T)g(Z)]&=\frac{\mathbb{E}_x\Big[f(\tau_{B})g(B_{\tau_{B}})M_{\tau_B}e^{-\rho \tau_B}\Big]}{\mathbb{E}_x\Big[M_{\tau_B}e^{-\rho \tau_B}\Big]}\\
&=\frac{\mathbb{E}_x\Big[f(\tau_{a,b}(X))e^{-\rho\,\tau_{a,b}(X)}g(X_{\tau_{a,b}(X)})\Big]}{\mathbb{E}_x[e^{-\rho\,\tau_{a,b}(X)}]},
\end{align*}
where $X$ stands for the diffusion defined by \eqref{eq:simple-sde}. In particular, if $\rho=0$ that is $\mu^2+\mu'$ is a non-negative function, we get
\[
\mathbb{E}_x[f(T)g(Z)]=\mathbb{E}_x\Big[f(\tau_{a,b}(X))g(X_{\tau_{a,b}(X)})\Big].
\]
\end{proof}
%
%
%
%
%
%
%
\subsection{Efficiency of the algorithm}
\label{subsec:DFET-efficiency}

Let us now focus our attention to the analysis of Algorithm (DET). Of course since the algorithm permits to simulate exactly the random variables desired, there is no error terms to deal with, it suffices therefore to describe the time needed by the algorithm. We introduced in Algorithm (DET) the random variable $\mathcal{N}_{\rm tot}$ which permits to have a precise idea of the efficiency. Let us just add the information concerning the starting position of the Brownian motion $B_0=x$ with the following notation $\mathcal{N}_{\rm tot}=\mathcal{N}_{\rm tot}^{x}$.  Let us also note that the efficiency of Algorithm (DET) in particular depends on two different parameters: $\texit$ which appears in the use of the algorithm  $\mbox{\scriptsize BROWNIAN\_EXIT\_SYM\,}$ and therefore also in $\mbox{\scriptsize BROWNIAN\_EXIT\_ASYMM\,}$ and $\tcond\in\mathbb{R}_+$ which appears in $\mbox{\scriptsize CONDITIONAL\_DISTR}$ (if $t\le \tcond$ we consider the algorithm associated to the small values of $t$ and for $t>\tcond$ the algorithm associated to the large values).

Using informations concerning the cost of each part of the algorithm, namely Proposition \ref{prop:average_number_first_kind}, Proposition \ref{prop:average_number_second_kind} and Proposition \ref{prop:efficient_sym}, we obtain a bound for $\mathbb{E}[\mathcal{N}_{\rm tot}^x]$.
\begin{thm}\label{thm:efficiency-exit} The random variable $\mathcal{N}_{\rm tot}^x$ which is one of the outcomes of Algorithm (DET) and represents its cost satisfies the following bound: there exists a constant $C(\tcond,\texit)>0$ independent of both the interval $[a,b]$ and the starting position $x$, such that
\[
\mathbb{E}[\mathcal{N}_{\rm tot}^x]\le \frac{C(\tcond,\texit)}{\mathbb{E}_x[e^{-\rho \tau_{a,b}(X)}]\mathbb{E}_{(a+b)/2}[e^{-\gamma_+ \tau_{a,b}(B)}]},\quad \forall x\in]a,b[.
\]
We recall that $\tau_{a,b}$ stands for the first exit time of the interval $[a,b]$. The parameters $\rho$ and $\gamma_+$ are defined in the introduction of Theorem \ref{thm:general_exit}. 
\end{thm}
Each term appearing in the denominator, that is $\mathbb{E}_x[e^{-\rho \tau_{a,b}(X)}]$ on one hand and $\mathbb{E}_{(a+b)/2}[e^{-\gamma_+ \tau_{a,b}(B)}]$ on the other hand, tends to $0$ when the interval size $b-a$ tends to infinity. It is therefore important to choose $\rho$ and $\gamma_+$ as small as possible in order to obtain a sharper bound. The exit problem of the Brownian motion can be precisely described using classical results on Laplace transforms and differential equations, see for instance \cite{Darling} or \cite{Borodin}. We obtain
\[
\mathbb{E}_{(a+b)/2}[e^{-\gamma_+ \tau_{a,b}(B)}]=\Big\{\cosh\Big(\sqrt{\frac{\gamma_+}{2}}(b-a)\Big)\Big\}^{-1}.
\]
Let us also note that $\mathbb{E}_x[e^{-\rho \tau_{a,b}(X)}]$ can be linked to  the two linearly independent solutions of the differential equation (see, for instance \cite{Darling})
\[
\frac{1}{2}\,\frac{d^2u}{dx^2}+\mu(x)\frac{du}{dx}-\lambda u=0,
\]
which permits to describe the asymptotic behaviour as $b-a$ tends to infinity. For the link between the Laplace transform and the speed measure, see for instance \cite{Bass}.

\begin{proof} The counter $\mathcal{N}_{\rm tot}$ introduced in the algorithm can be decomposed as follows: 
\[
\mathcal{N}_{\rm tot}^x=\sum_{k\ge 1}\mathcal{N}_{\rm tot}^{x,k},
\] 
where $\mathcal{N}_{\rm tot}^{x,k}$ represents the number of counter increases observed in-between the $k$-th and $(k+1)$-th passage through the item \emph{Step 0}. We recall that we defined $\mathcal{N}_0$ in the proof of Theorem \ref{thm:general_exit}: it corresponds to the number of \emph{Step 0} necessary to obtain the desired outcome. Since Algorithm (DET) is an acceptance-rejection algorithm, the random variable $\mathcal{N}_0$ is geometrically distributed. Let us also note that $\mathcal{N}_{\rm tot}^{x,k}=0$ a.s. on the event $\{\mathcal{N}_0<k\}$ and conditionally to $\{\mathcal{N}_0\ge k\} $, $\mathcal{N}_{\rm tot}^{x,k}$ has the same distribution as $\mathcal{N}_{\rm tot}^{x,1}$.
Hence
\begin{align}\label{eq:geom-bound}
\mathbb{E}[\mathcal{N}_{\rm tot}^x]&=\sum_{k\ge 1}\mathbb{E}[\mathcal{N}_{\rm tot}^{x,k} ]=\sum_{k\ge 1}\mathbb{E}[\mathcal{N}_{\rm tot}^{x,k} 1_{\{ \mathcal{N}_0\ge k \}}]=\mathbb{E}[\mathcal{N}_{\rm tot}^{x,1}] \sum_{k\ge 1}\mathbb{P}( \mathcal{N}_0\ge k )\nonumber\\
&=\frac{\mathbb{E}[\mathcal{N}_{\rm tot}^{x,1}]}{\mathbb{P}(\mathcal{N}_0=1)}=\frac{\mathbb{E}[\mathcal{N}_{\rm tot}^{x,1}]}{\mathbb{E}_x[e^{-\rho \tau_{a,b}(X)}]}.
\end{align}
The last equality is related to the third step in the proof of Theorem \ref{thm:general_exit}. Let us now describe $\mathbb{E}[\mathcal{N}_{\rm tot}^{x,1}]$.  Let $(S,Y,\mathcal{N}_{\rm as})$ stands for the result of the first use of the function  \mbox{\scriptsize BROWNIAN\_EXIT\_ASYMM\,} and $(Y_c,\mathcal{N}_c)$ of the first use of \mbox{\scriptsize CONDITIONAL\_DISTR\,}, we can therefore distinguish two different cases. 
\begin{itemize}
\item if $S<E$ then $\mathcal{N}^{x,1}_{\rm tot}=\mathcal{N}_{\rm as}$.
\item if $S>E$ then $\mathcal{N}^{x,1}_{\rm tot}=\mathcal{N}_{\rm as}+\mathcal{N}_c+\widehat{\mathcal{N}}^{Y_c,1}_{\rm tot}1_{\{ \gamma_0V>\gamma(Y_c) \}}$ where $\widehat{\mathcal{N}}^{x,1}_{\rm tot}$ is an independent copy of $\mathcal{N}^{x,1}_{\rm tot}$. Such a property is essentially based on the Markov property of the Brownian paths.
\end{itemize}
We deduce
\begin{align}\label{eq:opt}
\mathbb{E}[ \mathcal{N}^{x,1}_{\rm tot} ]&\le \mathbb{E}[ \mathcal{N}_{\rm as} ]+\mathbb{E}[ \mathcal{N}_{c}1_{\{S>E\}} ]+\mathbb{E}[\widehat{\mathcal{N}}^{Y_c,1}_{\rm tot}1_{\{S>E\}}1_{\{ \gamma_0V>\gamma(Y_c) \}}]\nonumber\\
&\le \mathbb{E}[ \mathcal{N}_{\rm as} ]+\mathbb{E}[ \mathcal{N}_{c}1_{\{S>E\}} ]+\mathbb{E}[\widehat{\mathcal{N}}^{Y_c,1}_{\rm tot}1_{\{S>E\}}]\nonumber\\
&\le\mathbb{E}[ \mathcal{N}_{\rm as} ]+\, \mathbb{E}[\mathbb{E}[ \mathcal{N}_{c}| E ]1_{\{S>E\}} ]+\mathbb{P}(S>E)\sup_{x\in[a,b]}\mathbb{E}[\mathcal{N}^{x,1}_{\rm tot}].
\end{align}
\begin{itemize}
\item Let us note that $\mathbb{P}(S>E)=\mathbb{P}_x(\tau_{a,b}(B)>\xi)$ where $\tau_{a,b}(B)$ is the exit time of the Brownian motion from the interval $[a,b]$ and $\xi$ is exponentially distributed with parameter $\gamma_+$. Using the scaling property \eqref{eq:scaling}, we have
\[
\mathbb{P}_x(\tau_{a,b}(B)>\xi)=\mathbb{P}_y(\tau>4\xi(b-a)^{-2}),
\]
where $\tau$ is the first Brownian exit time of the normalized interval $[-1,1]$ and $y=\frac{2x-a-b}{b-a}$.  Since $y\mapsto \mathbb{P}_y(\tau>4\xi(b-a)^{-2})$ is a concave function whose derivative vanishes for $y=0$ (see the expression \eqref{eq:pdf_gen}), we get 
\[
\mathbb{P}_x(\tau_{a,b}(B)>\xi)\le \mathbb{P}_0(\tau>4\xi(b-a)^{-2}).
\]
\item Since $\mathcal{N}_{\rm as}$ is the cost of the function  $\mbox{\scriptsize BROWNIAN\_EXIT\_ASYMM\,}(x,[a,b]) $ and using the scaling property \eqref{eq:scaling}, we obtain that it is equal to the cost of $\mbox{\scriptsize BROWNIAN\_EXIT\_ASYMM\,}(y,[-1,+1]) $ which satisfies due to Corollary \ref{cor:asymm}: \( \mathbb{E}[\mathcal{N}_{\rm as}]\le C_0(\texit)\) with
\[
C_0(\texit):=\sqrt{\frac{2\texit }{\pi}}\, e^{-\frac{1}{2\texit }}+3\,{\rm erfc}\Big(\frac{1}{\sqrt{2\texit }}\Big)+\frac{8}{\pi}\, e^{-\frac{\pi^2\texit }{8}}+\frac{8}{5\pi}\, \frac{e^{-25\frac{\pi^2\texit }{8}}}{1-e^{-5\pi^2 \texit }},
\]
where $x\mapsto{\rm erfc}(x)$ is the complementary error function and $\texit $ is the parameter appearing in {\scriptsize BROWNIAN\_EXIT\_SYMMETRIC}. This parameter $\texit $ satisfying \eqref{eq:interval} can be chosen in order to minimize this average.
\item Moreover we need some information on $\mathbb{E}_x[ \mathcal{N}_{c}| E=t ]$ where $\mathcal{N}_{c}$ is the cost of the function $\mbox{\scriptsize CONDITIONAL\_DISTR\,}(x,[a,b],t)$. Due to the scaling property \eqref{eq:scaling}, we know that the cost of this function is identical as the cost of  $\mbox{\scriptsize CONDITIONAL\_DISTR\,}(y,[-1,1],4t(b-a)^{-2})$. We choose a parameter $\tcond>0$ such that $4t(b-a)^{-2}\le \tcond $ corresponds to the algorithm for small time values and $4t(b-a)^{-2}>\tcond $ corresponds to large times. Combining Proposition \ref{prop:average_number_first_kind} and Proposition \ref{prop:average_number_second_kind} permits to have the following bound: for any $x\in [a,b]$,
\[
\mathbb{E}_x[ \mathcal{N}_{c}| E=t ]\le \frac{\Theta(4t(b-a)^{-2})}{\mathbb{P}_y(\tau>4t(b-a)^{-2})}\ \mbox{with}\ \Theta(t)=\mathcal{U}_1(t)1_{t\le \tcond }+\mathcal{U}_2(t)1_{t> \tcond }.
\]
Let us introduce $C_1(\tcond ):=\sup_{t\ge 0}\Theta(t)<\infty$, then 
\begin{align*}
\mathbb{E}[\mathbb{E}[ \mathcal{N}_{c}| E ]1_{\{S>E\}} ]&=\int_0^\infty \gamma_+\mathbb{E}[\mathbb{E}[ \mathcal{N}_{c}| E=t ]1_{\{S>t\}} ]e^{-\gamma_+ t}\,dt\\
&\le \int_0^\infty \gamma_+\mathbb{P}(S>t) \frac{\Theta(4t(b-a)^{-2})}{\mathbb{P}_y(\tau>4t(b-a)^{-2})}e^{-\gamma_+ t}\,dt\\
&\le C_1(\tcond )\int_0^\infty \gamma_+\,e^{-\gamma_+ t}\,dt=C_1(\tcond ).
\end{align*}
\end{itemize}
Using these three items and \eqref{eq:opt}, we obtain
\begin{align*}
\mathbb{E}[ \mathcal{N}^{x,1}_{\rm tot} ]&\le \frac{\mathbb{E}[\mathcal{N}_{\rm as}]+\mathbb{E}[\mathbb{E}[ \mathcal{N}_{c}| E ]1_{\{S>E\}} ]}{\mathbb{P}_0(\tau\le 4\xi(b-a)^{-2})}\le \frac{C_0(\texit)+C_1(\tcond )}{\mathbb{E}_0[e^{-\gamma_+(b-a)^2\tau/4}]}.
\end{align*}
The identity \eqref{eq:geom-bound} permits to conclude the proof.

\end{proof}

\subsection{Modifications of Algorithm (DET)}
Under the drift condition $\mu'+\mu^2\ge 0$, Algorithm (DET) permits to simulate in an exact way the exit time of the interval $[a,b]$ for diffusion processes. Let us now improve this algorithm in order to deal with any one dimensional diffusion. This generalization first requires a modified algorithm with outcome $(X_{\tau_{a,b}(X)\wedge \kappa},\tau_{a,b}(X)\wedge \kappa)$ for any $\kappa>0$. The simulation of $(X_{\tau_{a,b}(X)},\tau_{a,b}(X))$ can then be obtained by iteration due to the diffusion Markov property.

Let us present Algorithm ($\kappa$-DET).
\begin{framed}
\centerline{MODIFIED DIFFUSION EXIT TIME ($\kappa$-DET)} 
\centerline{Parameters: $\rho$ and $\gamma_0$, input functions $\gamma(\cdot)$ and $\beta_m(\cdot)$}

\vspace*{0.5cm}

\noindent
{\bf First initialization:} $\mathcal{N}_{\rm tot}=0$.\\[5pt]
{\bf Step 0: Initialization.} \emph{$K=\kappa$, $Z=x$, $T=0$, ${\rm test}=0$. Here $x$ stands for the initial value of the diffusion and $\kappa$ the time upper-bound.\\[5pt]
{\bf While} {\rm (test=0)} {\bf do:}\\[5pt]
{\bf Step 1} Generate an expon. distr. random variable $E$ with parameter $\gamma_0$ and   $U$, $V$ and $W$ three random variables uniformly distributed on $[0,1]$, the variables $E$, $U$, $V$ and $W$ being independent.\\[5pt]
{\bf Step 2.} Simulate the Brownian exit time and location 
\[
(S,Y,\mathcal{N}_{\rm as})=\mbox{\scriptsize BROWNIAN\_EXIT\_ASYMM\, (Z,[a,b]) }
\]
 and set $\mathcal{N}_{\rm tot}\leftarrow \mathcal{N}_{\rm tot}+\mathcal{N}_{\rm as}$.\\[5pt]
{\bf Step 3.} {\bf If} $S=\min(K,E,S)$ {\bf then} 
\begin{itemize}
\item {\bf if} $U\le \beta_m(Y)$ and $W\le e^{-\rho\,(K-S)}$ \\ 
\hspace*{2cm} {\bf then} set ${\rm test}=1$, $Z\leftarrow Y$ and $T\leftarrow T+K$\\  
\hspace*{2cm} {\bf else} go to Step 0 {\bf end if}, 
\end{itemize}
{\bf  elseif} $K=\min(K,E,S)$
\begin{itemize}
\item simulate the conditional Brownian location at time $K$:\\
\(
(Y_c,\mathcal{N}_c)=\mbox{\scriptsize CONDITIONAL\_DISTR\,(Z,[a,b],K)}
\)  and set $\mathcal{N}_{\rm tot}\leftarrow \mathcal{N}_{\rm tot}+\mathcal{N}_c$.
\item {\bf if} $U\le \beta_m(Y_c)$ {\bf then} set ${\rm test}=1$, $Z\leftarrow Y$ and $T\leftarrow T+S$\\  
\hspace*{2cm} {\bf else} go to Step 0 {\bf end if}, 
\end{itemize}
{\bf  elseif} $E=\min(K,E,S)$
\begin{itemize}
\item simulate the conditional Brownian location at time $E$:\\
\(
(Y_c,\mathcal{N}_c)=\mbox{\scriptsize CONDITIONAL\_DISTR\,(Z,[a,b],E)}
\) and set $\mathcal{N}_{\rm tot}\leftarrow \mathcal{N}_{\rm tot}+\mathcal{N}_c$.
\item if $V\le \gamma (Y_c)$ then go to Step 0 else set $Z\leftarrow Y_c$, $T\leftarrow T+E$ and $ K\leftarrow K-E$.
\end{itemize}
{\bf end if.}\\[5pt]
{\bf End While}\\[5pt]
{\bf Outcome:} the stopping time $T$ corresponding to the minimum between the diffusion exit time of the interval $[a,b]$ and the constant time $\kappa$,  the location $Z$ of the stopped diffusion at time $T$ and the efficiency index $\mathcal{N}_{\rm tot}$.}

\end{framed}

Let $\rho\ge 0$ such that $\gamma(x):=\frac{\mu^2(x)+\mu'(x)}{2}+\rho$ is a non negative function on the interval $[a,b]$. We recall that $\gamma_+$ and $\beta$  are defined in \eqref{def:beta}. We introduce 
\begin{equation}
\label{eq:def:betam}
\beta_m(x)=\frac{\beta(x)}{\sup_{x\in[a,b]}\beta(x)}.
\end{equation}
Let us just notice that $0\le \beta_m(x)\le 1$ for any $x\in[a,b]$.
\begin{thm}
\label{thm:general_exit_extension}
The outcome $(Z,T)$ of the Algorithm ($\kappa$-DET) with parameter $\rho$, $\gamma_0=\gamma_+$ and input functions $\gamma(\cdot)$ and $\beta_m(\cdot)$ has the same distribution as $(X_{\tau_{a,b}(X)\wedge \kappa},\tau_{a,b}(X)\wedge \kappa)$, where $X$ stands for the diffusion defined by \eqref{eq:simple-sde}.
\end{thm}
\begin{prop}\label{prop:efficiency-exit-gen} The random variable $\mathcal{N}_{\rm tot}^x$ which is one of the outcomes of Algorithm ($\kappa$-DET) satisfies the following bound: there exists a constant $C(\tcond,\texit)>0$ independent of both the interval $[a,b]$ and the starting position $x$, such that
\[
\mathbb{E}[\mathcal{N}_{\rm tot}^x]\le \frac{C(\tcond,\texit)\ e^{\rho\, \kappa}}{\beta_m(x)\ \mathbb{E}_{(a+b)/2}[e^{-\gamma_+ (\tau_{a,b}(B)\wedge \kappa)}]},\quad \forall x\in]a,b[.
\]
We recall that $\tau_{a,b}$ stands for the first exit time of the interval $[a,b]$. The parameters $\rho$ and $\gamma_+$ are defined in the introduction of Theorem \ref{thm:general_exit}.
\end{prop}
The proof of Proposition \ref{prop:efficiency-exit-gen} is just a slight modification of the proof of Theorem \ref{thm:efficiency-exit}.  The details are left to the reader (see Appendix).
\begin{proof}[Proof of Theorem \ref{thm:general_exit_extension}]
The key arguments are similar to those used in the proof of Theorem \ref{thm:general_exit} and the structure of the proof is the same. \\
\emph{Step 1. The Poisson-Brownian model. } Let us define the following expression depending on the Brownian paths $(B_t)_{t\ge 0}$, on the first exit time of the interval $[a,b]$ denoted by $\tau_B$ and on the independent Poisson process $N$: 
\begin{align*}
I(x,f,g,\kappa)&:=\mathbb{E}_x\Big[f(\tau_{B}\wedge \kappa)g(B_{\tau_{B}\wedge \kappa}) e^{-\int_0^{\tau_{B}\wedge \kappa}\gamma(B_s)\,ds}\Big]\\
&=\mathbb{E}_x\Big[f(\tau_{B}\wedge \kappa)g(B_{\tau_{B}\wedge \kappa})1_{\{ N(A)=0 \}}\Big],
\end{align*}
with
 \[
 A:=\Big\{ (t,y)\in [0,\tau_B\wedge \kappa]\times [0,\gamma_{+}]:\ y\le \gamma(B_t) \Big\}.
 \]
 Let us introduce the following series expansion:
 \[
 I(x,f,g,\kappa)=\sum_{n\ge 1} I_n(x,f,g,\kappa),
 \]
 the definition of $I_n(x,f,g,\kappa)$ is similar to the definition appearing in the proof of Theorem \ref{thm:general_exit}, it suffices to replace in the definition $\tau_B$ by $\tau_B\wedge \kappa$. Therefore 
 \[
 I_0(x,f,g,\kappa)=\mathbb{E}_x[f(\tau_B\wedge\kappa)g(B_{\tau_B\wedge \kappa})1_{\{ \tau_B\wedge\kappa\le \xi_1 \}}],
 \]
 where $(\xi_1,U_1)$ stands for the coordinates of the Poisson process point with the smallest abscissa. Moreover we obtain the following step by step property:
 \begin{equation}\label{eq:recu1}
 I_n(x,f,g,\kappa)=\mathbb{E}_x\Big[I_{n-1}(B_{\xi_1},f(\xi_1+\cdot),g,\kappa-\xi_1)1_{\{U_1>\gamma(B_{\xi_1}),\xi_1<\tau_B\wedge \kappa\}}\Big].
 \end{equation}
\emph{Step 2. Relation between Algorithm ($\kappa$-DET) and the Brownian-Poisson theoretical model presented in Step 1.}\\ If we denote by $\mathcal{N}_0$ the number of \emph{Step 0} used during the procedure which leads to the computation of the outcome $(Z,T)$ and $\mathcal{N}_1$ the number of exponentially distributed random variables generated (\emph{Step 1}), then we obviously obtain for $\mathcal{I}_n(x,f,g,\kappa):=\mathbb{E}_x[f(T)g(Z)1_{\{\mathcal{N}_0=1,\ \mathcal{N}_1=n+1 \}}]$,
\begin{align*}
\mathcal{I}_0(x,f,g,\kappa)&=\mathbb{E}\Big[f(S)g(Y)1_{\{ S<E\wedge K,\, U\le \beta_m(Y),\, W\le e^{-\rho(\kappa-S)} \}}\Big]\\
&+\mathbb{E}\Big[ f(\kappa)g(Y_c)1_{\{ \kappa\le E\wedge S,\, U\le \beta_m(Y_c) \}} \Big]\\
&=\mathbb{E}\Big[ f(\tau_B\wedge \kappa)e^{-\rho(\kappa-\tau_B\wedge\kappa)}g(B_{\tau_B\wedge\kappa})\beta_m(B_{\tau_B\wedge\kappa}) 1_{\{\xi_1\ge \tau_B\wedge \kappa\}} \Big]\\
&= I_0(x,f\times e^{-\rho(\kappa-\cdot)},g\times\beta_m,\kappa).
\end{align*}
Using similar arguments, we can prove that $\mathcal{I}_n$ satisfies the same step by step relation as\eqref{eq:recu1}. So by identification, we get $\mathcal{I}_n(x,f,g,\kappa)=I_n(x,f\times e^{-\rho(\kappa-\cdot)}, g\times\beta_m,\kappa)$ for all $n\ge 0$ and therefore 
\begin{align*}
\mathbb{E}[f(T)g(Z)1_{\{ \mathcal{N}_0=1 \}}]&=\mathcal{I}(x,f,g,\kappa)=\sum_{n\ge 0}\mathcal{I}_n(x,f,g,\kappa)\\
&= I(x,f\times e^{-\rho(\kappa-\cdot)}, g\times\beta_m,\kappa).
\end{align*}
Since Algorithm ($\kappa$-DET) is an acceptance rejection algorithm, we have
\begin{align*}
\mathbb{E}[f(T)g(Z)]&=\frac{\mathbb{E}[f(T)g(Z)1_{\{ \mathcal{N}_0=1 \}}]}{\mathbb{P}(\mathcal{N}_0=1)}=\frac{ I(x,f\times e^{-\rho(\kappa-\cdot)}, g\times\beta_m,\kappa)}{ I(x, e^{-\rho(\kappa-\cdot)}, \beta_m,\kappa)}.
\end{align*}
The link between $\beta_m$ and $\beta$ leads to
\begin{align*}
\mathbb{E}[f(T)g(Z)]&= \frac{ I(x,f\times e^{-\rho(\kappa-\cdot)}, g\times\beta,\kappa)}{ I(x, e^{-\rho(\kappa-\cdot)}, \beta,\kappa)} \\
&=\frac{\mathbb{E}_x\Big[f(\tau_B\wedge\kappa)g(B_{\tau_B\wedge\kappa})M_{\tau_B\wedge\kappa}\Big]}{\mathbb{E}_x[M_{\tau_B\wedge\kappa}]},
\end{align*}
where $M_t$ is the exponential martingale defined in \eqref{def:martingale}. It is actually important to note that since the time interval is bounded (upper-bounded by $\kappa$) we can use the time reversal expression $e^{-\rho(\kappa-\cdot)}$ and the ratio in the previous equation verifies a cancellation of the terms $e^{-\rho\kappa}$.\\ 
Finally it suffices therefore to use the Girsanov transformation in order to obtain the announced result:
\(
\mathbb{E}[f(T)g(Z)]=\mathbb{E}_x[f(\tau_X\wedge\kappa)g(X_{\tau_X\wedge\kappa})].
\)
\end{proof}
Let us now focus our attention on the exact simulation of $(X_{\tau_{a,b}},\tau_{a,b})$ for any diffusion process with regular drift term and constant diffusion coefficient. We have already seen that Algorithm (DET) permits to reach such objective however it is restricted to drift terms satisfying $\mu'+\mu^2\ge 0$ on the interval $[a,b]$. If such a condition is not verified, we propose the following procedure: first we choose some time parameter $\kappa>0$, then we apply:
 \begin{framed}
\centerline{GENERAL DIFFUSION EXIT TIME (GDET)} 
\centerline{parameters: $\kappa$, $a$ and $b$.}
\vspace*{0.2cm}

\noindent
{\bf Initialization.} \emph{$Z=x$, $T=0$, $\mathcal{N}_{\rm it}=0$  \\[5pt]
{\bf While} $Z\notin\{a,b\}$ {\bf do:}
\begin{itemize} 
\item Simulate $(S,Y)=\kappa\mbox{{\scriptsize-DET}}(Z,\kappa)$
\item $T\leftarrow T+S$, $Z\leftarrow Y$ and $\mathcal{N}_{\rm it}\leftarrow\mathcal{N}_{\rm it}+1$.
\end{itemize}
{\bf End While}\\[5pt]
{\bf Outcome:} the couple of random variables $(Z,T)$ and the number of iterations $\mathcal{N}_{\rm it}$.}
\end{framed}
Due to the Markov property, it is obvious that the outcome of such an algorithm and $(X_{\tau_{a,b}(X)},\tau_{a,b}(X))$ are identically distributed. Moreover the number of iterations of GDET has the same distribution as $\lfloor \frac{\tau_{a,b}(X)}{\kappa}\rfloor+1$. Hence
\[
\mathbb{E}_x[\mathcal{N}_{\rm it}]\le 1+\frac{\mathbb{E}_x[\tau_{a,b}(X)]}{\kappa}.
\]
It is evident that the number of iterations decreases as $\kappa$ becomes large but we should be careful for a clever choice of $\kappa$ since the number of rejections in Algorithm ($\kappa$-DET) grows exponentially fast when $\kappa$ enlarges, see Proposition \ref{prop:efficiency-exit-gen}. A reasonable choice is therefore $\kappa\approx\rho^{-1}$. 
%
%
%
%
%
%
%
\subsection{Examples and numerics}
\label{subsec:DFET-num}
The aim of this section is to emphasize the efficiency of Algorithm (DET) and Algorithm ($\kappa$-DET) through the analysis of two examples. The first situation concerns a diffusion whose drift term $b$ satisfies the condition $\mu'+\mu^2\ge 0$ and consequently only requires the basic Algorithm (DET). The second situation concerns the Ornstein-Uhlenbeck process which plays an essential role in several applications namely in neuroscience. For both examples, we set the parameters appearing in the algorithms: $\tcond=0.7$ and $\texit=0.5$.
\subsubsection{Example with the Algorithm (DET)}
We first consider a stochastic differential equation which was already presented in \cite{Herrmann-Zucca} for the simulation of the first passage time. Here the objective is clearly different since we focus our attention to the exit time and exit position of the diffusion and the algorithm is different too. Let us also note that a similar diffusion process was also introduced in \cite{beskos2005exact}.

We consider the following stochastic differential equation: 
\begin{equation}\label{eq:ex1}
dX_t=(2+\sin(X_t))\,dt+dB_t,\quad t\ge 0, \quad X_0=0.
\end{equation}
We first observe that $\gamma(x)=(\mu^2(x)+\mu'(x))/2=((2+\sin(x))^2+\cos(x))/2$ satisfies $0\le \gamma\le 5$. We deduce that we can apply Theorem \ref{thm:general_exit} with the particular choice $\rho=0$: the outcome $(Z,T)$ of Algorithm (DET)  has therefore the same distribution as $(X_{\tau_{a,b}(X)},\tau_{a,b}(X))$. The algorithm permits to obtain the histograms of respectively the exit time (Fig. \ref{fig:ex-DET} left)
and the counter $\mathcal{N}_{\rm tot}$ illustrating the efficiency of the algorithm DET (right). These histograms use a sample of size $100\,000$ and concerns the interval $[a,b]=[-0.5,0.5]$. The average value of the counter is $8.5$ and its estimated standard deviation is $8.88$.  Let us just compare the approximated results obtained by Algorithm (DET) for the interval $[a,b]=[-0.5,0.5]$ with a classical Euler method with step size $0.0001$ and $100\,000$ samples.\\[8pt]
{\renewcommand{\arraystretch}{1.5}
\centerline{\begin{tabular}{|c||c|c|c|c|c|}
\hline
Algo. & $\mathbb{E}[\tau_{a,b}]$ & $\sigma(\tau_{a,b})$ &  $\mathbb{E}[\tau_{a,b}1_{\{X_{\tau_{a,b}}=a\}}]$ & $\mathbb{P}(X_{\tau_{a,b}}=a)$\\
\hline
\hline
Euler method & 0.18262 & 0.13796 & 0.18446 & 0.12530 \\
\hline
(DET) & 0.17927 & 0.13667 & 0.18018 &  0.12685
\\
\hline
\end{tabular}}}

\vspace*{0.2cm}
When considering larger intervals like for instance $[-1,2]$ (Fig. \ref{fig:ex-DET-1}), the counter becomes large (average: 1205) and the algorithm DET is rather time consuming (C++ progaming: CPU 1,77 sec for $100\,000$ samples of the exit from the interval $[-0.5,0.5]$ and CPU 230,8 sec for the interval $[-1,2]$). Figure \ref{fig:ex-DET-2} is an illustration of the high level of rejection for the algorithm (DET) as the interval size increases. 
%
%
\begin{figure}[!h]
\centerline{\includegraphics[width=6.5cm]{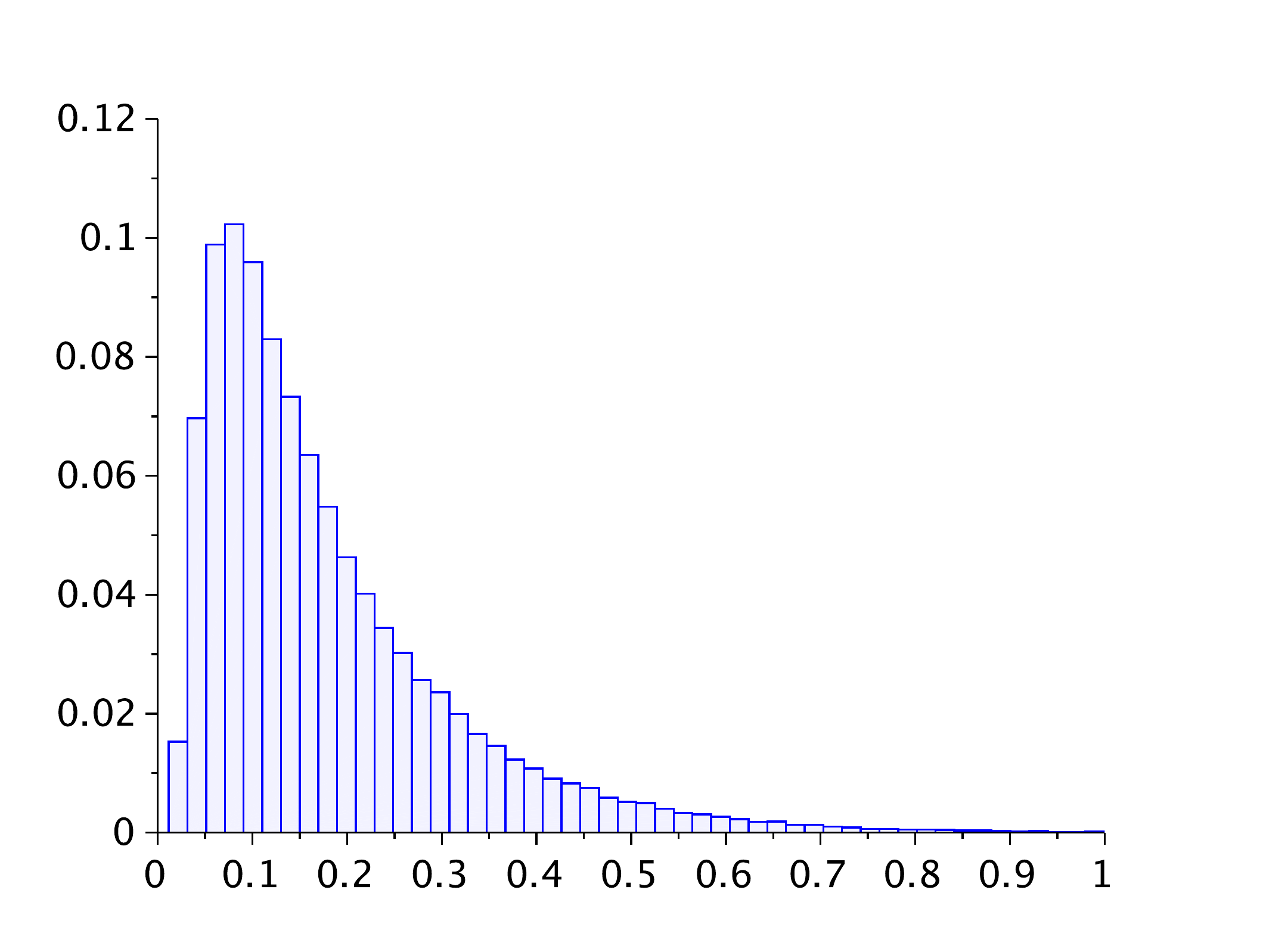}\hspace*{0.1cm}\includegraphics[width=6.5cm]{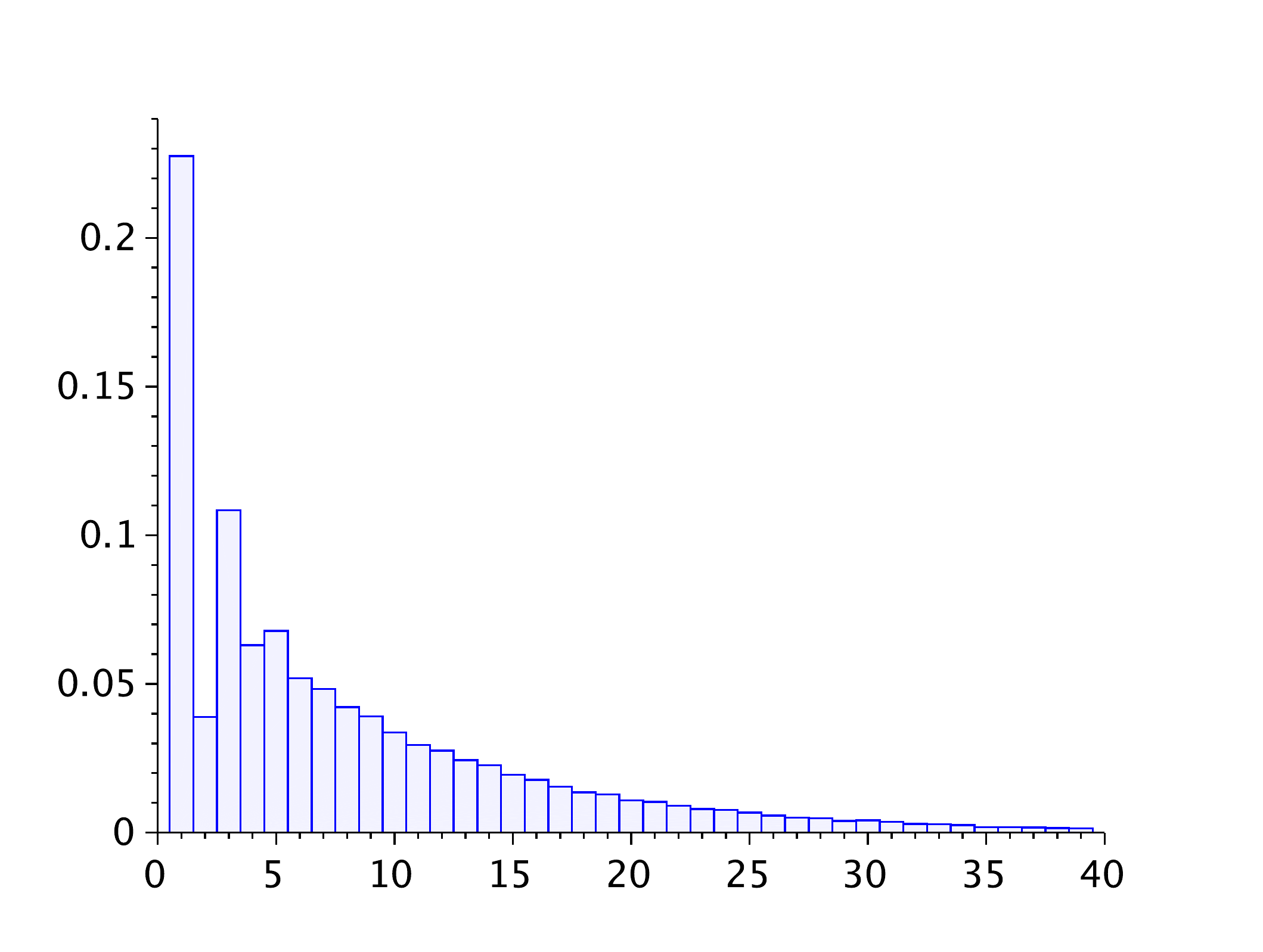}}
\caption{\small Simulation of the first exit time from the interval $[a,b]=[-1/2,1/2]$ for the diffusion \eqref{eq:ex1} ($100\,000$ samples). Histograms of the exit time variable (left) and of the counter $\mathcal{N}_{\rm tot}$ (right). }\label{fig:ex-DET}
\end{figure}
\begin{figure}[!h]
\centerline{\includegraphics[width=6cm]{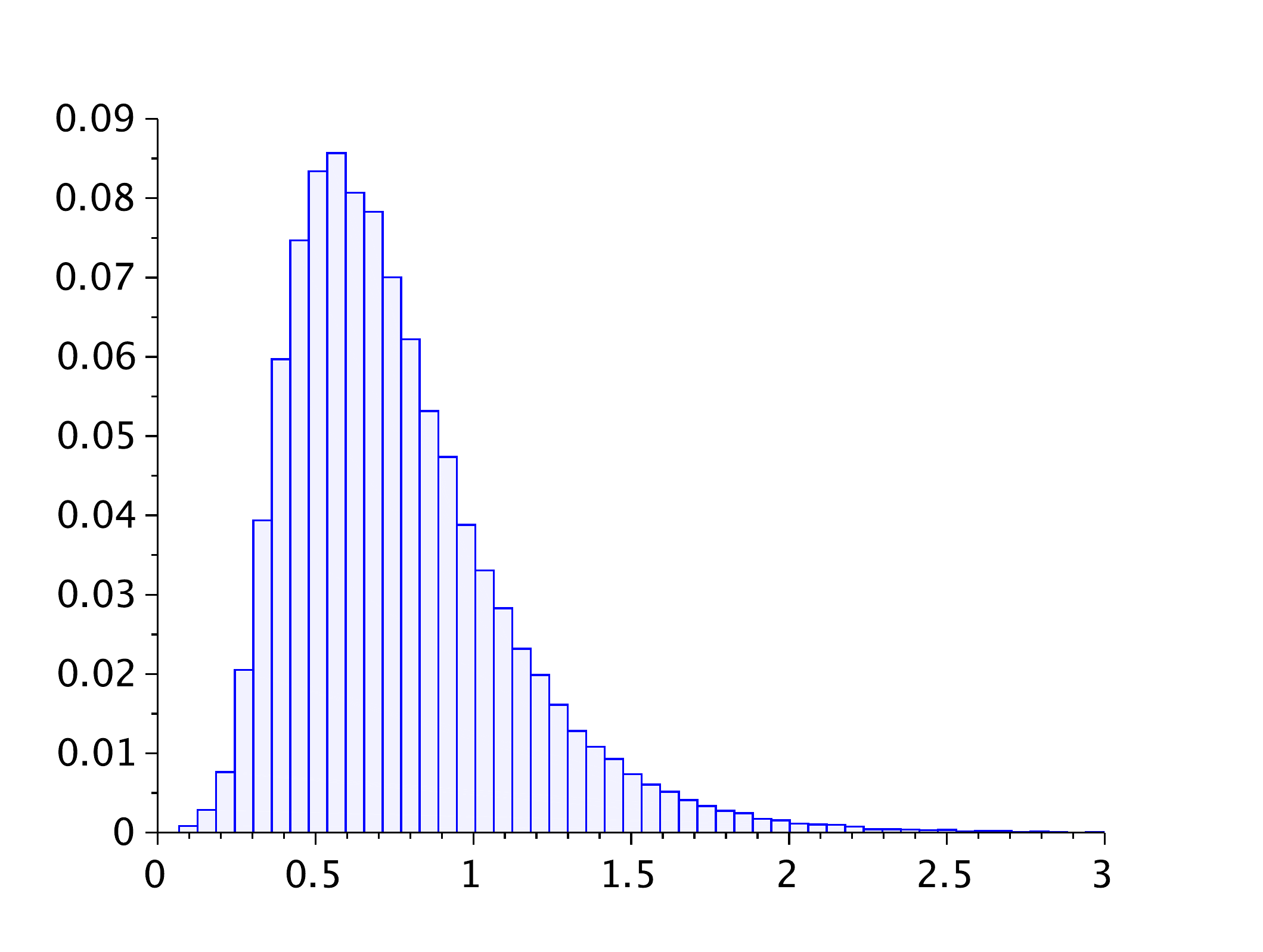}\hspace*{0.1cm}\includegraphics[width=6cm]{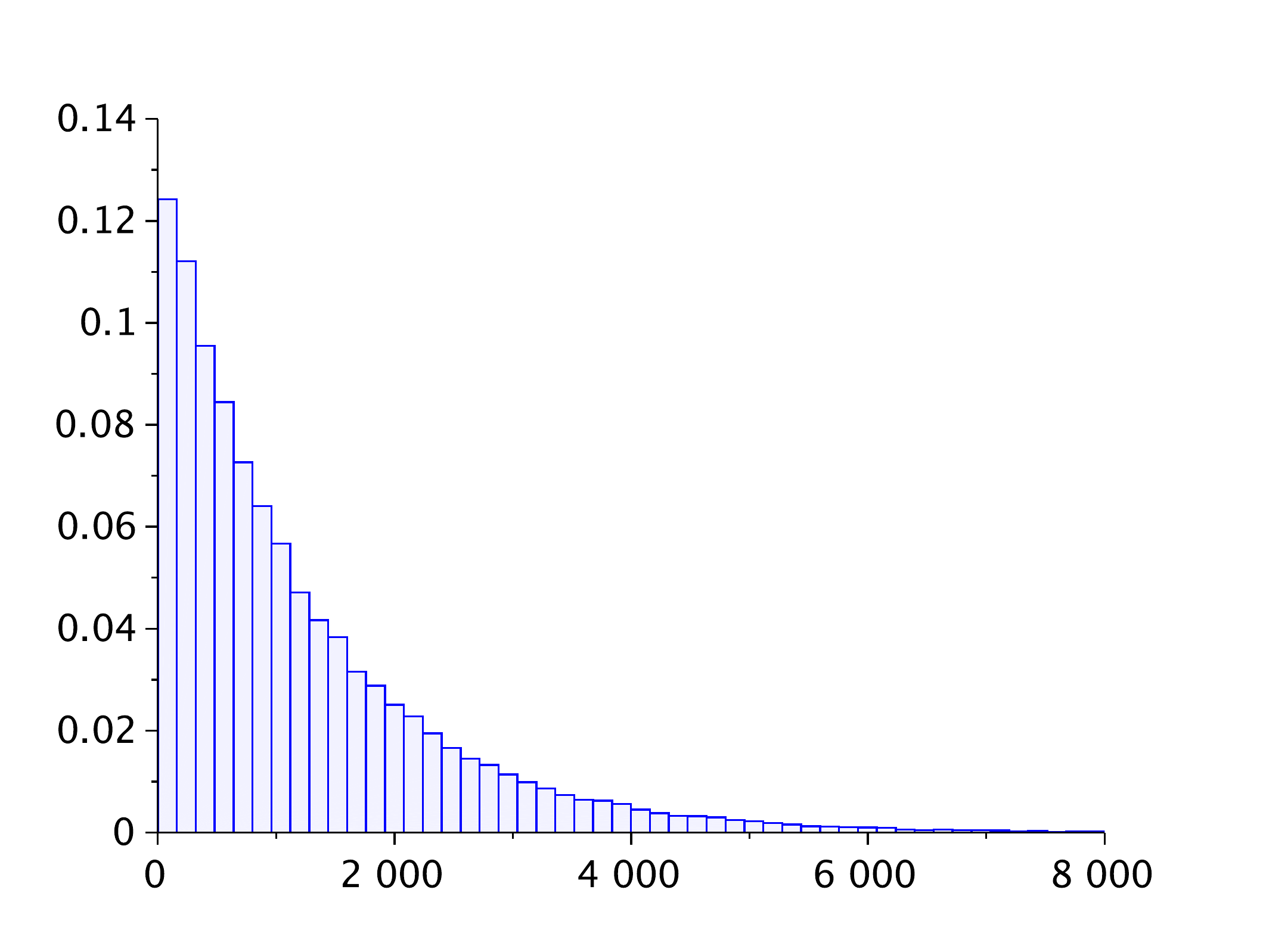}}
\caption{\small Simulation of the first exit time from the interval $[a,b]=[-1,2]$ for the diffusion \eqref{eq:ex1} ($100\,000$ samples). Histograms of the exit time variable (left) and of the counter $\mathcal{N}_{\rm tot}$ (right).}
\label{fig:ex-DET-1}
\end{figure}
\begin{figure}[!h]
\centerline{\includegraphics[width=6.5cm]{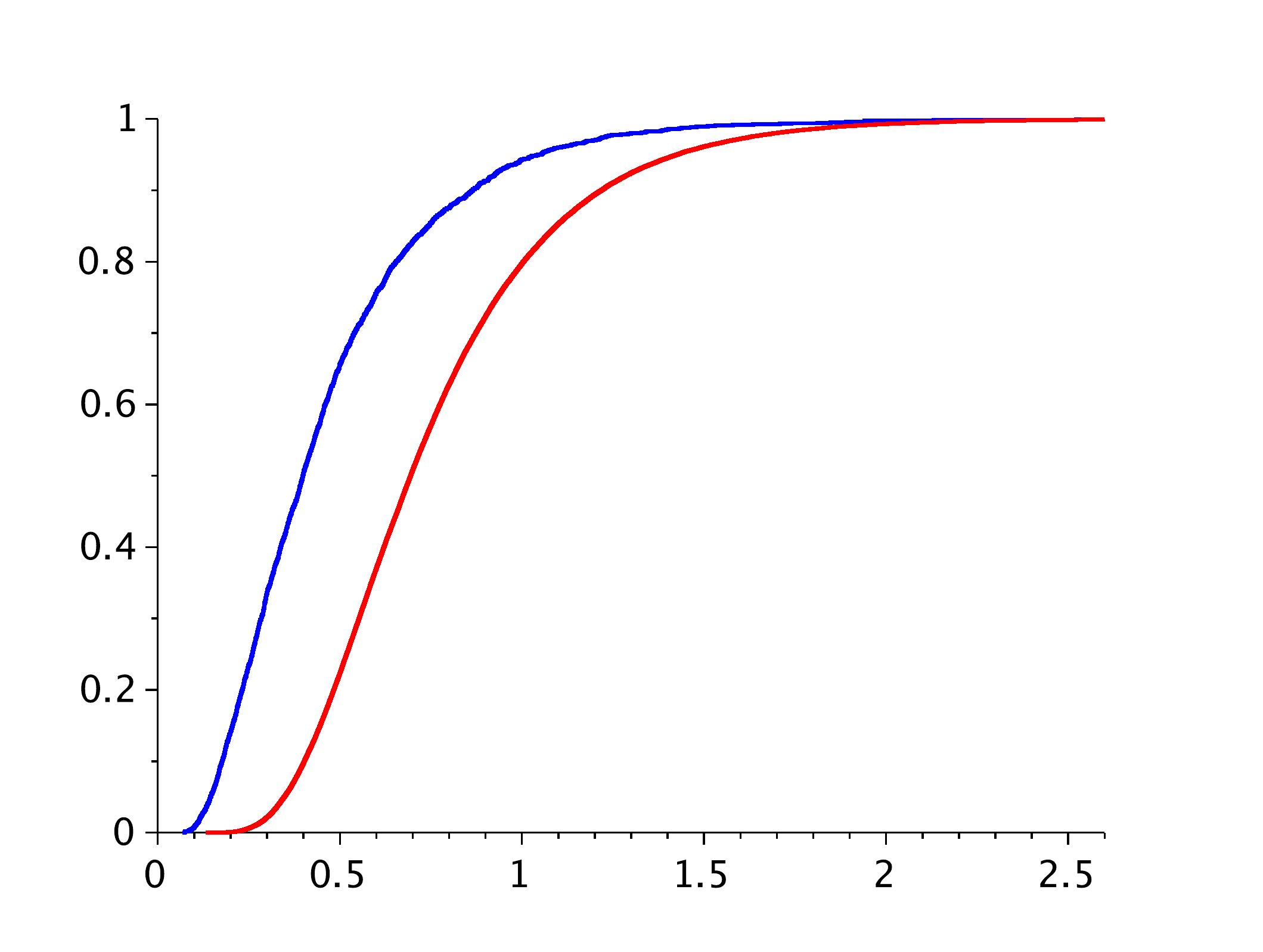}\hspace*{0.1cm}\includegraphics[width=6.5cm]{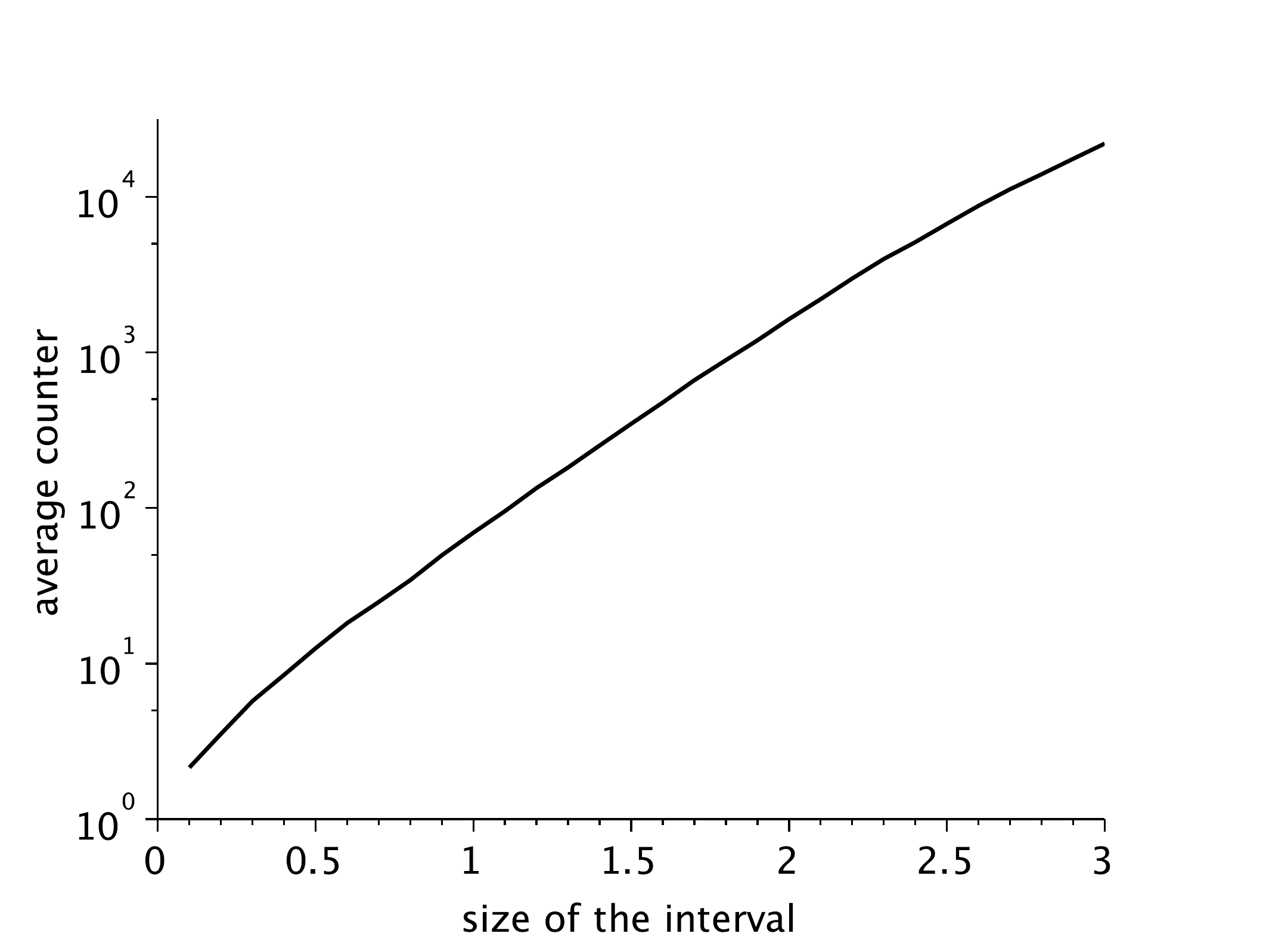}}
\caption{\small FET from the interval $[a,b]=[-1,2]$ for the diffusion \eqref{eq:ex1} ($100\,000$ samples). Empirical c.d.f. of the exit times when the exit occurs at the top -- lower curve -- or at the bottom of the interval -- upper curve -- (left). Simulation of the FET from the interval $[-a,a]$ for the diffusion \eqref{eq:ex1} ($10\,000$ samples). Average counter $\mathbb{E}[\mathcal{N}_{\rm tot}]$ in logarithmic scale versus $a$ (right).}
\label{fig:ex-DET-2}
\end{figure}

\subsubsection{The Ornstein-Uhlenbeck case.}
Let us now consider the Ornstein-Uhlenbeck process given by the following SDE
\begin{equation}
\label{eq:OU}
dX_t=-\mu_0 X_t\,dt+dB_t,\quad X_0=0.
\end{equation}
We shall determine the distribution of the exit time and location from the interval $[-a,a]$ with $a>0$. The drift term is defined by $\mu(x)=-\mu_0 x$,  we can therefore find $\rho$ such that \[\gamma(x):=\frac{\mu^2(x)+\mu'(x)}{2}+\rho=\frac{\mu_0^2x^2-\mu_0}{2}+\rho\ge 0.\] It suffices to choose $\rho=\frac{\mu_0}{2}$. Such a choice implies
\[
\sup_{x\in[-a,a]}\gamma(x)\le \frac{\mu_0^2a^2}{2}.
\]
Since the coefficient $\rho$ is strictly positive (for positive $\mu_0>0$), we cannot use the (DET)-algorithm in order to simulate the exit time. We shall therefore use the ($\kappa$-DET) algorithm in order to simulate exactly the couple $(X_{\tau_{-a,a}(X)\wedge \kappa},\tau_{-a,a}(X)\wedge \kappa)$ for any given constant time $\kappa>0$. Fig. \ref{fig:ex-KDET-1} represents the distribution of $X_{\tau_{-a,a}(X)\wedge \kappa}$ and Fig. \ref{fig:ex-KDET-2} illustrates the time distribution and describes the counter values $\mathcal{N}_{\rm tot}$. 
\begin{figure}[!h]
\centerline{\includegraphics[width=6.5cm]{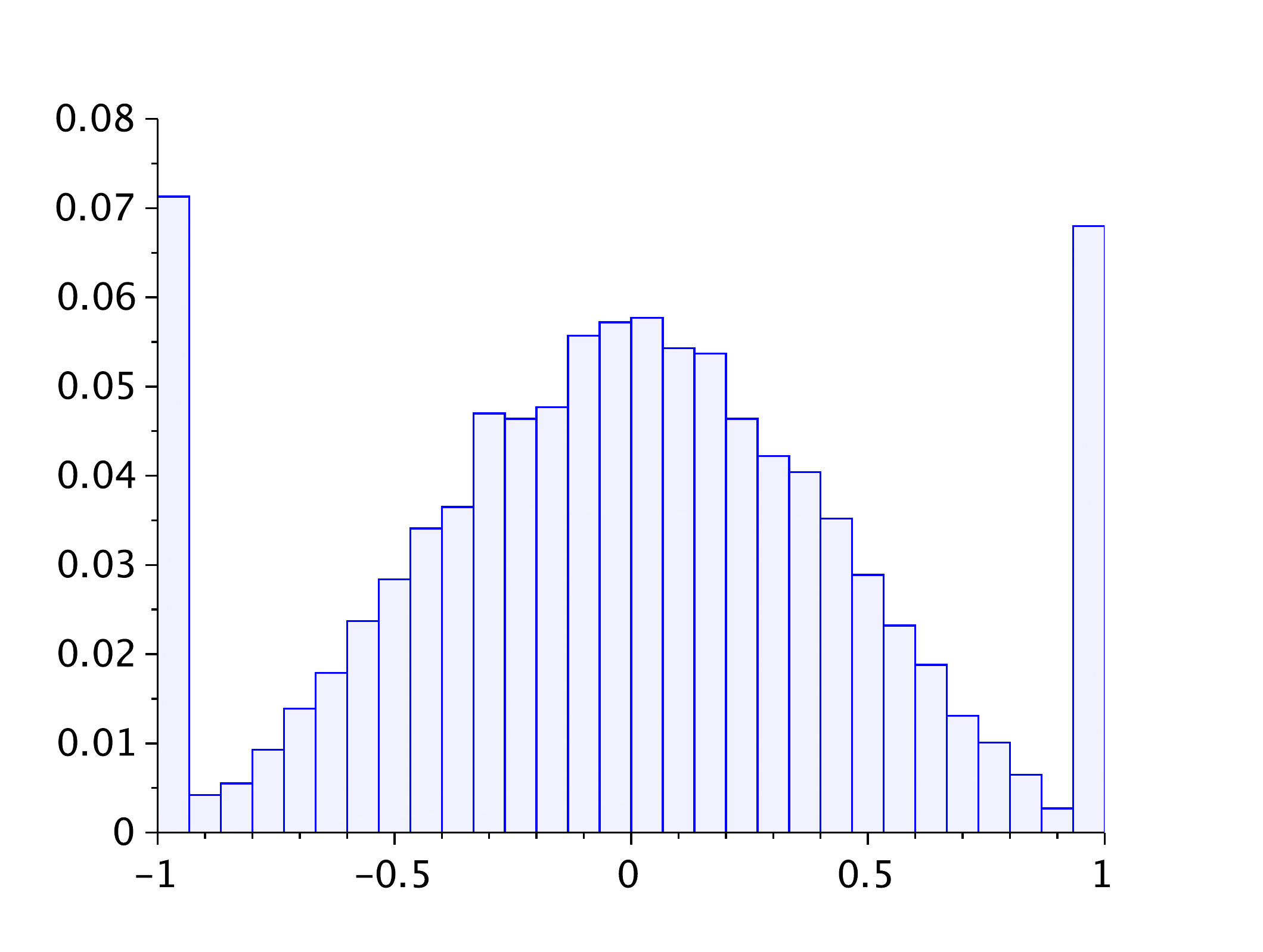}\hspace*{0.2cm}\includegraphics[width=6.5cm]{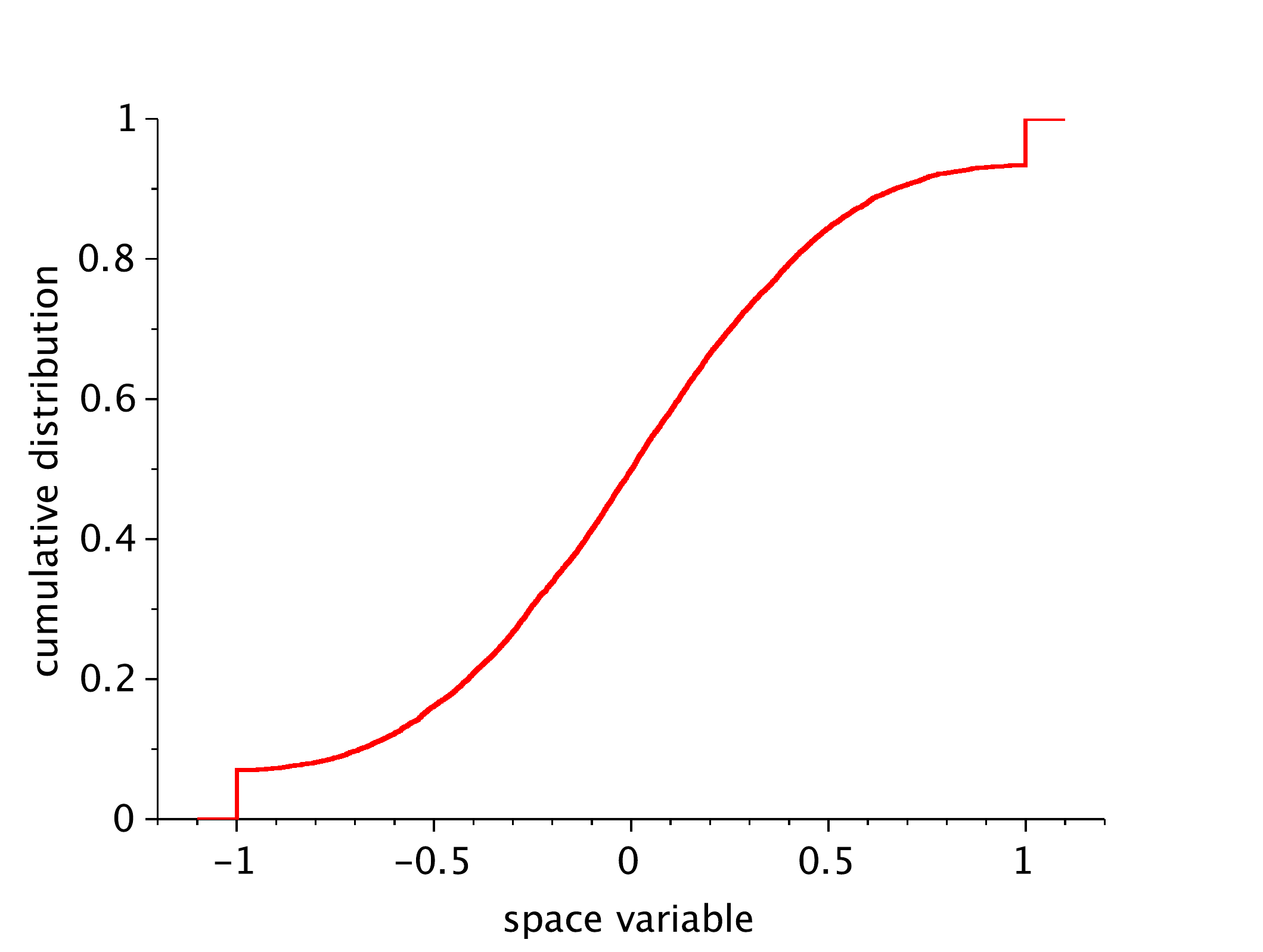}}
\caption{\small Distribution of $X_{\tau_{-a,a}(X)\wedge \kappa}$ with a sample of size $10\,000$, for the drift parameter $\mu_0=2$, the interval size $a=1$ and the constant time upper-bound $\kappa=0.5$: histogram (left) and cumulative distribution (right).}
\label{fig:ex-KDET-1}
\end{figure}
\begin{figure}[!h]
\centerline{\includegraphics[width=6.5cm]{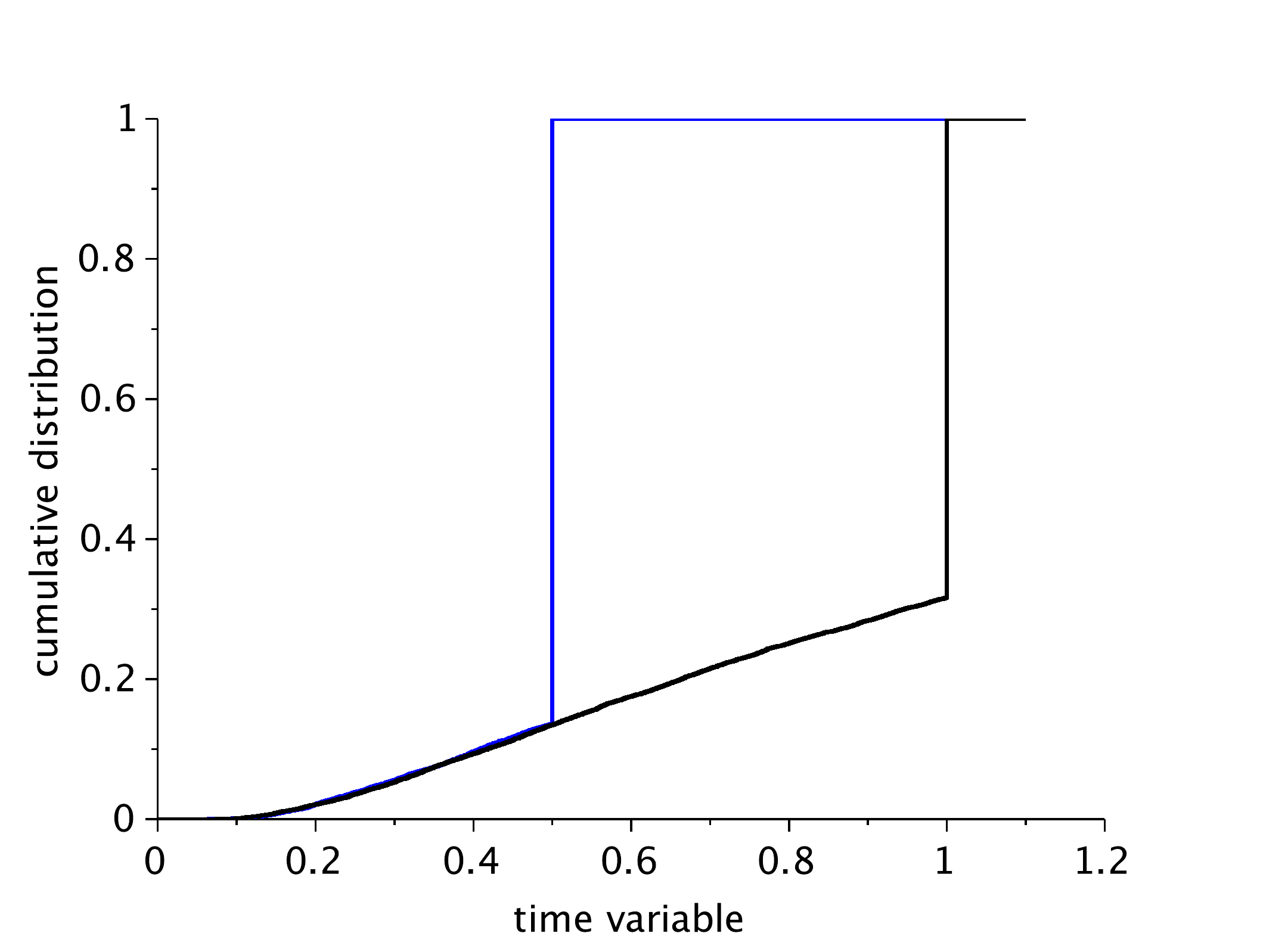}\hspace*{0.1cm}\includegraphics[width=6.5cm]{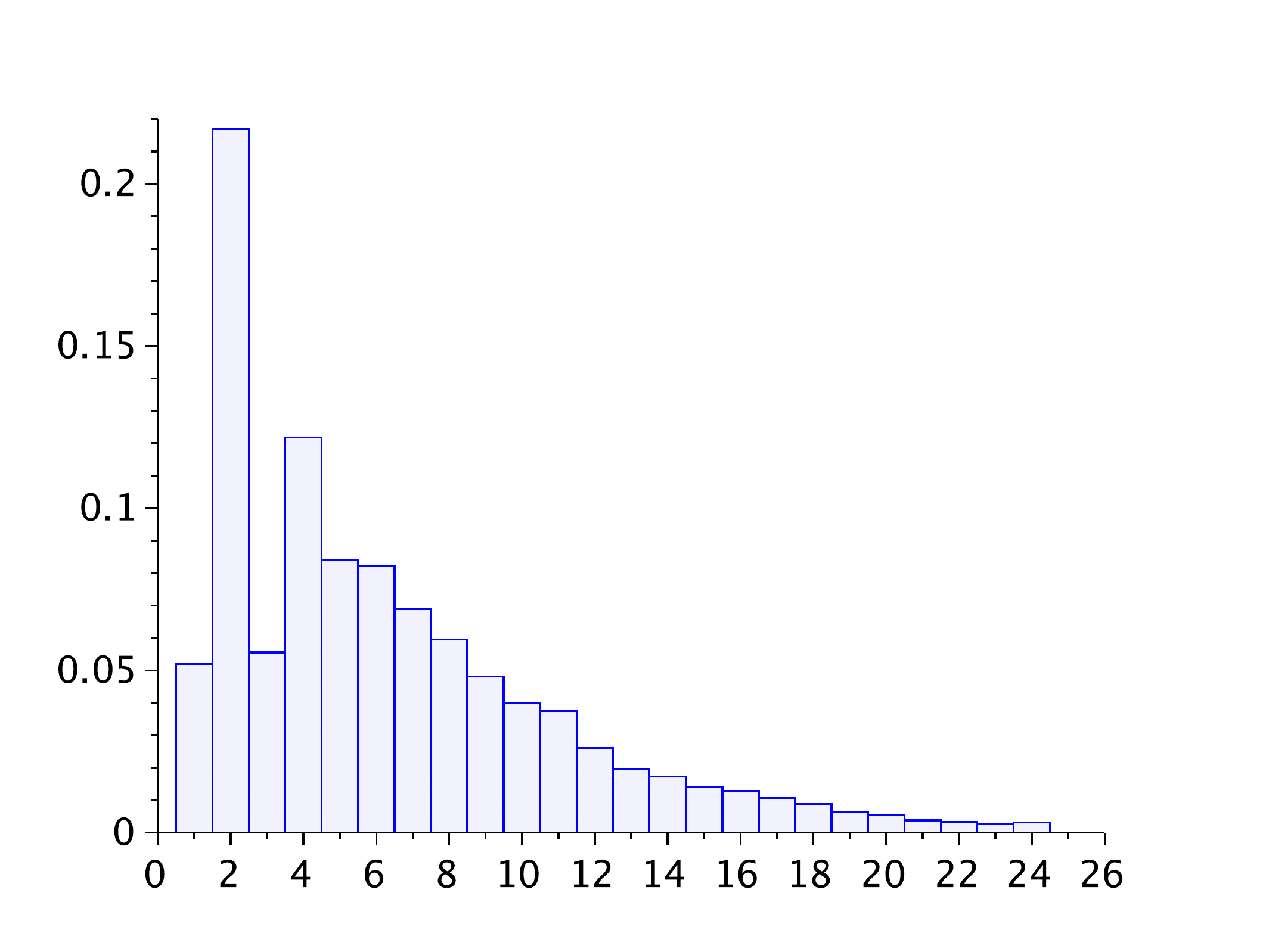}}
\caption{\small Cumulative distribution of $\tau_{-a,a}(X)\wedge \kappa$ with a sample of size $10\,000$, for the drift parameter $\mu_0=2$, the interval size $a=1$ and the constant time upper-bound $\kappa=0.5$ (left, blue curve) or $\kappa=1$ (left, black curve) and the corresponding histogram of the counter $\mathcal{N}_{\rm tot}$ for $\kappa=0.5$ (right).}
\label{fig:ex-KDET-2}
\end{figure}
In order to simulate exactly the exit time for the Ornstein-Uhlenbeck process, we use the (GDET)-algorithm. The histogram in Figure \ref{fig:ex-GDET-1} emphasizes the distribution of the exit time for the particular case: $\mu_0=2$ and $a=1$ and the efficiency of the algorithm for such a simulation. We can easily observe that the exit time increases as the parameter $\mu_0$ increases, see Fig. \ref{fig:ex-GDET-2}. 
\begin{figure}[!h]
\centerline{\includegraphics[width=6.5cm]{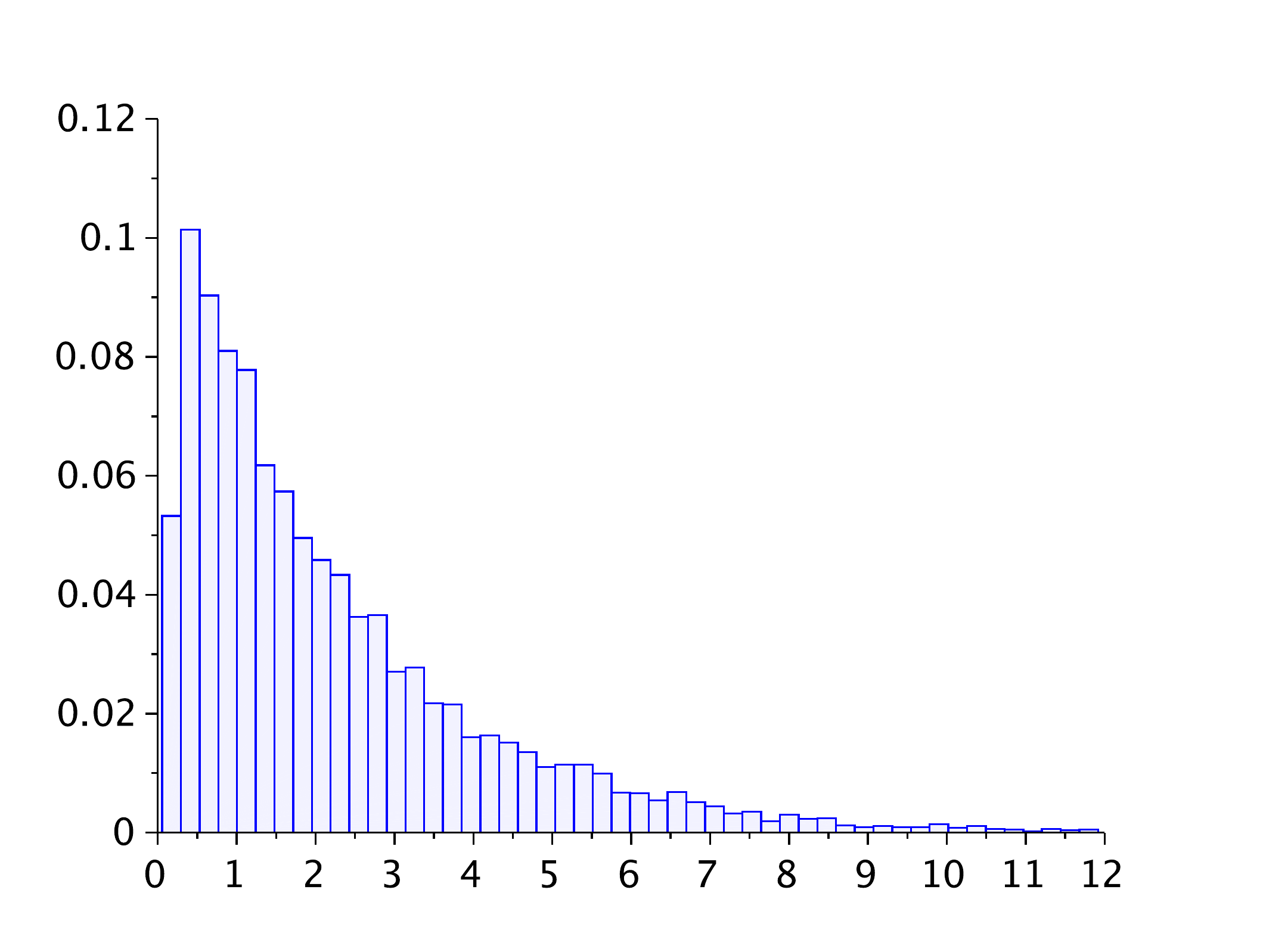}\hspace*{0.1cm}\includegraphics[width=6.5cm]{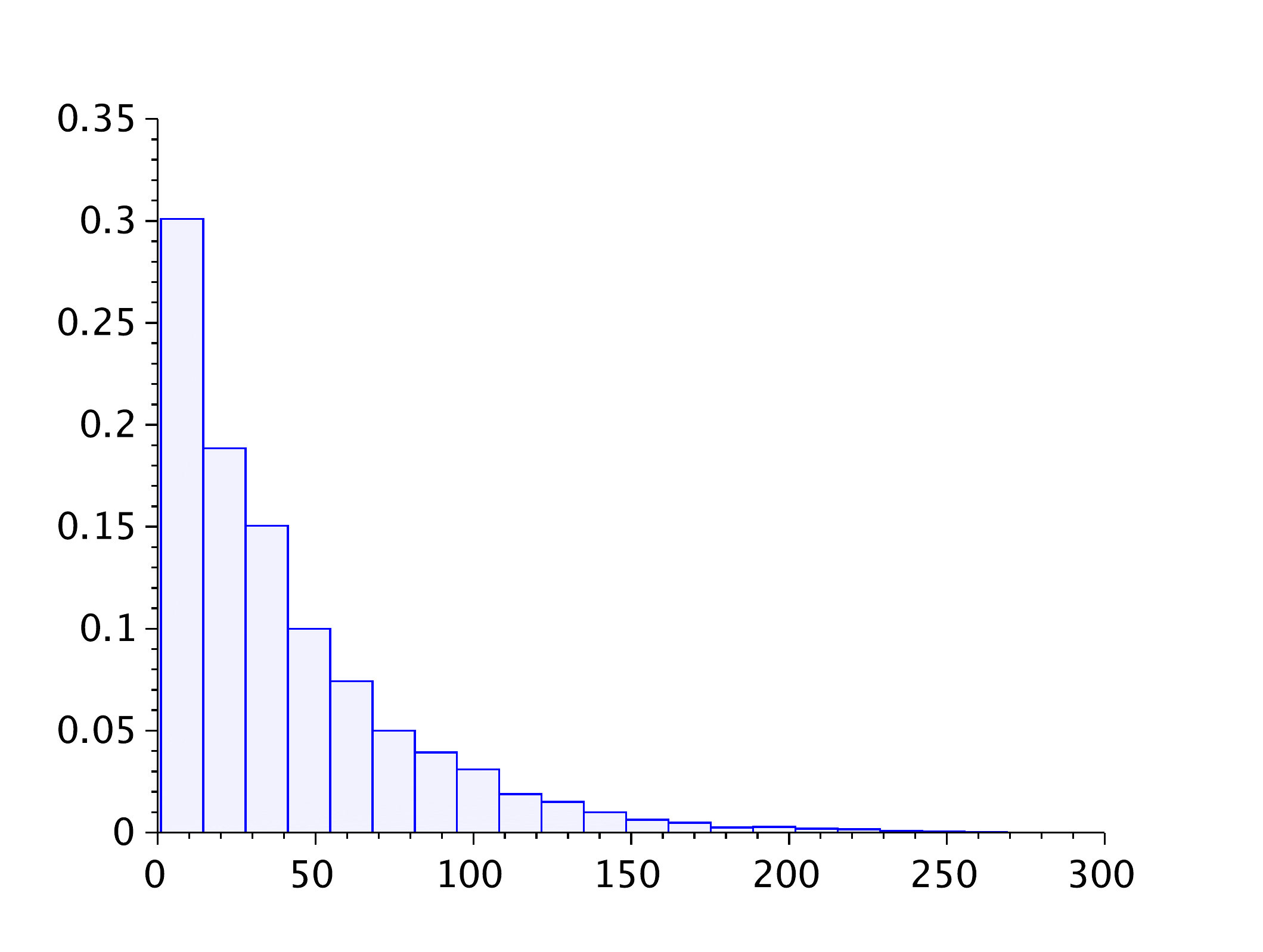}}
\caption{\small Histogram of the exit time distribution for the Ornstein-Uhlenbeck process with parameter $\mu_0=2$ and interval $[-1,1]$ (left), histogram of the associated counter $\mathcal{N}_{\rm tot}$ (right).}
\label{fig:ex-GDET-1}
\end{figure}
\begin{figure}[!h]
\centerline{\includegraphics[width=6.5cm]{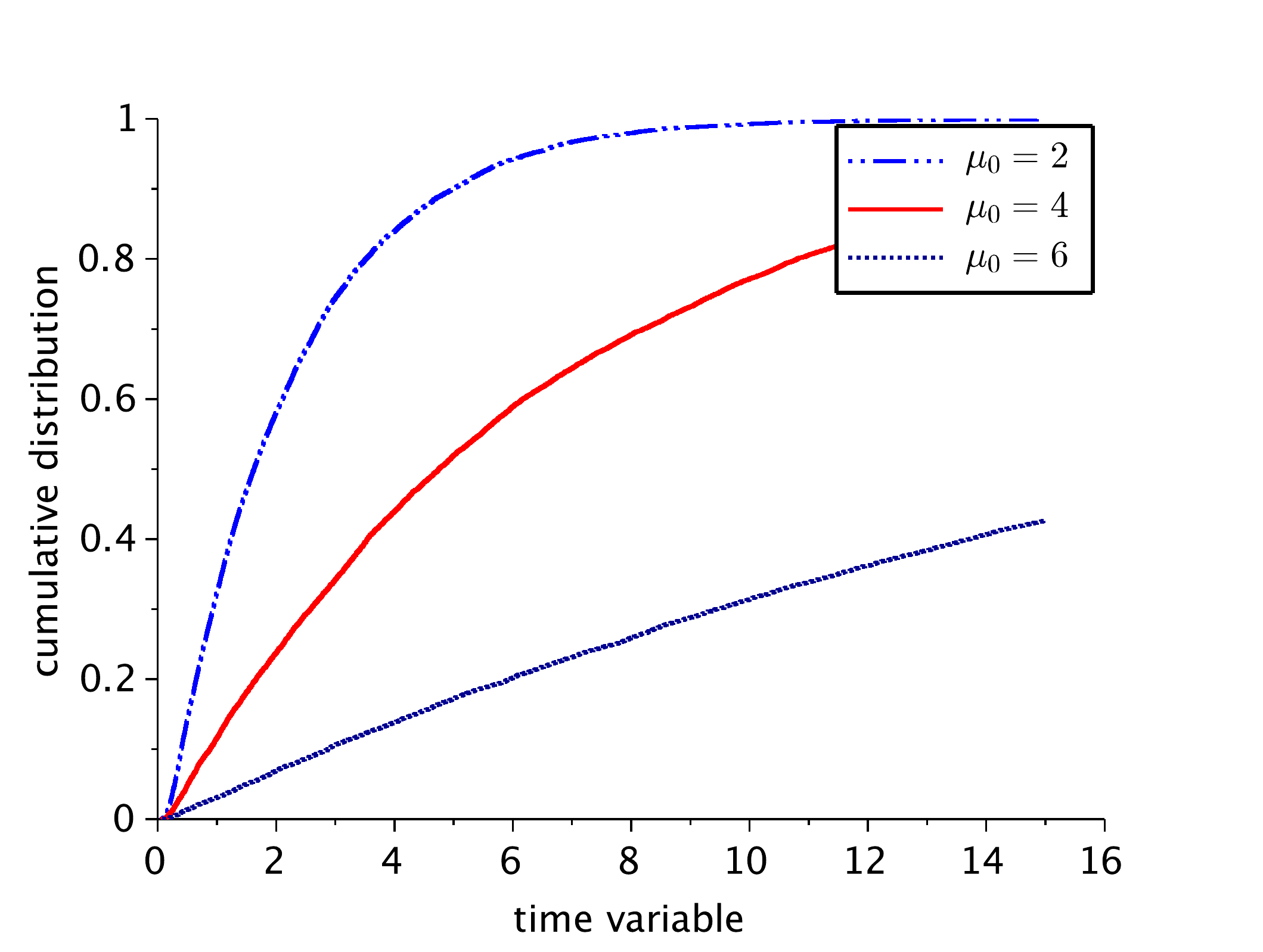}
}
\caption{\small Cumulative distribution on the time interval $[0,15]$ for the O.-U.- exit time from the interval $[-1,1]$}
\label{fig:ex-GDET-2}
\end{figure}
\begin{figure}[!h]
\centerline{\includegraphics[width=6.5cm]{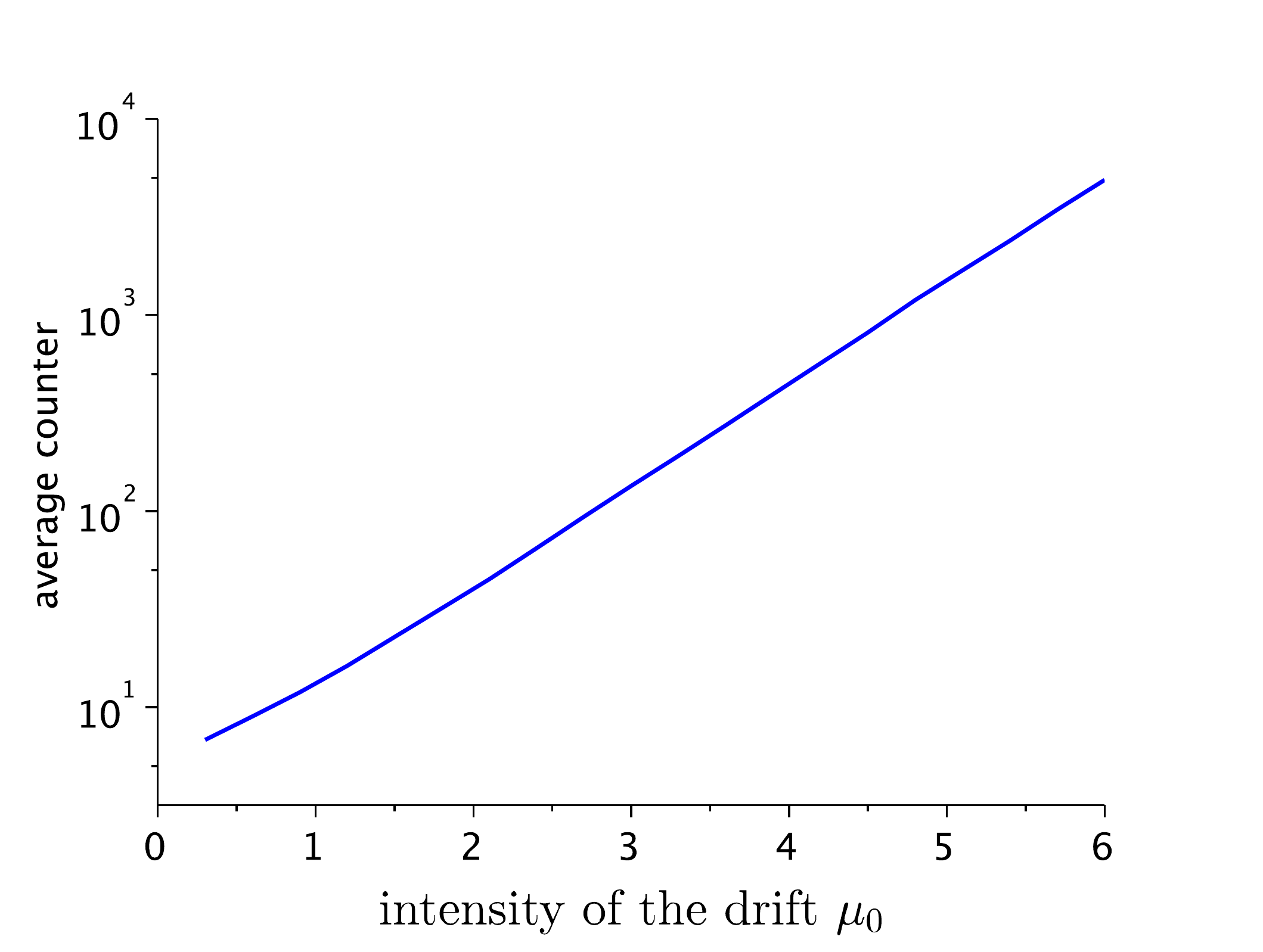}\hspace*{0.2cm}\includegraphics[width=6.5cm]{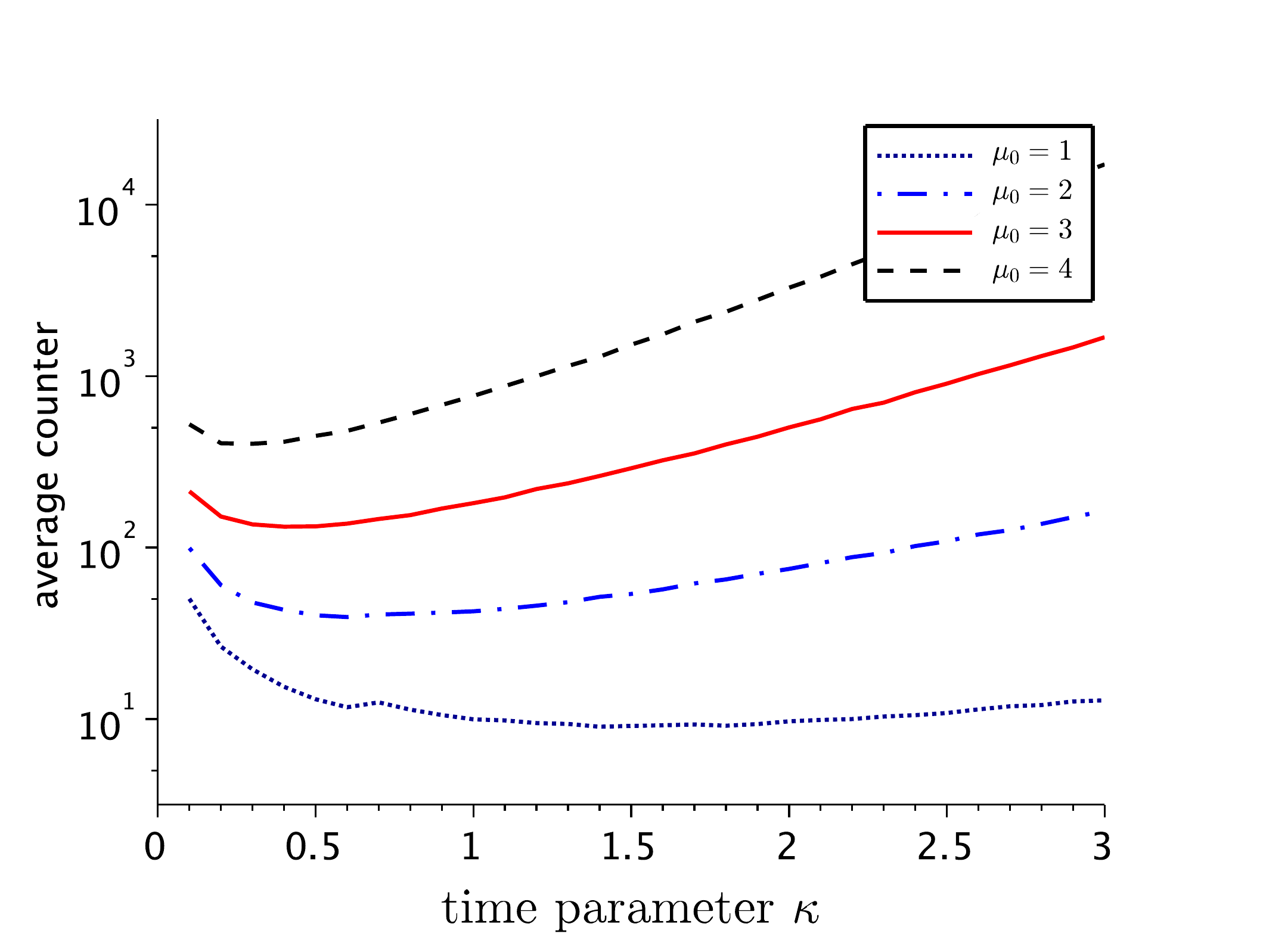}}
\caption{\small Average counter for the (GDET)-algorithm versus the intensity of the drift ($\mu_0$ varies between $0.3$ and $6$) in logarithmic scale (time constant: $\kappa=0.5$, sample size: $10\,000$) (left), average counter versus the time parameter $\kappa$ used in the (GDET)-algorithm (right).}
\label{fig:ex-GDET-3}
\end{figure}

\clearpage 
\noindent
{\bf Acknowledgment:} We thank Antoine Lejay for helpful discussions about the Brownian exit times and their simulation. We particularly appreciated his working paper \cite{lejay}.

%
%
%
%
%
%
%
\section*{Appendix}
\begin{proof}[Proof of Proposition \ref{prop:efficiency-exit-gen}]
The counter $\mathcal{N}_{\rm tot}$ depends on the parameter $\kappa$. So we set $\mathcal{N}_{\rm tot}=\mathcal{N}_{\kappa,{\rm tot}}$. This number can be decomposed as follows: 
\[
\mathcal{N}_{\kappa,{\rm tot}}^x=\sum_{k\ge 1}\mathcal{N}_{\kappa,{\rm tot}}^{x,k},
\] 
where $\mathcal{N}_{\kappa,{\rm tot}}^{x,k}$ represents the number of counter increases observed in-between the $k$-th and $(k+1)$-th passage through the item \emph{Step 0}. Since Algorithm ($\kappa$-DET) is an acceptance-rejection algorithm, the random variable $\mathcal{N}_0$ is geometrically distributed. Let us also note that $\mathcal{N}_{\kappa,{\rm tot}}^{x,k}=0$ a.s. on the event $\{\mathcal{N}_0<k\}$ and conditionally to $\{\mathcal{N}_0\ge k\} $, $\mathcal{N}_{\kappa,{\rm tot}}^{x,k}$ has the same distribution as $\mathcal{N}_{\kappa,{\rm tot}}^{x,1}$.
Hence
\begin{align}\label{eq:geom-bound-appen}
\mathbb{E}[\mathcal{N}_{\kappa,{\rm tot}}^x]&=\frac{\mathbb{E}[\mathcal{N}_{\kappa,{\rm tot}}^{x,1}]}{\mathbb{P}(\mathcal{N}_0=1)}=e^{\rho\kappa}\,\frac{\mathbb{E}[\mathcal{N}_{\kappa,{\rm tot}}^{x,1}]}{\beta_m(x)},
\end{align}
since $\mathbb{P}(\mathcal{N}_0=1)=I(x,e^{-\rho(\kappa-\cdot)},\beta_m,\kappa)=e^{-\rho\kappa}\beta_m(x)\mathbb{E}_x[M_{\tau_B\wedge \kappa}]=e^{-\rho\kappa}\beta_m(x)$, the martingale $(M_t)$ being defined in \eqref{def:martingale}. 

Let us now describe $\mathbb{E}[\mathcal{N}_{\kappa,{\rm tot}}^{x,1}]$.  Let $(S,Y,\mathcal{N}_{\rm as})$ stands for the result of the first use of the function  \mbox{\scriptsize BROWNIAN\_EXIT\_ASYMM\,} and $(Y_c^t,\mathcal{N}_c^t)$ of the first use of \mbox{\scriptsize CONDITIONAL\_DISTR\,} at time $t$, we can therefore distinguish three different cases. 
\begin{itemize}
\item if $S<E\wedge \kappa$ then $\mathcal{N}^{x,1}_{\kappa,{\rm tot}}=\mathcal{N}_{\rm as}$.
\item if $\kappa<E\wedge S$ then $\mathcal{N}^{x,1}_{\kappa, {\rm tot}}=\mathcal{N}_{\rm as}+\mathcal{N}_c^\kappa$.
\item if $E<S\wedge \kappa$ then $\mathcal{N}^{x,1}_{\kappa, {\rm tot}}=\mathcal{N}_{\rm as}+\mathcal{N}_c^E+\widehat{\mathcal{N}}^{Y_c^E,1}_{\kappa-E, {\rm tot}}1_{\{ \gamma_0V>\gamma(Y_c^E) \}}$ where $\widehat{\mathcal{N}}^{x,1}_{\kappa, {\rm tot}}$ is an independent copy of $\mathcal{N}^{x,1}_{\kappa, {\rm tot}}$. 
\end{itemize}
We deduce
\begin{align}\label{eq:opt-appen}
\mathbb{E}[ \mathcal{N}^{x,1}_{\kappa, {\rm tot}} ]&\le \mathbb{E}[ \mathcal{N}_{\rm as} ]+\mathbb{E}[ \mathcal{N}_{c}^\kappa1_{\{S\wedge E>\kappa\}} ]+\mathbb{E}[ \mathcal{N}_{c}^E1_{\{S\wedge \kappa>E\}} ]\nonumber\\
&+\mathbb{E}[\widehat{\mathcal{N}}^{Y_c^E,1}_{\kappa-E,{\rm tot}}1_{\{S\wedge \kappa>E\}}1_{\{ \gamma_0V>\gamma(Y_c^E) \}}]\nonumber\\
&\le \mathbb{E}[ \mathcal{N}_{\rm as} ]+\mathbb{E}[ \mathcal{N}_{c}^\kappa 1_{\{S>\kappa\}} ]\mathbb{P}(E>\kappa)+\mathbb{E}[ \mathcal{N}_{c}^E1_{\{S\wedge \kappa>E\}} ]\nonumber\\
&+\mathbb{E}[\widehat{\mathcal{N}}^{Y_c^E,1}_{\kappa-E,{\rm tot}}1_{\{S\wedge\kappa>E\}}]
\end{align}
Using the same arguments as those developed in the proof of Theorem \ref{thm:efficiency-exit}, we get the existence of two constants $C_0(\texit )>0$ and $C_1(\tcond )>0$ such that $\mathbb{E}[ \mathcal{N}_{\rm as} ]\le C_0(\texit )$ and $\mathbb{E}[\mathcal{N}_c^t1_{\{S>t\}}]\le C_1(\tcond )$ for all $t\ge 0$. We  then deduce
\begin{align*}
\mathbb{E}[ \mathcal{N}^{x,1}_{\kappa, {\rm tot}} ]&\le
C_0(\texit )+C_1(\tcond ) +\sup_{y\in[a,b]}\mathbb{E}[\widehat{\mathcal{N}}^{y,1}_{\kappa-E,{\rm tot}}]\mathbb{P}_{(a+b)/2}(S\wedge\kappa>E).
\end{align*}
This result can be generalized for any $K\le \kappa$ and therefore we obtain
\begin{align}\label{eq:append1}
\mathbb{P}_{(a+b)/2}(S\wedge\kappa\le E)\sup_{x\in[a,b],\ K\le \kappa}\mathbb{E}[ \mathcal{N}^{x,1}_{K, {\rm tot}} ]&\le
C_0(\texit )+C_1(\tcond ).
\end{align}
Let us note that $\mathbb{P}_{(a+b)/2}(S\wedge\kappa\le E)=\mathbb{P}_{(a+b)/2}(\tau_{a,b}(B)\wedge \kappa\le\xi)$ where $\tau_{a,b}(B)$ is the exit time of the Brownian motion from the interval $[a,b]$ and $\xi$ is exponentially distributed with parameter $\gamma_+$. Using the scaling property \eqref{eq:scaling}, we have
\begin{equation}\label{eq:append2}
\mathbb{P}_{(a+b)/2}(\tau_{a,b}(B)\wedge \kappa\le\xi)=\mathbb{E}_{(a+b)/2}[e^{-\gamma_+(\tau_{a,b}(B)\wedge \kappa)}].
\end{equation}
Combining  \eqref{eq:append1} \eqref{eq:append2} and \eqref{eq:geom-bound-appen}, we obtain
\begin{align*}
\mathbb{E}[\mathcal{N}_{\kappa,{\rm tot}}^x]&\le e^{\rho\kappa}\,\frac{\mathbb{E}[\mathcal{N}_{\kappa,{\rm tot}}^{x,1}]}{\beta_m(x)}\le e^{\rho\kappa}\,\frac{C_0(\texit)+C_1(\tcond)}{\beta_m(x)\ \mathbb{E}_{(a+b)/2}[e^{-\gamma_+(\tau_{a,b}(B)\wedge \kappa)}] }.
\end{align*}
\end{proof}

\end{document}